\documentclass{amsart}
  \usepackage[top=1in, bottom=1in, left=1.375in, right=1.375in]{geometry}
 \usepackage{amsmath,amssymb}
 \usepackage{algpseudocode}
 \usepackage{algorithm}
 \usepackage{algorithmicx}
 \usepackage{caption}
 \usepackage{subcaption}
 \usepackage{amsthm}
 \numberwithin{equation}{section}
 \usepackage{amsfonts}
 \usepackage{graphicx}
 \usepackage{hyperref}
 \usepackage{makecell}
 \usepackage{longtable}

\usepackage[shortlabels]{enumitem}
 \usepackage{amsmath}
\usepackage{amssymb}
\usepackage{amsthm}
\usepackage{amsmath}
\usepackage{setspace}
\usepackage{xcolor}
\usepackage{fancyhdr}
\usepackage{amssymb}
\usepackage{amsthm}
\usepackage{listings}
    \usepackage{cite}
    \usepackage{tabularx}
    \usepackage{graphicx}
    \usepackage{epstopdf}    
    \usepackage{epsfig}
    \usepackage{float}
    \usepackage{ listings}
    \usepackage{appendix}
 \theoremstyle{plain}

 \newtheorem{thm}{Theorem}[section]
 \newtheorem{prop}[thm]{Proposition}
 \newtheorem{lem}[thm]{Lemma}
 \newtheorem{cor}[thm]{Corollary}
 \theoremstyle{definition}

 \theoremstyle{remark}
 \newtheorem{remark}[thm]{Remark}

 \let\pa=\partial
 \let\al=\alpha
 \let\b=\beta
 \let\d=\delta
 \let\g=\gamma
 \let\e=\varepsilon

 \let \kp = \kappa
 \let\lam=\lambda

 \let\f=\frac
 \let \les = \lesssim
  \let \gtr = \gtrsim

 \let \th = \theta
  
 \let \pr = \prime

 \let\G= \Gamma
\let\B = \Big
 \let\D=\Delta

 \let\Om=\Omega
 \let\td = \tilde

 \let\teq \triangleq
 
 \let\pa=\partial
 \let \bsh = \backslash

 \def\cE{{\mathcal E}}

 \def\cM{{\mathcal M}}

 \def\cR{{\mathcal R}}

 \def\cU{{\mathcal U}}
 \def\cV{{\mathcal V}}

 \def\cY{{\mathcal Y}}
 \def\cZ{{\mathcal Z}}

 \def\cM{{\mathcal M}}

 \def\na{\nabla}

\def\one{\mathbf{1}}

 \newcommand{\bseq}{\begin{subequations}}
 \newcommand{\eseq}{\end{subequations}}

 \newcommand{\beq}{\begin{equation}}
 \newcommand{\eeq}{\end{equation}}
  \newcommand{\bal}{\begin{aligned} }
  \newcommand{\eal}{\end{aligned}}
    \newcommand{\bga}{ \begin{gathered} }
  \newcommand{\ega}{ \end{gathered} }
 \newcommand{\ben}{\begin{eqnarray}}
 \newcommand{\een}{\end{eqnarray}}
 \newcommand{\beno}{\begin{eqnarray*}}
 \newcommand{\eeno}{\end{eqnarray*}}

  \newcommand{\BZ}{\mathbb{Z}}

 \newcommand{\uu}{\mathbf{u}}

 \newcommand{\BB}{\mathbf{B}}

 \newcommand{\GG}{\mathbf{G}}
 \newcommand{\HH}{\mathbf{H}}

 \newcommand{\PP}{\mathbf{P}}
 
 \newcommand{\R}{\mathbb{R}}
\newcommand{\C}{\mathbb{C}}
 \newcommand{\N}{\mathbb{N}}

\newcommand{\sgn}{\mathrm{sgn}}

\newcommand \sn[1]{ { [ {#1}] } }
\newcommand \rn[1]{ { (#1) } }

\def\laml{\lambda_+}
\def\lams{\lambda_-}
\def\IC{\nu}
\def \ds{ \mathrm{ds}}

 \let \mfr = \mathfrak
 
 \let \mw = \mathrm

\author[T.~Buckmaster]{Tristan Buckmaster}
\address{Courant Institute of Mathematical Sciences, New York University, New York, NY 10012.}
\email{\href{buckmaster@cims.nyu.edu}{buckmaster@cims.nyu.edu}}

 \author[J.~Chen]{Jiajie Chen}
\address{Courant Institute of Mathematical Sciences, New York University, New York, NY 10012.}
\email{\href{jiajie.chen@cims.nyu.edu}{jiajie.chen@cims.nyu.edu}}

 \date{ \today}

\title[Blowup for defocusing nonlinear wave equation]{Blowup for the defocusing septic complex-valued nonlinear wave equation in $\mathbb{R}^{4+1}$}

\begin{document}

\begin{abstract}
In this paper, we prove blowup for the defocusing septic complex-valued nonlinear wave equation in $\mathbb{R}^{4+1}$. This work builds on the earlier results of Shao, Wei, and Zhang \cite{shao2024self,shao2024blow}, reducing the order of the nonlinearity from $29$ to $7$ in $\mathbb{R}^{4+1}$. As in \cite{shao2024self,shao2024blow}, the proof hinges on a connection between solutions to the nonlinear wave equation and the relativistic Euler equations via a front compression blowup mechanism. More specifically, the problem is reduced to constructing smooth, radially symmetric, self-similar imploding profiles for the relativistic Euler equations.

As with implosion for the compressible Euler equations, the relativistic analogue admits a countable family of smooth imploding profiles. The result in \cite{shao2024self} represents the construction of the first profile in this family. In this paper, we construct a sequence of solutions corresponding to the higher-order profiles in the family. This allows us to saturate the inequalities necessary to show blowup for the defocusing complex-valued nonlinear wave equation with an integer order of nonlinearity and radial symmetry via this mechanism.

\end{abstract}

 \maketitle

\section{Introduction}

We investigate finite time blowup for the complex-valued defocusing nonlinear wave equation:
\beq\label{eq:wave}
 \pa_{tt} w = \D w - |w|^{p-1} w,
\eeq
for $ w(t, x) : \R \times \R^d \to \C$.  Given smooth and well-localized initial data $(w(0), \pa_t w(0)) $, the classical Cauchy theory states that
\eqref{eq:wave} admits a smooth and strong local-in-time solution $w(t)$ \cite{sogge1995lectures}.  The parameters $(d,p)$, can be naturally classified in terms of the critically of the conserved energy
\[
E(w(t)) \teq  \int_{\R^d} \f{1}{2} (|\pa_t w |^2 + |\na w|^2  ) + \f{1}{p+1} |w|^{p+1} d x,
\]
in terms of the scaling symmetry
\[
w(t, x) \to w_{\lam}(t, x) \teq \lam^{ \f{2}{p-1} } w( \lam t , \lam x), \quad \lam > 0.
\]
Precisely, we obtain three regimes: the subcritical case,  $d \leq 2 $ or $ p < 1 + \f{4}{d-2}$ for $d \geq 3$; the critical case,  $ p =1 + \f{4}{d- 2}$ and $d \geq 3$; and the supercritical case,
$ p >1 + \f{4}{d- 2}$ and $d \geq 3$.

The blowup of nonlinear wave equations has been extensively studied 
in the focusing case (cf.\ \cite{Alinhac,lindblad1990blow,John, Levine,Donninger2017,Duyckaerts-Kenig-Merle2012,Duyckaerts-Kenig-Merle2013,Duyckaerts-Yang2018,Jendrej2017,Kenig2015,Kenig-Merle2008,Krieger-Schlag2014,Krieger-Schlag-Tataru2009}). 
For the defocusing case, the equation is globally well-posed in the subcritical and critical case \cite{ginibre1985global,ginibre1989global,grillakis1990regularity,Grillakis1992,shatah1993regularity,shatah1994well,Jorgens1961,struwe1988globally}. 
Nevertheless, the question of blowup or global well-posedness for smooth solutions to defocusing nonlinear Schr\"odinger or wave equation has remained a long-standing open problem \cite{tao2006nonlinear}. Tao, in \cite{Tao2016}, showed if one considers a supercritical, high dimensional, defocusing nonlinear wave system $\Box u= (\nabla_{\mathbb R^m} F)(u)$, for some smooth, positive, carefully designed potential $F:\mathbb R^m\rightarrow \mathbb R$, one proves blowup in finite time. 
Recently, in the breakthrough work of Merle, Raphael, Rodnianski and Szeftel \cite{merle2022blow}, 
blowup for the supercritical defocusing nonlinear Schr\"odinger equation was proven via a novel front compression mechanism, whereby the phase of the solution blows up. The 
work \cite{merle2002blow} leverages a connection between the nonlinear Schr\"odinger equations and the compressible Euler equations with an added \emph{quantum pressure} term via Madelung’s transformation. Under the appropriate scaling assumptions -- which necessitates considering dimensions $\geq 5$ and high order nonlinearities -- the quantum pressure term becomes lower order, which enabled the use of self-similar imploding profiles to the compressible Euler equations constructed in \cite{merle2022implosion1}, as asymptotic self-similar profiles to the defocusing nonlinear Schr\"odinger equation.

Inspired by the work \cite{merle2022blow},  Shao,  Wei, and Zhang in \cite{shao2024self,shao2024blow} recently exploited a similar front compression mechanism to establish blowup for the defocusing nonlinear wave equation 
\eqref{eq:wave} via a self-similar profile to the relativistic Euler equations. 
Specifically, they proved finite time blowup of \eqref{eq:wave} from smooth compactly supported 
initial data with dimension $d$ and odd nonlinearity $p$ satisfying $d = 4, p \geq 29$ and $d \geq 5, p \geq 17$. See also Theorem \ref{thm:SWZ}.

\subsection{Self-similar ansatz}

We consider the phase-amplitude representation $ w = \rho e^{i \phi}$ with $\rho, \phi$ being real and rewrite \eqref{eq:wave} as the system of $(\rho, \phi)$. To this end, we compute
\[
\bal
\pa_{tt} w & = \B( \pa_{tt} \rho + i \pa_t ( \pa_t \phi \rho) 
+ i \pa_t \phi ( \pa_t \rho + i\pa_t \phi \rho)  \B) e^{i \phi}, \\
\D w & = (\D \rho + i \na \cdot ( \rho \na \phi ) + i \na \rho \cdot \na \phi - \rho |\na \phi|^2) e^{i \phi}.
\eal
\]
Substituting the above computation into \eqref{eq:wave}, dividing by $e^{i\phi}$ and then taking the real and imaginary part of the equation, we derive the equations for $\rho$ and $\phi$
\bseq\label{eq:wave_decomp}
\begin{align}
\pa_{tt} \rho  & = \rho |\pa_t \phi|^2 + \D \rho - \rho |\na \phi|^2 - |\rho|^{p-1} \rho , \label{eq:wave_w} \\
\pa_{tt} \phi \cdot \rho  & =  - 2 \pa_t \phi \pa_t \rho + \rho \D \phi + 2 \na \rho \cdot \na \phi  .
\label{eq:wave_phase} 
\end{align}
\eseq

We consider the general self-similar blowup ansatz 
\beq\label{eq:SS_ansatz}
\rho = \f{1}{(T-t)^a} W( \f{x}{ (T-t)^c} ),
\quad \phi = \f{1}{ (T-t)^b} \Phi( \f{x}{ (T-t)^c} ), 
\eeq
for some $a,b,c$ to be determined. 

In line with front compression, where the phase $\phi$ blows up, we consider $ b > 0$ for our self-similar ansatz \eqref{eq:SS_ansatz}. Moreover, to obtain a collapsing blowup, which is natural due to the conservation of the energy, we impose $c > 0$. We use the notation $A \asymp B$ to denote that $A $ and $B$ are comparable up to a $t-$independent constant. Under these assumptions, for $T - t \ll 1$, we get 
\beq\label{eq:SS_ansatz_power}
\bga 
    \pa_{tt} \rho \asymp (T-t)^{-a-2}, \ \D \rho \asymp (T-t)^{-a-2c}, \\ 
     \rho |\pa_t \phi|^2 \asymp (T-t)^{-a-2-2b}, \quad  \rho |\na \phi|^2 \asymp (T-t)^{-a-2b-2c}, \quad |\rho|^{p-1} \rho \asymp (T-t)^{ -pa}.
   \ega 
\eeq
Since $b>0$, we treat  $\pa_{tt} \rho,  \D \rho$  in \eqref{eq:wave_w} as lower order terms. 
For the $\phi$-equation \eqref{eq:wave_phase}, it is easy to observe that the scalings of each terms are balanced. Thus, dropping $\pa_{tt} \rho, \D \rho $ in \eqref{eq:wave_w} and then dividing $\rho$ on both sides of \eqref{eq:wave_w},  we derive the following leading order system of \eqref{eq:wave_w}
\beq\label{eq:vel}
 |\pa_t \phi|^2 -  |\na \phi|^2 = |\rho|^{p-1} .
\eeq
To balance these three terms, using \eqref{eq:SS_ansatz_power}, we impose the following relation among $a,b, c$ in \eqref{eq:SS_ansatz}
\beq\label{eq:SS_ansatz2}
 c = 1, \quad (p-1) a = 2 (b+ 1), \quad b > 0.
\eeq

Note that equations \eqref{eq:vel} and \eqref{eq:wave_phase} reduce to isentropic relativistic Euler equation. In fact, by introducing $\ell = \f{4}{p-1} + 1$, the flat Lorentz metric 
$g = (g^{\mu \nu})$ and derivative $\pa^\mu$ as
\[
 g^{00} = -1, \quad g^{\mu \nu} = \d_{\mu \nu}, \quad  \forall (\mu, \nu) \neq (0, 0),
 \quad \pa^{\mu} = g^{\mu \nu} \pa_{\nu}, 
\]
the relativistic velocity $\uu = (u^0, u^1, .., u^d)$, the energy density $\varrho$, the energy momentum tensor $T^{\mu \nu}$, and the pressure $P$ subject to the isothermal equation of state as follows
\bseq\label{eq:rel_euler}
\beq\label{eq:rel_eulera}
\varrho =\rho^{p+1},
\quad u^0 = \rho^{-\f{p-1}{2} } \pa_t \phi, \quad u^i =\rho^{-\f{p-1}{2} }  \pa^i \phi, 
\quad T^{\mu \nu} = (P + \varrho) u^{\mu} u^{\nu} + P g^{\mu \nu},
\eeq
we derive the relativistic Euler equations 
\beq
u^{\mu} \pa_{\mu} \varrho +(P + \varrho) \pa_{\mu} u^{\mu} = 0,
\quad \pa_{\mu} T^{\mu \nu} = 0, \quad 
P = \f{1}{\ell} \varrho, 
\eeq
\eseq
with normalization $\| \uu \|_{g}^2 = g_{\mu \nu} u^{\mu} u^{\nu} = -1 $. The above connection between the defocusing wave equations and the relativistic Euler equations has been used in \cite{shao2024self}.

The local existence of smooth solutions to the relativistic Euler equations was established in \cite{MU}, with some global existence results in \cite{ST,Ruan_Zhu,Chen_Schrecker}.
Meanwhile, the equations can develop a finite time singularity \cite{guo1999formation,pan2006blowup,Ch2007}. 
For more discussions on the mathematical aspects of relativistic fluids, see the surveys \cite{disconzi2023recent,abbrescia2023relativistic}.

\subsection{Main results}
By imposing the radial symmetry on $\rho, \phi$ and plugging the self-similar ansatz \eqref{eq:SS_ansatz} in 
equations \eqref{eq:wave_phase} and \eqref{eq:vel}, we can derive the ODE \eqref{eq:ODE} for the self-similar profiles $W, \Phi$ \eqref{eq:SS_ansatz}. See the derivation in Section \ref{sec:ODE_sub}. 

To use the front compression mechanism for smooth blowup, we require $(d, p)$ to satisfy
\beq\label{eq:range_dp}
  \ell + \ell^{1/2} < d - 1 , \quad \ell = 1 + \f{4}{p-1} > 1,
\eeq
which implies $d  \geq 4$ and the non-linear exponent $ p \geq p(d)$. 
For $d = 4$ and integer $p$, this implies $p \geq 7$. See Section \ref{sec:ODE_sub} for more details about these constraints. In this paper, we consider the lowest dimension $d = 4$ in this setting  and the smallest integer exponent $ p = 7$ in this dimension. 

Our main result is stated as follows.
\begin{thm}\label{thm:main}
Let $d = 4, p = 7, \ell = \f{4}{p-1}+1$, and $n$ be an odd number and large enough. There exists $\g_n$ accumulating at $\ell^{-1/2}$ with $ \g_n >\ell^{-1/2} , b_n = \f{ d-1 }{ \ell (\g_n + 1)} - 1>0 $  such that the ODE \eqref{eq:ODE} admits a smooth solution $V^\rn{\g_n}  \in C^{\infty}[0, \infty)$ with $V^\rn{\g_n}(0) = 0$, 
\beq\label{eq:thm_prop}
V^\rn{\g_n}(Z) \in (-1, 1),\quad V^\rn{\g_n}(Z) < Z ,
\eeq
for any $Z > 0$, and $V(Z) = Z \td V(Z^2)$ for some function $\td  V \in C^{\infty}[0, \infty)$. Moreover, the solution gives rise to a smooth radially symmetric self-similar profile $(W, \Phi)$ in \eqref{eq:SS_ansatz} and a solution to the 
relativistic Euler equations \eqref{eq:wave_phase} and \eqref{eq:vel} on $[0, T) \times \R^d$:
\beq\label{eq:euler_implod}
u^0(t, x) = U^0(Z),  \quad u^i(t, x) = U(Z) \f{x_i}{r},
\quad \varrho(t, x) = (T-t)^{  - \f{(d-1) (\ell + 1)}{ \ell (\g +1) }  }  W^{p+1}(Z) ,
\eeq
where $r = |x|, Z = \f{r}{T- t}$. The solution develops an imploding singularity at $t = T, x = 0$ and satisfies the asymptotics 
\beq\label{eq:blow_asym}
\lim_{Z \to \infty} U^0(Z) = U^0_{\infty},
\quad \lim_{Z \to \infty} U(Z) = U_{\infty}, 
\quad \lim_{Z \to \infty} \varrho(t, x) = \varrho_{\infty} |x|^{ - \f{(d-1) (\ell + 1)}{ \ell (\g +1) }  }
\eeq
for any $x \neq 0$ and some constants $ U^0_{\infty} \neq 0, U_{\infty} , \varrho_{\infty} > 0$

\end{thm}


To further prove blowup of the nonlinear wave equation \eqref{eq:wave}, we use \cite[Theorem 1.1]{shao2024blow}: 
\begin{thm}[Theorem 1.1 \cite{shao2024blow}]\label{thm:SWZ}
Let $d \in \BZ \cap [4, \infty), p \in 2 \BZ_+ + 1, \ell =  \f{4}{p-1} + 1$ 
with $d-1> \ell + \ell^{1/2}$. Suppose that there exists $\g$ with $\g > \ell^{-1/2}$ and $(d-1) > \ell(\g + 1)$ 
\footnote{
  The parameter $\b, k$ in \cite{shao2024blow} corresponds to $\b = 1 + b = \f{d-1}{\ell(\g + 1)} $ 
\eqref{eq:SS_ansatz}, \eqref{eq:para} and $k = d-1$ in this paper. The condition $\b \in (1, k / (\ell + \ell^{1/2}))$ in \cite[Theorem 1]{shao2024blow} is equivalent to $\g > \ell^{-1/2}$ and 
$d- 1 > \ell(\g+1) $\eqref{eq:para_ineq1}. 
}
such that the ODE \eqref{eq:ODE} with parameters $(d, p, \g)$ admits a smooth solution $V(Z)$ defined in $Z \in [0, \infty)$ satisfying $V(Z) \in (-1, 1), V(0) =0$ and $V(Z) = Z \td V(Z^2)$ for some function $\td  V \in C^{\infty}[0, \infty)$. 

Then there exists smooth compactly supported functions $w_0, w_1: \R^d \to \C$ such that the defocusing nonlinear wave equation \eqref{eq:wave} with nonlinearity $p$ develops a finite time singularity from the initial data 
$w(0) = w_0, \pa_t w(0) = w_1$. 
\end{thm}

In Theorem \ref{thm:main}, since $\g_n > \ell^{-1/2}$ and $\g_n$ converges to $\ell^{-1/2}$ as $n\to \infty$, 
for $n$ large enough, the parameters $(d=4, p=7, \g_n)$ satisfies the inequalities in Theorem \ref{thm:SWZ}. Moreover, the smooth solution $V^\rn{\g_n}$ constructed in Theorem \ref{thm:main} satisfies the assumptions in 
Theorem \ref{thm:SWZ}. Thus, using Theorems \ref{thm:main} and \ref{thm:SWZ}, we establish: 
\begin{cor}\label{cor:blowup_wave}
There exists smooth compactly supported functions $w_0, w_1: \R^4 \to \C$ such that the 
septic defocusing nonlinear wave equation \eqref{eq:wave} with $p=7$ develops a finite time singularity from the initial data 
$w(0) = w_0, \pa_t w(0) = w_1$. 
\end{cor}

\begin{remark}
In dimension $d=4$ with nonlinearity $p \in 2 \N + 1$, the range of parameter in the supercritical case 
$p > 1 + \f{4}{d-2}$ implies that $ p \geq 5$ and $p \in 2 \N + 1 $. Our result resolves the case  $ p=7$. 
It is conceivable that our methods for proving Theorem \ref{thm:main} can apply to a specific pair of parameters $(d, p)$ with $p \in 2 \N + 1$, provided they satisfy \eqref{eq:range_dp}. Note that the parameters $(d, p)$ in the supercritical case $p > 1 + \f{4}{p-2}$ with $d \geq 4$ and $p \in 2 \N + 1$ that do not satisfy \eqref{eq:range_dp} are $(d, p) = (4, 5), (5, 3)$. 

Similar to the results in \cite{shao2024blow}, we only construct blowup for \eqref{eq:wave} with a complex-valued solution. Blowup of \eqref{eq:wave} with real-valued solution remains open.
\end{remark}

The main difficulty to prove Theorem \ref{thm:main} is that the ODE \eqref{eq:ODE} of $(Z, V)$ is singular near a sonic point. Additionally, the ODE degenerates near the sonic point as $\g \to \ell^{-1/2}$.  We remark that, due to the constraint \eqref{eq:range_dp}, constructing the blowup solution to \eqref{eq:wave} with smaller $p$ requires considering $\g$ close to $\ell^{-1/2}$. To overcome these difficulties, we renormalize the ODE \eqref{eq:ODE} to obtain a new ODE of $(Y, U)$ \eqref{eq:UY_ODE}, which is non-degenerate near the sonic point as $\g \to \ell^{-1/2}$. 
To analyze the $(Y, U)$-ODE near the sonic point, we perform power series expansion of $U(Y)$ and estimate the asymptotics of the power series coefficients $U_n$ using induction. Using barrier arguments, a shooting argument, and a few monotone properties, we extend the local power series solution to a global solution of the ODE. 

The proof involves light computer assistance, mainly to derive the power series coefficients $U_i$ for $i \leq 500$ and to verify the signs of a few lower-order polynomials. 
These calculations can be performed on a personal laptop in a few seconds. See more details in Appendix \ref{app:comp}.

We will derive the ODE \eqref{eq:ODE} and its renormalization \eqref{eq:UY_ODE} and analyze its basic properties in Section \ref{sec:ODE}. We  outline the proof and the organization of Sections \ref{sec:pow_sonic}-\ref{sec:proof} in Section \ref{sec:strategy}.

\subsection{Comparison with the methods in existing works}


Our proof share some similarities with \cite{buckmaster2022smooth}, including some barrier arguments,  estimates of the power series coefficients using induction, and computer-assisted proofs. Different from \cite{buckmaster2022smooth}, we need to renormalize the ODE to overcome the degeneracy as $\g \to \ell^{-1/2}$.

In \cite{buckmaster2022smooth}, the estimates of the asymptotics of 
power series coefficients $U_n$ involve ad-hoc partitions of $n$. We develop a systematic approach for the estimates by identifying a few leading order terms in the recursive formula of $U_n$ and treating the remaining terms perturbatively. We check a few properties of $U_i, i=1,2,..,N$ for $N$ suitably large with computer assistance and then perform induction for the remaining coefficients. This allows us to impose a stronger induction hypothesis, resulting in a streamlined approach to estimating the higher order coefficients, 
which further leads to a simplified splitting in $n$.  
Moreover, our scheme of estimates does not utilize special forms of the ODE and are readily generalized to ODE 
\eqref{eq:ODE} or \eqref{eq:UY_ODE} with numerator and denominator being higher order polynomials. Additional simplifications of the combinatorial estimates are made via renormalization of the power series coefficients.
See Section \ref{sec:pow_sonic}.

We follow the connection between the defocusing nonlinear wave equation and the relativistic Euler equation used in \cite{shao2024self} to derive the ODE \eqref{eq:ODE} governing the profiles, use the same renormalization as in \cite{shao2024self}, and adopt a few basic derivations for the ODEs from \cite{shao2024self}. However, there are a few key differences. Firstly, we apply renormalization to overcome the degeneracy of the ODE as $\g \to \ell^{-1/2}$, whereas in \cite{shao2024self}, renormalization is used to simplify computations and obtain blowup results for a wider range of $(d, p)$. We can also employ other renormalization to overcome this degeneracy; see Remark \ref{rem:chnage}. We adopt the renormalization from \cite{shao2024self} to simplify certain presentations and derivations that are less essential. Secondly, since we consider a smaller nonlinearity ($p = 7$) than those considered in \cite{shao2024self}, we need to use much higher-order barrier functions near the sonic point than those in \cite{shao2024self}, which are based on quartic or lower-order polynomials. Thirdly, by taking $\g \to \ell^{-1/2}$, we obtain a large parameter $\kp(\g)$ \eqref{eq:kappa} and use it and light computer assistance to bypass several detailed computations in 
\cite{shao2024self}.

  We note that in \cite{guo2022gravitational}, Guo, Hadzic, Jang, and Schrecker employed similar arguments, such as Taylor expansions, dynamical systems analysis, and computer-assisted proofs, to construct smooth self-similar solutions for the gravitational Euler-Poisson system.

\subsection{Related results}

In this section, we review related results on singularity formation and computer-assisted proofs in fluid mechanics



\vspace{0.05in}
\paragraph{\bf{Smooth Implosions}}

The construction of the self-similar imploding singularity in Theorem \ref{thm:main} is closely related to that in compressible fluids.
Guderley \cite{Gu1942} was the first to construct radial, non-smooth, imploding self-similar singularities along with converging shock waves in compressible fluids. 
While Guderley’s setting has been extensively studied, the existence of a finite-time smooth, radially symmetric implosion (without shock waves) was only recently established in \cite{merle2022implosion1,merle2022implosion2,merle2022blow}. Subsequently, radially symmetric implosion in $\R^3$ with a larger range of adiabatic exponents was established in \cite{buckmaster2022smooth}, and non-radial implosion was established in \cite{cao2023non,cao2024non}. While these results are inspired by Guderley's setting, constructing $C^\infty$ self-similar implosion profiles is challenging due to the degeneracy at the sonic point, requiring sophisticated mathematical techniques \cite{merle2022implosion1} or computer-assisted methods \cite{buckmaster2022smooth}. 
The smooth imploding solutions are proved to be potentially highly unstable \cite{merle2022implosion2,buckmaster2022smooth,cao2023non,merle2022blow}, with numerical studies of instabilities presented in \cite{biasi2021self}. 
The self-similar profile in Guderley's setting was recently constructed in \cite{2023arXiv231018483J}. By perturbing the radial implosion \cite{merle2022implosion1}, exploiting axisymmetry, and proving the full stability of the perturbation generated by angular velocity, vorticity blowup in the compressible Euler equations was established in \cite{chen2024Euler} for the case of $\R^2$ and in \cite{chen2024vorticity} for $\R^d$ with $d \geq 3$.

\vspace{0.1in}
\paragraph{\bf{Shock formation}}



The prototypical singularity for the compressible Euler equations is the development of a shock wave. Rigorous work regarding shocks traces back to the work of Lax \cite{Lax1964}. Shock wave singularities in the multi-dimensional, irrotational, isentropic setting was first analyzed by Christodoulou in the work \cite{Ch2007} (cf.\ \cite{ChMi2014}). This work was later extended to  include non-trivial vorticity in the 2-dimensional setting by Luk and Speck \cite{LuSp2018}. Restricting to 2-dimensions and assuming azimuthial-symmetry, the first complete description of the formation of a shock singularity, including its self-similar structure was given in a work of the first author, Shkoller, and Vicol \cite{BuShVi2022}. This latter work was extended by the authors to 3-dimensions in the absence of symmetry assumptions in \cite{BuShVi2023a} and to the 3-dimensional non-isentropic setting in \cite{BuShVi2023b}. Returning to the setting of  2-dimensions under azimuthial-symmetry, the first full description of shock development past the first singularity was proven in a work by the first author, Drivas, Shkoller, and Vicol \cite{BuDrShVi2022}. Shock development in the absence of symmetry remains an open problem; however, recently Shkoller, and Vicol in the work \cite{ShVi2023} constructed compact in time maximal development in the 2-dimensional setting, which is a crucial step towards resolving this major open problem (partial results in direction of maximal development were earlier attained in \cite{AbSp2022}).

\vspace{0.1in}
\paragraph{\bf{Computer-assisted proofs in fluids}}

In recent years, there have been substantial developments in computer-assisted proofs in mathematical fluid mechanics. We highlight a few advancements in incompressible fluids: singularity formation in the 3D Euler equations \cite{ChenHou2023a,ChenHou2023b} and related models \cite{chen2019finite,chen2021HL}, constructing nontrivial global smooth solutions to SQG \cite{castro2020global}, and some applications to Navier-Stokes \cite{arioli2021uniqueness,van2021spontaneous}.  See also the survey \cite{gomez2019computer}.

\section{Autonomous ODE for the self-similar profile }\label{sec:ODE}

In this section, we derive the ODE governing the self-similar profile \eqref{eq:SS_ansatz} and perform renormalization to overcome some degeneracy. 

\subsection{ODE for the profile}\label{sec:ODE_sub}
We introduce $\ell = \ell(p)$ and a free parameter $ \g$ to 
rewrite $(a, b)$ in \eqref{eq:SS_ansatz}, \eqref{eq:SS_ansatz2} as follows 
 \beq\label{eq:para}
 \ell = \f{4}{p-1} + 1, \quad 
 b = \f{ d-1}{ \ell (\g+1)} - 1, \quad   
a = \f{2 (d-1) }{ (p-1) \ell (\g + 1) } .
\eeq
where $d$ is the dimension. We impose radial symmetry on the solution $(\rho, \phi)$. As a result, the profiles $W, \Phi$ in \eqref{eq:SS_ansatz2} are radially symmetric. We introduce the radial and self-similar variables 
\bseq\label{eq:SS_var}
\beq
r = |x|, \quad z = \f{x}{ T- t}, \quad  Z = |z| = \f{|x|}{T-t} .
\eeq
From the ansatz \eqref{eq:SS_ansatz} and \eqref{eq:SS_ansatz2}, it is easy to see that 
$  \rho^{- \f{p-1}{2} } \pa_t \phi(x, t) , \rho^{- \f{p-1}{2} } \pa_r \phi(x, t)$ only depend on $Z$. Thus, to solve \eqref{eq:vel}, we consider 
\beq
  \rho^{- \f{p-1}{2} } \pa_t \phi(x, t) =  \f{1}{  (1 - V(Z)^2)^{1/2}} ,  \quad  \rho^{- \f{p-1}{2} } (\pa_r \phi)(x, t) =  \f{V(Z)}{  (1 - V(Z)^2)^{1/2}} .
\eeq
\eseq
The relation \eqref{eq:SS_var} and \eqref{eq:wave_phase} imply that $V(Z)$ solves the following ODE for $Z \geq 0$
\beq\label{eq:ODE}
\bal
  \f{ d V}{d Z}  & = \f{ \D_V(Z, V)}{ \D_Z(Z, V) },  \\
    \D_V  & = (d-1) (1- V^2) \B(  \f{1}{\g + 1} (1 -V^2) Z -  V (1 - V Z)  \B), \\
  \D_Z  & = Z( (1 - Z V)^2 - \ell (V - Z)^2 ) .
\eal
\eeq
This ODE was first derived in \cite{shao2024self}. For completeness, we derive the ODE in Appendix \ref{app:ODE_deri}. 
\footnote{
We have substitute $k =  d - 1, m = \f{d-1}{1 + \g}$ in the ODEs derived in  \cite[Section 2]{shao2024self} to reduce the number of variables. Moreover, we have used $\b = \f{(d-1) (\ell + 1) }{ \ell (\g+1)}$ to remove the parameter $\b$ used in \cite{shao2024self}. Note that the meaning of $\b$ is different in \cite{shao2024self} and \cite{shao2024blow}.
}


\vspace{0.1in}
\paragraph{\bf{Constraints on the parameters}}

From $b > 0$ \eqref{eq:SS_ansatz2}, which is related to the front compression mechanism and the relation \eqref{eq:para}, we first require $\g$ to satisfy
\beq\label{eq:para_ineq1}
 \f{d-1}{\ell (\g + 1)} > 1 \iff d-1 > \ell(\g + 1). 
\eeq

To obtain a smooth ODE solution, we have an additional constraint on $\g$ established in \cite[Lemma 2.1]{shao2024self}.

\begin{lem}[Lemma 2.1 \cite{shao2024self}]\label{lem:gamma_low}
If $V(Z) : [0, 1] \to (-1, 1)$ is a $C^1$ solution to \eqref{eq:ODE} with $V(0) =0$ and $\ell > 1, \g > -1, d >1$, then $\g > \ell^{-1/2}$.\footnote{
  In \cite[Lemma 2.1]{shao2024self}, the assumptions $m > 0, k >0$ are equivalent to $\g > -1, d > 1$ in our notations due to the relation $m = \f{k}{\g + 1} , k = d-1$. 
  \cite[Lemma 2.1]{shao2024self} implies that $ k > m (1 + \ell^{-1/2})$, which is equivalent to $\g > \ell^{-1/2}$. Again, we do not use the parameters $k, m$ as \cite{shao2024self} to reduce the number of variables.
}
\end{lem}

Combining \eqref{eq:para_ineq1} and the lower bound $\g > \ell^{-1/2}$ from the Lemma \ref{lem:gamma_low}, we obtain
\beq\label{eq:para_ineq2}
\quad \ell + \ell^{1/2} < d-1,
 \quad \ell < \ell_u(d):= ( \sqrt{ d - \f{3}{4}} - \f{1}{2} )^2 .
\eeq

Since $\ell>1$, we need $d \geq 4$. Recall the power of non-linearity $p = \f{4}{\ell-1}+1, \ell(p) = 1 + \f{4}{p-1}$ \eqref{eq:para}. Let $p(d)$ be the value with $\ell(p(d)) = \ell_u(d)$.  For $d = 4,5,6$, the constraint \eqref{eq:para_ineq2} implies
\[
 p > p(d), \quad p(4) \approx 6.737, \quad p(5) \approx 3.7808, \quad p(6) \approx 2.8110.
\]

In the remainder of the paper, we will focus on the lowest dimension $d = 4$ and the smallest integer power $p=7$ in this dimension. Then, we have 
\beq\label{eq:para1}
 d = 4, \quad p = 7, \quad \ell = \f{5}{3}.
\eeq
We still have one free parameter $\g $, which will be chosen so that $ 0 < \g - \ell^{-1/2} \ll 1$.

\vspace{0.1in}

\paragraph{\bf{Double roots and the sonic point}}

Solving $\D_V(P) = \D_Z(P) = 0$, we get the special points 
\beq\label{eq:root_P}
\begin{gathered}
 P_O =   (0, 0) , 
\quad P_s = ( Z_0, V_0 ) = \B(  \tfrac{ (\g + 1)  \sqrt{\ell}}{ \ell\g + 1  } , 
 \tfrac{ 1} { \g \sqrt{\ell} } \B), \quad P_2 = (1, 1), \quad P_3 = (0, 1), 
\end{gathered}
\eeq
where $P_s$ is the sonic point. The letters $O, s$ are short for \textit{origin, sonic}, respectively.

  The smoothness of the ODE solution to \eqref{eq:ODE} is closely related to 
  the Jacobian of a renormalization of the ODE \eqref{eq:ODE} by the factor $\Delta_Z$: 
\beq\label{eq:grad}
M_P := 
 \begin{pmatrix}
\pa_V \D_V  & \pa_Z \D_V \\
\pa_V \D_Z  & \pa_Z \D_Z \\
 \end{pmatrix}.
\eeq

A direct computation shows that the entries degenerate: $|M_{P, i j}| \les |\g^2 \ell - 1|$, as $\g \to \g_* = \ell^{-1/2}$,  which complicates the analysis of the ODE near the sonic point.


\vspace{0.1in}
\paragraph{\bf{Notations and parameters}}

Throughout the paper, we will use $(Z, V)$ for the variables in the original ODE system \eqref{eq:ODE},  and $(Y, U)$ for the renormalized ODE \eqref{eq:UY_ODE} to be introduced. We further introduce some parameters 
\beq\label{eq:para2}
\bal
 & \e = \ell \g^2 - 1, \quad A = d+ 1- (d-1- 2 \ell) \g,
\quad B = 2 d -1 - \ell,  \\
\eal
\eeq
and will use them to denote the coefficients of a polynomial in the renormalized ODE of $(U, Y)$.

We fix the following constants 
\beq\label{eq:para3}
\hat \d = 0.049, \quad  \d = 0.05 ,
\eeq
and will use them to denote the size of relative error for asymptotic series around the sonic point (see Lemma \ref{lem:asym} and Corollary \ref{cor:asym}).

We denote $A \les B$ if $A \leq C B$ for some absolute constant $C >0$, and denote $ A \asymp B$ if $A \les B$ and $B \les A$.
We denote $A \les_m B$ if $A \leq C(m) B$ for some constant $C(m)>0$ depending on $m$, and define the notations $ A \asymp_m B$ similarly.

\subsection{Renormalization}

To overcome the degeneracy of the Jacobian $M_P$ \eqref{eq:grad} as $\g \to \g_*$, we study an ODE system equivalent to \eqref{eq:ODE} by performing a renormalization of the system. We adopt the change of coordinate  $(Z, V) \to (Y, U) = (\cY, \cU)(Z, V)$ from \cite{shao2024self} with
\beq\label{eq:sys_UY}
\bal
\cY(Z, V) = \f{ (1 - V^2)Z - (\g + 1) V (1 - V Z) }{  Z(1 - V^2)},
\quad \cU(Z, V) =  \f{ (\g + 1)^2 (1 - V Z)^2}{  (1 - V^2) Z^2}.  \\
\eal
\eeq
We will see later that the Jacobian of the new ODE system in $(Y, U)$ at the sonic point 
\eqref{eq:grad_Q} is non-degenerate as $\g \to \g_*$. There are a few other change of coordinates that achieve this purpose. See Remark \ref{rem:chnage}. The above transform \eqref{eq:sys_UY} leads to lower order polynomials $\D_U, \D_Y$ in $U$, which simplifies some analysis. Note that we do not use this property in an essential way.

We can invert the transform from $(Y, U)$ to $(Z, V) = (\cZ, \cV)(Y, U)$ using the following formulas 
\beq\label{eq:UY_to_VZ}
\cZ(Y, U) = \f{ \sqrt{ U + (1 - Y)^2}}{ \f{1}{1 + \g} U + 1 - Y },
\quad \cV(Y, U) = \f{1 - Y}{ \sqrt{ U + (1 - Y)^2 } }  .
\eeq

The bijective properties are proved in \cite[Lemma 3.8]{shao2024self} :
\begin{lem}\label{lem:bijec}
Denote $\cR_{ZV} = \{ (Z, V) : 0 < V < 1, 0 < Z V < 1 \}, \cR_{YU} 
= \{ (Y, U) : U > 0, Y < 1 \}$. Then $(\cY, \cU): \cR_{ZV} \to \cR_{YU}$ is a bijection. 
\end{lem}
A direct computation yields that the inversion of $(\cY, \cU)$ is given by $(\cZ, \cV)$. Thus, for $(Z, V) \in \cR_{ZV}$ or $(Y, U) \in \cR_{YU}$, we have 
\beq\label{eq:invert}
(\cY, \cU) \circ (\cZ, \cV ) =\mw{Id},  \quad (\cZ, \cV ) \circ ( \cY, \cU)  = \mw{Id} .
\eeq

Suppose that $(Z, V(Z))$ solves \eqref{eq:ODE}. We have the following ODE for $(Y, U) 
= (\cY, \cU)(Z, V(Z)) $:
\beq\label{eq:UY_ODE}
\bal
  \f{d U}{d Y} & = \f{\D_U(Y, U)}{\D_Y(Y, U)}, \\
  \D_U & = 2 U ( U + f(Y) + (d-1) Y(1 - Y)),\\
 \D_Y  & = ( d Y - 1) U + (Y -1) f(Y),\quad f(Y) = - \e - A Y + B Y^2 ,
\eal
\eeq
where $\e, A, B$ are defined in \eqref{eq:para2}. We refer to \cite[Section 3.2]{shao2024self} for the derivations of \eqref{eq:UY_ODE}.

Denote by $d_U, d_Y$ the degree of $\D_U, \D_Y$ as a polynomial in $U$, respectively. We have 
\bseq
\beq\label{eq:poly_degUY}
d_U = 2, \quad d_Y = 1.
\eeq
We can expand $\D_U, \D_Y$ as polynomials in $U$ 
\beq\label{eq:ODE_UY}
 \D_U = \sum_{i \leq  d_U} F_i(Y) U^i,
 \quad \D_Y =  \sum_{i\leq d_Y} G_i(Y) U^i , 
\eeq
where $F_i, G_i$ are polynomials in $Y$. We further define the maximum degree of these polynomials in $Y$ as follows 
\beq\label{eq:poly_deg}
d_F = \max_i \deg F_i(Y),
\quad d_G = \max_i \deg G_i(Y) .
\eeq

For the above ODE system \eqref{eq:UY_ODE}, we have 
\beq\label{eq:poly_deg2}
d_F = 2 , \quad d_G = 3 .
\eeq
\eseq

In this new coordinate system $(Y, U)$, the points $P_O, P_s$ defined in \eqref{eq:root_P} 
map to 
\beq\label{eq:sonic_Q}
Q_O = (Y = Y_O, U = \infty), \quad 
Q_s = (Y_0, U_0), 
\quad Y_O = \f{1}{d}, \quad Y_0 = 0,  \quad U_0 =\e.
\eeq

  \begin{remark}\label{rem:chnage}
  One can choose other change of variables to make the matrix $M_P$ \eqref{eq:grad} nonsingular and analyze the ODE in the new system. One candidate is the following simpler transform 
  \[
\cU(Z, V) = \f{ 1 - V Z }{1  - V^2 }, \quad \cY(Z, V) = \f{ V-  Z}{ 1 - V^2 } .
\]

In the new system, the gradient at the sonic point does not degenerate in the limit $\g \to \g_*$
\[
  |\na_{U, Y} \D_U| \asymp 1, \quad  | \na_{U, Y} \D_Y | \asymp 1.
\]
The reason is that the transformation from $(Z, V)$ to $(Y, U)$ is singular near $P_s$ as $\g \to \g_*$ with $ |\na_{Z, V} (\cU, \cY) |_{P_s}  \sim (\g - \g_*)^{-1}$, which compensates the degeneracy: $|\na \D_Z|  , |\na \D_v |\asymp \g - \g_*$.

  \end{remark}

\subsubsection{Change of coordinates}

Recall the maps $\cZ, \cV$ from $(Y, U)$ to $(Z, V)$ \eqref{eq:sys_UY} and 
$(\cY, \cU)$ from  $(Z, V)$ to $(Y, U)$ \eqref{eq:UY_to_VZ}. Denote by $\cM_1, \cM_2$ the Jacobians 
\beq\label{eq:Jacob_M}
\cM_1(Y, U) = 
\begin{pmatrix}
 \pa_Y \cZ  & \pa_U \cZ  \\
 \pa_Y \cV & \pa_U \cV  \\
\end{pmatrix}(Y, U),
\quad \cM_2(Z, V) = 
\begin{pmatrix}
 \pa_Z \cY & \pa_V \cY \\ 
 \pa_Z \cU & \pa_V \cU \\
\end{pmatrix}(Z, V)  .
\eeq
For $(Z, V) \in \cR_{Z V}, (Y, U) \in \cR_{YU}$ with $\cR_{ZV}, \cR_{YU}$ defined in Lemma \ref{lem:bijec},
the denominators in the maps \eqref{eq:sys_UY}, \eqref{eq:UY_to_VZ} do not vanish. Thus, the matrices $\cM_1, \cM_2$ are not singular. Using \eqref{eq:invert},  we have 
\beq\label{eq:invert2}
\cM_1( ( Y, U) \cM_2(Z, V) = \mw{Id},
\quad 
\mw{for} \quad  (Y, U)=(\cY, \cU)(Z, V) .
\eeq
 Since $M_i$ is not singular, we get $\det(\cM_i) \neq 0$. 

Along the ODE solution curve $(Y, U(Y))$ or $(Z, V(Z))$, the ODE \eqref{eq:UY_ODE} and the above change of coordinates implies 
\[
 \f{\D_U}{\D_Y} =  \f{d U}{d Y} = \f{ \f{d \cU(Z, V(Z))}{d Z} }{ \f{d \cY(Z, V(Z))}{d Z}  }
 = \f{ \cM_{2, 21} + \cM_{2, 22}  \f{\D_V}{\D_Z} }{ \cM_{2, 11} + \cM_{2, 12}  \f{\D_V}{\D_Z}  } 
 = \f{ \cM_{2, 21} \D_Z + \cM_{2, 22} \D_V }{ \cM_{2, 11} \D_Z + \cM_{2, 12}  \D_V   } , 
\]
where $Q_{ij}$ denotes the $ij$-th component of a matrix $Q$. Thus, there exists some continuous function $m(Y, U)$ such that 
\beq\label{eq:Mv}
\cM_2 \cdot 
\begin{pmatrix}
\D_Z \\
\D_V \\
\end{pmatrix} = m(Y, U) 
\begin{pmatrix}
\D_Y \\
\D_U \\
\end{pmatrix} , 
\quad 
\cM_1 \cdot 
\begin{pmatrix}
\D_Y \\
\D_U \\
\end{pmatrix} = m^{-1}
\begin{pmatrix}
\D_Z \\
\D_V \\
\end{pmatrix} , 
\quad m \neq 0,
\eeq
for $(Z, V) = ( \cZ, \cV) (Y, U) \in \cR_{ZV} \bsh P_s$ or equivalently $(Y, U) \in \cR_{YU} \bsh Q_s $. For $(Y, U) \in \cR_{YU} \bsh Q_s$, since $\det(\cM_2) \neq 0$ 
and $(\D_Z, \D_V) \neq (0,0)$, we obtain $m(Y, U) \neq 0$.

\subsection{Roots of $\D_Z, \D_V, \D_Y, \D_U$}

We can decompose $\D_V, \D_Z$, defined in  \eqref{eq:ODE} as follows
\beq\label{eq:root_mono_ZV0}
\bal
\D_V &= \f{d-1}{\g+1} ( 1 - V^2) ( 1 + \g V^2 ) ( Z - Z_V(V)),
\quad Z_V(V) = \f{(1 + \g) V}{  1 + \g V^2  } , \\
\D_Z & = Z( V^2 - \ell  )( Z -  Z_{\pm}(V)) ,
\quad Z_{\pm}(V)= \f{ \pm \ell^{1/2} V + 1}{ V \pm \ell^{1/2}}.
\eal
\eeq
The derivations follow from a direct computation. 
See Figure \ref{fig:coordinate1} for an illustration of the curves $Z_V(V)$ (red)
and $Z_{\pm}(V)$ (blue). For $\g < 1 < \ell , |V|<1$, a direct computation yields 
\beq\label{eq:root_mono_ZV}
\bga
Z_V^{\pr}(V) = (1 + \g) \f{ 1 + \g V^2 - 2 \g V^2}{ (1 + \g V^2)^2 } > 0,
\quad Z_{\pm}^{\pr}(V)  = \f{\ell-1}{(V \pm \ell^{1/2})^2} > 0, \\
\quad Z_+(V) - Z_-(V) = \f{2 \ell^{1/2}(V^2-1) }{V^2 - \ell} > 0.
\ega
\eeq

\begin{figure}[t]
  \centering
  \includegraphics[width=0.6\linewidth]{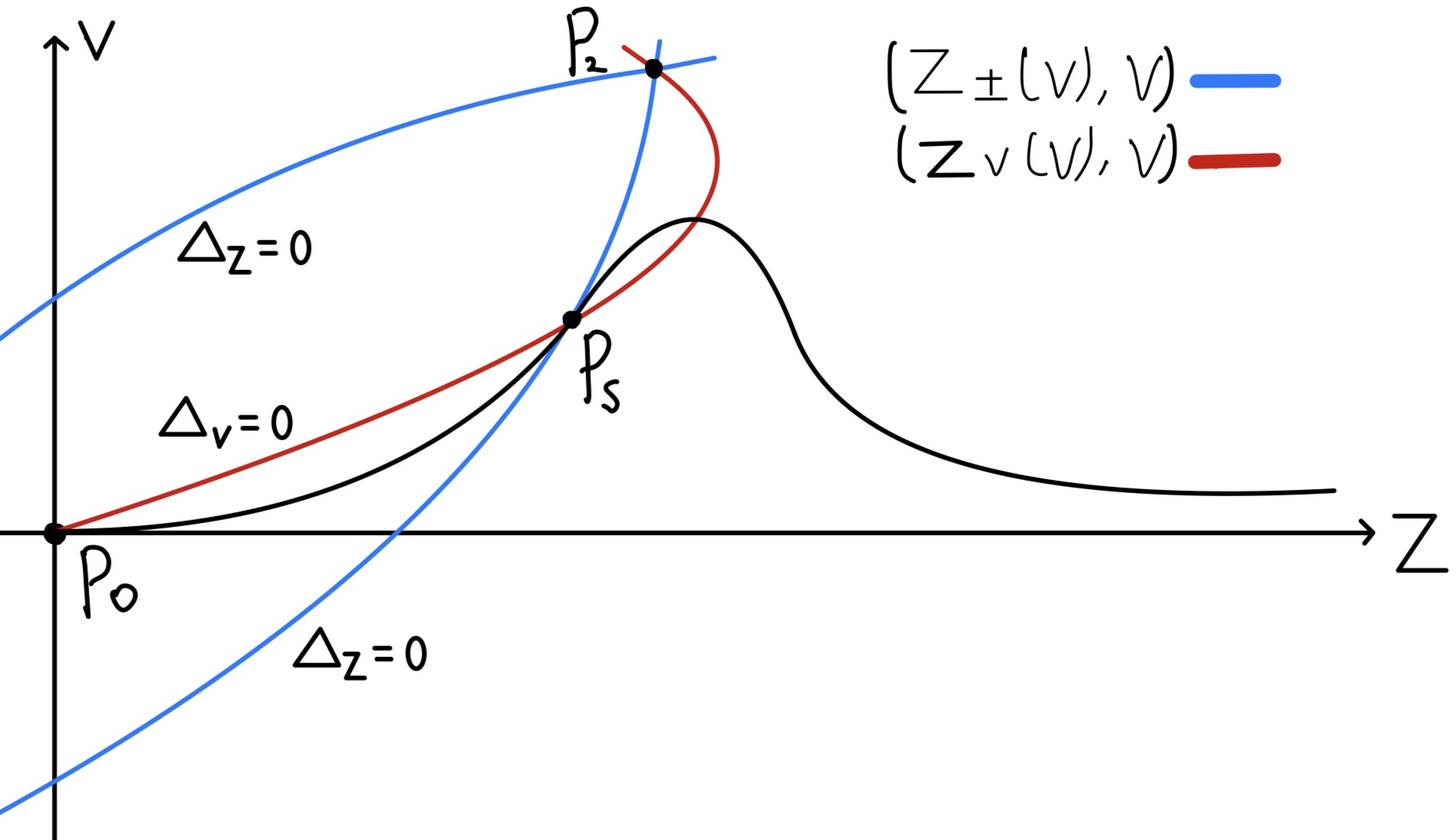}
  \caption{Illustrations of phase portrait of the $(Z, V)$-ODE \eqref{eq:ODE}.
The solution curve is in black and $Z_{\pm}(V), Z_{V}(V)$ defined in \eqref{eq:root_mono_ZV0} are roots of $\D_Z, \D_V$. 
   }
  \label{fig:coordinate1}
\end{figure}

Plugging the formulas of $Z, V$ \eqref{eq:UY_to_VZ} in $\D_Z$ \eqref{eq:ODE}
with $(Y, U) = (\cY, \cU)(Z, V)$, we can also rewrite $\D_Z$ as 
\beq\label{eq:DZ_UY}
\D_Z = Z ( (1 - Z V)^2 - \ell (V-Z)^2 ) 
= \f{Z U^2 ( U + (1 -Y)^2 - \ell (Y + \g)^2 ) }{  ( U + (1 - Y)^2 ) ( U + (1 + \g)(1 - Y)  )^2 }
\B|_{(Y, U) = (\cY, \cU)(Z, V) } .
\eeq
We refer to \cite[Section 3.2]{shao2024self} for the derivation of \eqref{eq:DZ_UY}. 
The numerator can be written as $Z U^2 (U - U_g(Y))$, where we define 
\beq\label{eq:root_DZ}
U_g(Y) = \ell(Y + \g)^2 - (1 - Y)^2.
\eeq

We will use the function $U_g(Y)$ in Section \ref{sec:Q_lower} to control the sign of $\D_Z$ on the solution curve.

We can decompose $\D_U, \D_Y$ \eqref{eq:UY_ODE} as follows 
 \bseq\label{eq:root_UY}
\beq
\D_U = 2 U ( U - U_{\D_U}(Y) ), 
\quad \D_Y = (d Y- 1) (U - U_{\D_Y}(Y) ) ,
\eeq
 with $ U_{\D_U}(Y)), U_{\D_Y}(Y)$ given by 
\beq
U_{\D_Y }(Y) = - \f{ (Y-1) f(Y)}{ d Y - 1},
\quad U_{ \D_U }(Y) = - f(Y) - (d-1) Y(1- Y). 
\eeq
 \eseq

See Figure \ref{fig:coordinate2} for an illustration of the curves $U_{\D_Y }(Y)$ (red), $U_{\D_U}(Y)$ (blue).

We have the following basic properties about $U_{\D_Y}, U_{\D_U}, U_g$. 

\begin{lem}\label{lem:mono}
Suppose that $(d, \ell, \g)$ satisfies $ \ell^{-1/2} < \g < 1, 3 + 2 \ell > d, \ell < d- 1$. For $ Y < 0$, we have $U_{\D_Y}^{\pr}(Y) > 0, U_{\D_U}^{\pr}(Y) > 0$.
For $ - \g < Y < 0$, we have $U_g^{\pr}(Y)>0$. 
\end{lem}

The above properties hold for a wider range of parameters $(\ell, d, \g)$, but we restrict the range to simplify the proof. The parameters in \eqref{eq:para1} with $\ell^{-1/2} < \g < 1$ satisfy the above assumptions.

\begin{proof}
Using the definition of $f(Y)$ \eqref{eq:UY_ODE}, we compute
\[
\bal
 U_{\D_U}^{\pr}(Y) & = - (- A + 2 B Y) - (d-1) +2(d-1) Y 
 = A - (d-1) + 2 (d-1 - B) Y  \\
 & = 2 - (d-1 -2 \ell) \g + 2 (\ell - d) Y  .
 \eal
\]
Since $0 <\g < 1, \ell < d-1,  3 + 2 \ell > d$, and $d \geq 1$, using the formula of $A$ \eqref{eq:para2}, we obtain  
\beq\label{eq:mono_Asign}
A \geq 2 - (d-1 -2 \ell) \g  > \min(2, 2 - (d-1-2\ell) )
= \min(2, 3 + 2 \ell - d ) > 0. 
\eeq
For $Y < 0$, it follows that $2 - (d-1 -2 \ell) \g > 0, 2(\ell - d) Y > 0$, and $ U_{\D_U}^{\pr}(Y) > 0$.

For $U_g^{\pr}$ with $Y \in (-\g, 0)$, we compute 
\[
\bal
U_g^{\pr}(Y) &= 2 \ell (Y + \g) -2(Y-1)
= 2(\ell - 1) Y + 2 \ell \g + 2 \\
& > - 2 (\ell - 1) \g  + 2 \ell \g + 2
= 2 \g+ 2 > 0.
\eal
\]

For $ U_{\D_Y}^{\pr}(Y)$, we compute 
\[
 U_{\D_Y}^{\pr} = \f{  C(Y)  }{(dY-1)^2},
 \quad C(Y) = - ( (Y-1) f(Y) )^{\pr}(d Y-1) + d(Y-1) f(Y).
\]
Thus, we only need to show $C(Y) > 0$. Using the formula of $f(Y)$ \eqref{eq:UY_ODE}, we compute 
\[
\bal
(Y-1) f(Y)
& = (Y-1) (-\e - A Y + BY^2) =  B Y^3 - (A + B) Y^2 + (A-\e) Y + \e , \\
( (Y-1) f(Y))^{\pr}
& = 3  B Y^2 - 2(A + B) Y + A-\e .
\eal
\]

It follows 
\[
\bal
C(Y) & = - ( 3  B Y^2 - 2(A + B) Y + A-\e   ) (d Y - 1)
+ d (  B Y^3 - (A + B) Y^2 + (A-\e) Y + \e ) \\
& = - 2 B d Y^3 + (A d + B(3 +d)) Y^2 -2(A + B) Y + A +  (d-1 )\e.
\eal
\]

Recall $A > 0$ from \eqref{eq:mono_Asign}. Since $ \g > \ell^{-1/2}$ and $\ell < d-1$, we get $\e > 0$ and $B = 2d-1-\ell > d > 0$ \eqref{eq:para2}. Using the estimates of the sign of each term and $Y \leq 0$, we obtain that each monomial in $C(Y)$ is non-negative for $Y \leq 0$. Since
$ A +  (d-1 )\e > 0$, we prove $C(Y) > 0$.
\end{proof}

\subsection{Eigen-system near $Q_s$}
We define 
\bseq\label{eq:grad_Q}
\beq
M_Q := \begin{pmatrix}
\pa_U \D_U  & \pa_Y \D_U \\
\pa_U \D_Y  & \pa_Y \D_Y \\
 \end{pmatrix} \B|_{Q= Q_s}= 
   \begin{pmatrix}
c_1 & c_3 \\
c_2   &  c_4  \\
 \end{pmatrix}.
 \eeq
Using the definitions in \eqref{eq:UY_ODE} and a direct computation, 
at $ (Y, U) =Q_s =(0, \e )$, we yield 
\footnote{
 These formulas have also been derived in \cite[Sections 4.1, 4.2]{shao2024self}.  }
\beq\label{eq:grad_Q_b}
\bal
c_1 & = 4 U + 2 (f(Y)+ k Y(1 - Y))
= 2 \e,  \  && c_3 = 2 U ( d-1 - 2 k Y - A + 2 B Y) 
= 2 \e( d -1 -A), \\
 c_2 &= - 1, \  && c_4 = d U + f(0) - (-A)= (d-1 ) \e + A , \\
\eal
\eeq
\eseq
where $\e, A, B$ are defined in \eqref{eq:para3}. The eigenvalues of $M_Q$ are given by 
\bseq\label{eq:grad_lam_two}
\beq\label{eq:grad_lam}
\bal
 & \lam^2 - (c_1 + c_4)\lam + (c_1 c_4 - c_2 c_3) = 0, \\
 & \lam_{\pm} = \f{  c_1 + c_4 \pm \sqrt{  (c_1 - c_4)^2 + 4 c_2 c_3   }  }{2 } .
 \eal
\eeq
Using the above formulas of $c_i$ and \eqref{eq:para1}, \eqref{eq:para2}, we compute 
\beq\label{eq:grad_lamb}
\bga
 c_1 + c_4 = (d+1) \e + A > 0,  \quad  c_1 c_4 - c_2 c_3 = 2 (d-1) \e (\e + 1) > 0,  \\
 c_3 = 2 \e ( (d-1- 2 \ell) \g - 2) < 0, \quad  (c_1 + c_4)^2   -4(c_1 c_4 - c_2 c_3) > 0 .
\ega
\eeq
 \eseq
We view all the parameters as functions in $\g$. From \eqref{eq:grad_lam_two}, for $\g > \ell^{-1/2}$, we obtain 
\beq\label{eq:lam_sign}
 0 < \lams < \laml  .
\eeq
Since $\e|_{\g = \ell^{-1/2}} = 0$ \eqref{eq:para2}, using the above estimates and \eqref{eq:para2}, we obtain 
\beq\label{eq:lam_limit}
 \lim_{ \g \to (\ell^{-1/2})^+ } \lams(\g) = 0,
 \quad  \lim_{ \g \to ( \ell^{-1/2})^+ } \laml(\g) = \laml(\ell^{-1/2}) > 0.
\eeq

We introduce the parameter 
\beq\label{eq:kappa}
\kp(\g) = \f{\laml}{\lams},
\eeq
related to the asymptotics of the power series of $U$ near the sonic point. 
Using \eqref{eq:grad_lam_two}, we obtain
\beq\label{eq:grad_smooth}
\f{ ( \kp(\g) + 1)^2 }{\kp(\g)}
= \f{ (\laml + \lams)^2}{ \laml \lams}
 = \f{ ( c_1 + c_4)^2}{ c_1 c_4 - c_2 c_3}
 = \f{ ( (d + 1 ) \ell \g - (d - 1-2 \ell ) )^2 }{2 (d-1) \ell (\ell \g^2 - 1)} .
\eeq

We want to study the smooth self-similar profile with $\kp(\g)$ sufficiently large, which provides an important large parameter in our analysis. Since $\g \ell > \ell^{1/2}> 1$, the numerator is always larger than $4$. Thus, we want to choose $\g > \g_*$ with $\g$ close to $\g_*$
\beq\label{eq:gamma_star}
 \g \to  (\g_*)^+,  \quad \g_* =  \ell^{-1/2 } ,
\eeq
which is consistent with the constraint in Lemma \ref{lem:gamma_low}. Note that given $\kp$, we can determine $\g$ via \eqref{eq:grad_smooth}. It is not difficult to see that $\kp(\g)$ decreases in $\g$ for $\g \in (\ell^{1/2}, \ell^{1/2} + c)$ with some absolute constant $c>0$, $\kp(\g) $ is a smooth bijection, 
and it admits an inverse map $\g = \G(\kp)$:
\beq\label{eq:kappa_bi}
\bga
\kp(\g) : I_{\g} \to I_{\kp}, 
\quad \G(\kp) : I_{\kp} \to I_{\g} ,  \quad \G(\kp(\g)) = \g ,  \\
I_{\g} \teq (\ell^{1/2}, \ell^{1/2} + c) ,
\quad I_{\kp} \teq  ( C_{\kp}, \infty )  ,
\quad  C_{\kp} = \kp(\ell^{1/2} + c). 
\ega 
\eeq

Thus, for each integer $n > C_{\kp}$, there exists strictly decreasing $\g_n > \ell^{1/2}$ such that 
\beq\label{eq:kappa_n}
\kp(\g_n) = n > C_{\kp} .
\eeq
In the remainder of the paper,  we drop the dependence of $\kp$ on $\g$ for simplicity. Whenever we refer to choosing $\kappa$ sufficiently large, it is equivalent to taking $\gamma$ sufficiently close to $\ell^{-1/2}$.

Next, we use the ODE \eqref{eq:ODE_UY} to compute the slope of the solution $(Y, U(Y))$ at 
the sonic point $Q_s$, which is given by $(1, U_1)$ with 
\[
U_1 = \f{d U}{d Y} \B|_{ Y = Y_0} .
\]
Denote $ \xi = Y - Y_0$. Applying Taylor expansion near  $Q_s =(Y_0, U_0)$ \eqref{eq:grad_Q}, we get 
\beq\label{eq:eqn_U10}
 \D_{\al}(Y, U(Y)) |_{Q_s} = \f{d}{d Y}  \D_{\al}(Y, U(Y)) \B|_{ Y = Y_0} \xi + O( \xi^2)
  =( \pa_Y \D_{\al} + \pa_U \D_{\al} \cdot U_1 ) \xi + O( \xi^2).
\eeq
Using the derivation for $\na_{Y, U} \D_{\al}$ from \eqref{eq:grad_Q} and the ODE \eqref{eq:ODE_UY}, we yield 
\beq\label{eq:eqn_U1}
 U_1  = \f{ c_1 U_1 + c_3}{ c_2 U_1 + c_4}, \quad c_2 U_1^2 + (c_4 - c_1) U_1 - c_3 = 0 .
\eeq
It is not difficult to show that \eqref{eq:eqn_U1} implies that $(1 , U_1)$ is an eigenfunction of  $M_Q$ \eqref{eq:grad_Q}. Moreover, using \eqref{eq:grad_Q_b} and \eqref{eq:grad_lamb}, we get $-c_3 / c_2 = c_3 < 0$. Thus, the above equation has two roots $U_{1, -}< 0 < U_{1, +}$. The eigenfunction associated with eigenvalue $\lam$ is parallel to 
\[
\uu_{\lam} = \begin{pmatrix}
c_3  \\
\lam - c_1  \\
\end{pmatrix}
=
\begin{pmatrix}
\lam - c_4 \\
c_2 \\
\end{pmatrix}  .
\]
We want to construct a smooth curve that pasts through $Q_s$ along the positive direction $(1, U_{1, +})$. See the black curve in Figure \ref{fig:coordinate2} for an illustration of the direction. Using \eqref{eq:grad_lam_two}, \eqref{eq:grad_Q}, \eqref{eq:para1}, \eqref{eq:para2}, for $\g$ close to $\ell^{-1/2}$, which implies $0< \lams , \e \ll 1$, we get that the direction $(1, U_{1, +})$ corresponds to $u_{\lams}$. We fix $U_1 = U_{1, +}$. Since $(1, U_1)$ and $u_{\lams}$ are parallel, we can represent $U_1$ and $\kp(\g)$ \eqref{eq:grad_smooth} as follows 
\beq\label{eq:grad_lam2}
\begin{gathered}
 U_1 = \f{\lams - c_4}{c_2}, \quad \lams = c_2 U_1 + c_4,
\quad \laml = c_1 + c_4 - \lams = - c_2 U_1 + c_1,  \\
 \kp(\g) = \f{\laml}{\lams} = \f{ c_1 - c_2 U_1}{  c_2 U_1 + c_4 }.  
\end{gathered}
\eeq

\subsection{Outline of the proof}\label{sec:strategy}

In this section, we outline the proof of Theorem \ref{thm:main}. 

In Section \ref{sec:pow_sonic}, we analyze the ODE \eqref{eq:UY_ODE} near the sonic point $ Q_s = ( Y_0, U_0)$ \eqref{eq:sonic_Q} by constructing power series solution $U(Y)$ and estimating the asymptotics of the power series coefficients $U_n$ using induction. 

In Section \ref{sec:shoot}, we use a double barrier argument, a shooting argument, to extend the local power series solution $U(Y)$ to $Y \in [ -c, Y_O - c ]$ for some small $c$, which along with a gluing argument gives rise to a smooth ODE solution $V(Z)$ to \eqref{eq:ODE} for $Z \in [0, Z_0 + \e_1]$ with small $\e_1 > 0$. 
It corresponds to a solution curve $(Z, V)$ connecting $P_O, P_s$. See the black curve in Figure \ref{fig:coordinate1} or the curves above $Q_s$ in Figure \ref{fig:coordinate2} for  illustrations.

In Section \ref{sec:Q_lower}, we use a barrier argument and some monotone properties to 
extend the local power series solution $U(Y)$ for $Y < 0$. 
We further study a desingularized ODE \eqref{eq:ODE_ds} of $(Y_\ds, U_\ds)$ and extend the solution curve $(Y_\ds, U_\ds)$ below $Q_s$ across $Y = 0$. See the black curve below $Q_s$ in Figure \ref{fig:coordinate2} for an illustration.

In Section \ref{sec:proof}, we glue the solutions of the $(Z, V)$-ODE \eqref{eq:ODE}, the $(Y, U)$-ODE \eqref{eq:UY_ODE}, and the desingularized ODE $(Y_\ds, U_\ds)$ \eqref{eq:ODE_ds}, and use some monotonicity properties of the $(Z, V)$-ODE \eqref{eq:ODE} to obtain a global solution to the $(Z, V)$ ODE. 
We then proceed to prove Theorem \ref{thm:main}.

\section{Power series near the sonic point}\label{sec:pow_sonic}

In this section, we study the behavior of the ODE \eqref{eq:UY_ODE} near the sonic point $ Q_s = ( Y_0, U_0)$ \eqref{eq:sonic_Q} using power series expansion
 \beq\label{eq:Tal_U}
 Y = y_0 + \xi, \quad 
U \teq \sum_{ n \geq 0} u_n \xi^n.
\eeq

Although we have $Y_0 = 0$ in our case, we aim to develop a general method which does not depend on fine properties of the ODEs, e.g., specific value or degree of the polynomials. 

\subsection{Recursive formula }\label{sec:recur}

In this section, our goal is to establish Lemma \ref{lem:pow_recur}
for the recursive formula of the power series coefficients $U_i$ of the ODE solution near the sonic point $Q_s = (Y_0, U_0)$.

\vspace{0.1in}
\paragraph{\bf{Notations}}
Throughout this section, we use $y_0,  u_0, u_1,..$ to denote \textit{variables}, and $Y_0, U_0, U_1,..$ to denote the \textit{value} of the power series coefficients of the solution $(Y, U)$ to the ODE \eqref{eq:UY_ODE}.

Firstly, for any two power series $A, B$ with coefficients $\{A_n\}_{n\geq 0} ,  \{B_n\}_{n\geq 0}$, i.e. 
$C(Y) = \sum_{n \geq 0} C_n \xi^n  , C= A, B$, we get 
\beq\label{eq:power_prod}
AB = \sum_{n \geq 0}  ( AB )_n  \xi^n ,
\quad   (AB)_n = \sum_{i \leq n}   A_i B_{n-i}.
\eeq

Given any power series $U$  with coefficients $\{ u_i \}_{i\geq 0}$ and $Y = y_0 + \xi$ \eqref{eq:Tal_U}, since $\D_Y, \D_U$ are polynomials of $U, Y$, using the above convolution formula, we can obtain the power series of $\D_{\al}$
\beq\label{eq:Tal_del}
\bal
  \D_\al( Y, U ) & = \sum_{n \geq 0}  \D_{\al, n}(y_0, u_0, u_1,.., u_n) \xi^n ,
  \quad \al = U, Y .
\eal
\eeq
Applying chain rule, one can obtain that $\D_{\al, n}$ only depends on $y_0, u_0,.., u_n$. For example, we have the following formula for the first term $\D_{\al, 1}$
\beq\label{eq:Tal_del1}
 \D_{\al, 1}(y_0, u_0, u_1) = \f{d}{d \xi } \D_{\al}(Y, U ) \B|_{ \xi = 0}
 = u_1 \cdot \pa_U \D_{\al}( y_0, u_0)
 + \pa_Y \D_{\al}(y_0, u_0) .
\eeq

We introduce the following functions 
\beq\label{eq:recur_const}
c_{\al, j}( y_0, u_0, .., u_j)  \teq \f{1}{j !}  \B( \f{d^j }{d \xi^j }  (\pa_U \D_{\al})(Y, U ) \B)  \B|_{\xi = 0} ,
 \quad \al = U, \ Y .
\eeq
It is easy to see that $c_{\al, j}$ only depends on $y_0, u_0, u_1, .., u_j$ using the chain rule.  
We will use them in Lemmas \ref{lem:pow_lead}, \ref{lem:pow_recur}. 
For simplicity, we drop the dependence of $c_{\al, j}$ on $y_0, u_i$.

We have the following formula regarding the coefficients of $\D_{\al, n}, \al= U, Y$.
\begin{lem}\label{lem:pow_lead}
Let $c_{\al, i}$ be the functions defined in \eqref{eq:recur_const}. For any $(i, j, n)$ with $n \geq 1,  i \geq 0$,  $ j > i$, and $\al = U, Y$, the coefficients $\D_{\al, \cdot}$ defined via \eqref{eq:Tal_del} satisfies
\bseq\label{eq:pow_lead}
\begin{align}
  \pa_{u_n} \D_{\al, n + i}(y_0, u_0, .., u_{n+i}) & =  c_{\al, i} , 
   \label{eq:pow_lead_a} \\
   \pa_{u_j} \pa_{u_n} \D_{\al,  n+ i}(y_0, u_0, .., u_{ n+ i}) & = 0. 
      \label{eq:pow_lead_b} 
\end{align}
Specifically, $ \pa_{u_n} \D_{\al, n + i}(y_0, u_0, .., u_{n+ i})$ does not depend on $u_j$ for $j > i$. 
For any $i < n$, we have 
\beq\label{eq:pow_lead_c}
   \pa_{u_n} \D_{\al,  i}(y_0, u_0, .., u_{i}) = 0. 
\eeq
\eseq

\end{lem}

\begin{proof}

Recall the expansion \eqref{eq:Tal_U}. For $ i <n$, since $\D_{\al, i}$ only depends on $y_0, u_0, .., u_i$, we obtain \eqref{eq:pow_lead_c} trivially. 

Since $U = \sum_{i\geq 0} u_j \xi^j$, taking $\pa_{u_n} $ on both sides of \eqref{eq:Tal_del} 
and then using \eqref{eq:pow_lead_b}, we yield
\[
 (\pa_U \D_{\al})(Y,  U ) \cdot \xi^n
  = \sum_{ j \geq n}  \pa_{u_n} \D_{\al, j}(y_0, u_0, .., u_{n}, .., u_j) 
 \xi^j 
\]


Matching the coefficients of $\xi^{n+ i}$ on both sides and then evaluating at $\xi = 0$, we get 
\beq\label{eq:pow_lead_pf}
 c_{\al, i} =  \f{1}{i!} \B( \f{d^i}{ d \xi^i} (\pa_U \D_{\al})( Y, U  \B) \B|_{\xi = 0} = \pa_{u_n} \D_{\al, n + i}(y_0, u_0, .., u_{n+ i}) 
\eeq
and prove \eqref{eq:pow_lead_a}. It is not difficult to see that the left hand side of \eqref{eq:pow_lead_pf} only depends on $U$ via $u_0, u_1, .., u_i$. Therefore, if $ j > i$, we yield 
\[
 \pa_{u_j}  \pa_{u_n} \D_{\al, n+i} = \pa_{u_j} c_{\al, i} = 0,
\]
and prove \eqref{eq:pow_lead_b}. 
\end{proof}

\subsubsection{Derivation of the recursive formulas}

Now, we derive the recursive formula for the power series coefficients $U_n, n \geq 0$ of $U$ \eqref{eq:Tal_U} near $Q_s = (Y_0, U_0)$: 
\beq\label{eq:Tal_U2}
 U = \sum_{i \geq 0} U_i \xi^i, \quad Y = Y_0 + \xi 
\eeq
with $U$ satisfying the ODE \eqref{eq:ODE}.

Firstly, recall the notations from \eqref{eq:sonic_Q} and \eqref{eq:grad_Q}. We have $Q_s = (  U_0,  Y_0)$ and 
\beq\label{eq:Tal_nota}
\D_{\al}(Q_s) = 0, \quad \al = U, Y .
\eeq
Thus, the leading coefficient satisfies $\D_{\al, 0} = 0$ for $\al = U, Y$. Using the ODE \eqref{eq:ODE} together with the expansions \eqref{eq:Tal_del} and \eqref{eq:Tal_U2}, we obtain 
\[
 \D_Y(Y, U(Y)) \f{d U}{d Y} = \D_U(Y, U(Y )) , \quad  
  \f{ d U}{d Y} = \sum_{i \geq 0} (i+1) \xi^{i}   U_{i+1} .
\]
Matching the coefficients of $\xi^n$, we get
\beq\label{eq:recur_0}
 \sum_{ 1 \leq i \leq n}  \D_{Y, i} \cdot (n-i+1)  U_{n- i + 1} -   \D_{U, n} = 0, 
\eeq
where $\D_{Y, i}, \D_{U, i}$ evaluate on $y_0= Y_0, u_j =  U_j,  j\leq i$.

Let us define
\beq\label{eq:recur_topN_coe0}
\mfr R_n(y_0, u_0, .., u_n) =  \sum_{ 1 \leq i \leq n}  \D_{Y, i}(y_0, u_0, .., u_n)  \cdot (n-i+1)  u_{n- i + 1} -  \D_{U, n} (y_0, u_0, .., u_n) .
\eeq
From \eqref{eq:recur_0}, evaluating at ($Y_0, U_0, .., U_n$), we have
\beq\label{eq:recur_Rn0} 
\mfr R_n(Y_0, U_0, .., U_n) =0.
\eeq
Next, we derive another formula of $\mfr R_n$ by determining the explicit dependence on the $N$ terms $u_n,.., u_{n-N+1}$. Then, by evaluating 
$\mfr R_n$ on $y_0 = Y_0, u_i = U_i, i \geq 0$, we can obtain the recursive formula of $U_n$ in terms of $U_i$ for  $i < n$. We will choose 
\beq\label{eq:para_N}
N = O(1) \ll n, \quad n - N > \f{n}{2}+2 .
\eeq
When we use the recursive formula later, e.g., in Section \ref{sec:asym_coe}, we choose $n\geq n_1$ with $(n_1, N)$ given in \eqref{eq:induc_para}. 

\vspace{0.1in}
\paragraph{\bf{Expansion of $\mfr R_n$}} 
For $ m$ with $ n - N < m \leq n$, taking $u_m$ derivative on \eqref{eq:recur_topN_coe0} and using \eqref{eq:pow_lead_c}, we obtain
\[
\bal
 \pa_{u_m} \mfr R_n
 & =  \sum_{ m \leq i \leq n} \pa_{u_m} \D_{Y, i} (n-i + 1) u_{n-i+1} 
 +  m \D_{Y, n+1 - m } - \pa_{u_m} \D_{U, n} ,
 \eal
\]
and then using \eqref{eq:pow_lead_a} and changing $ i = m + j$, we get 
\[
 \pa_{u_m} \mfr R_n 
 = \sum_{ 0\leq j \leq n- m} c_{Y, j} ( n - m - j + 1) u_{n+1 - m -j}
 +   m \D_{Y, n+1 - m } - c_{U, n- m}.
\] These identities hold for functions $ \mfr R_n, c_{\al, i}, \D_{\al, i}$ evaluated in any $y_0, u_j, j \geq 0$. 

For $ m$ with $n - N< m \leq n $ and $N < (n-2)/2$ (see the bound in $N$ \eqref{eq:para_N}), we have $n + 1 - m , n+1- m -j \leq N < m$. Thus, the right-hand side of the above formula depends only on $y_0, u_0, u_1, .., u_N$ and is independent of $u_{m}$, and $\mfr R_n$ depends linearly on $u_m$.

To simplify the notations, we introduce the following functions 
\bseq\label{eq:recur_topN_coe3}
\beq\label{eq:recur_topN_coe3a}
\bal
e_{ l} 
& = \sum_{0 \leq j \leq l } c_{Y, j} (l +1  - j)  u_{l +1 - j}
- c_{U, l } ,
\eal
\eeq
where $c_{Y,j}, c_{U,j}$ are defined in \eqref{eq:recur_const}.
Then we can rewrite $ \pa_{u_m} \mfr R_n $ as follows 
\beq\label{eq:recur_topN_coe3b}
a_{n, m}(y_0, u_0, .., u_{n-m}, u_{n-m + 1}) \teq  \pa_{u_m} \mfr R_n 
=  e_{ n- m}  + m \D_{Y, n -m + 1}. 
\eeq
\eseq

The above discussion implies the following expansion of $\mfr R_n$ in the top $N$ terms 
\beq\label{eq:recur_topN_0}
\mfr R_n(y_0, u_0, .., u_n) = \sum_{ n-N < m \leq n}   a_{n, m} u_m + \mfr R_n(y_0, u_0, .., u_{n - N}, 0, ..,0 ) .
\eeq

Combining \eqref{eq:recur_Rn0} and \eqref{eq:recur_topN_coe0}, we have the following results.

\begin{lem}\label{lem:pow_recur}
Given $  U_0,  Y_0, U_1$, for $ 2 N < n -2$, the coefficients $\{ U_i \}_{i\geq 0}$ of the power series  \eqref{eq:Tal_U} solving \eqref{eq:ODE} near $(U_0, Y_0)$ satisfy
\beq\label{eq:recur_topN}
\bal
\sum_{ n -N < m \leq n} a_{n, m}  U_m
& = \D_{U, n}( Y_0, U_0, ..,  U_{n-N}, 0,..,0 ) \\
 & \quad - \sum_{ N+1 \leq i \leq n} (n+1 - i)  U_{n +1 - i}
 \D_{Y, i}( Y_0, U_0, ..,  U_{n -N}, 0.., 0),
 \eal
\eeq
where we abuse notation by denoting $a_{n, m}$ the value of the functions $a_{n, m}$ 
\eqref{eq:recur_topN_coe3} evaluating on $y_0 = Y_0, u_j = U_j$ for $j \geq 0$.
For $m$ with $n - N < m \leq n$, $a_{n, m}$ only depend on $Y_0, U_0, U_1, .., U_N$. Here, we use the convention 
\[\D_{Y, i}( Y_0, U_0, ..,  U_{n -N}, 0.., 0)
= \D_{Y, i}(Y_0, U_0,  .., U_i)\]
if $n- N \geq i$.
In particular, for $ N = 1$ and $ n \geq 2$,  we have
\beq\label{eq:recur_top}
\bal
( n \lams - \laml)  U_n 
 & = \D_{U, n}( Y_0, U_0, ..,  U_{n-1}, 0 )
 - \sum_{ 2 \leq i \leq n} (n+1 - i)  U_{n +1 - i}
 \D_{Y, i}( Y_0, U_0, .. , U_{n-1}, 0) ,
 \eal
\eeq
where $\lams, \laml$ are defined in \eqref{eq:grad_lam} and the right hand side is independent of $U_n$.

Recall $c_1, c_2$ from \eqref{eq:grad_Q}. With $U_n$, we can further update $\D_{U,n}, \D_{Y,n}$ as follows 
\beq\label{eq:recur_top_update}
\bal
 \D_{U,n}( Y_0, U_0, ..,  U_{n-1}, U_n )
  & =  \D_{U,n}( Y_0, U_0, ..,  U_{n-1}, 0 )
  + c_1 U_n , \\
\D_{Y,n}( Y_0, U_0, ..,  U_{n-1}, U_n )
  & =  \D_{Y,n}( Y_0, U_0, ..,  U_{n-1}, 0 )
  + c_2 U_n  .
  \eal
\eeq

\end{lem}

\begin{remark}[Parameters and coefficients]
Since in the rest of the paper, we only use the values of the functions $c_{\al, i}, e_{i}, a_{n, m}$ \eqref{eq:recur_const}, \eqref{eq:recur_topN_coe3} evaluating on $y_0 = Y_0, u_j = U_j$ for any $j \geq 0$, we abuse notation by using $c_{\al, i}, e_{i}, a_{n, m}$ to denote the values. 
\end{remark}

\begin{proof}

Evaluating \eqref{eq:recur_topN_0} on $y_0 = Y_0, u_j = U_j$ for any $j \geq 0$, and applying 
$\mfr R_n = 0$ \eqref{eq:recur_0} and the definition \eqref{eq:recur_topN_coe0} to $\mfr R_n(Y_0, U_0, .., U_{n-N}, 0,0, .., 0)$, we prove the identity \eqref{eq:recur_topN}. From \eqref{eq:Tal_del}, we know that $\D_{\al, i}$ only depends on $Y_0, U_0, .., U_i$.

Note that the right hand side of \eqref{eq:recur_top} is the same as \eqref{eq:recur_topN} with $N=1$. Thus to prove \eqref{eq:recur_top}, 
we only need to evaluate $a_{n, n}$ on $Y_0, U_j, j \geq 0$. We first compute $  c_{U, 0}(Y_0, U_0), c_{Y, 0}(Y_0, U_0)$. Using the definition \eqref{eq:recur_const} and the notations \eqref{eq:grad_Q}, we have
\beq\label{eq:recur_const_2}
\begin{gathered}
  c_{U, 0}(Y_0, U_0) = \pa_U \D_U(Y_0, U_0)= c_1, \quad  c_{Y, 0}(Y_0, U_0)
 = \pa_U \D_Y(Y_0, U_0) =  c_2 .
 \end{gathered}
\eeq
Using \eqref{eq:recur_topN_coe3} and the formulas of $c_{U, 0}, c_{Y, 0}$ \eqref{eq:recur_const_2}, $\D_{Y, 1}$ \eqref{eq:Tal_del1}, and then $\lams, \laml$ \eqref{eq:grad_lam2}, we obtain 
\[
\bal
a_{n, n}  & = e_{1,0} + n e_{2, 0}
= c_{Y, 0}U_1 - c_{U, 0} + n \D_{Y, 1}
= c_2 U_1 - c_1 + n  ( U_1 \pa_U \D_Y + \pa_Y \D_Y ) \\
& =  c_2 U_1 - c_1 + n  ( c_2 U_1  + c_4 ) = n \lams - \laml ,
\eal 
\]
and prove \eqref{eq:recur_top}. Using \eqref{eq:pow_lead_a} with $i=0$, we obtain 
\[
\pa_{u_n} \D_{U, n}(y_0, u_0, .., u_n) = c_{U,0} = c_1, \quad 
\pa_{u_n} \D_{Y, n}(y_0, u_0, .., u_n) = c_{Y,0} = c_2.  
\]
From \eqref{eq:recur_const_2}, we obtain that $c_1, c_2$ is independent of $U_n$ \eqref{eq:recur_const_2}. Thus, we prove \eqref{eq:recur_top_update}. 
\end{proof}

To simplify the estimate of the power series coefficients, we consider the renormalized coefficient $\hat U_n$
\bseq\label{eq:cat_num}
\beq
U_n = \hat U_n \mfr C_n , \quad  \mfr C_n = \f{1}{n+1} \binom{2n}{n}, 
\eeq
where $\mfr C_n$ is the Catalan number and $\mfr C_n$ satisfies the following asymptotics and identity 
\beq\label{eq:cat_num_b}
\f{ \mfr C_{n+1}}{ \mfr C_n } = \f{4 n + 2}{ n + 2} ,
\quad 
\lim_{i \to \infty} \f{ \mfr C_{i+1}}{ \mfr C_i } = 4, 
\quad  \sum_{i=0}^n \mfr C_{i} \mfr C_{n-i} =\mfr C_{n+1}. 
\eeq
for any $ n \geq 0$. See \cite[Section 1]{stanley2015catalan} for the identity. Moreover, for $n \geq 4$, we yield 
\beq\label{eq:cat_num_c}
 3 \mfr C_n \leq \mfr  C_{n+1}  < 4 \mfr C_n. 
\eeq

\eseq

\subsection{Asymptotics of the coefficients}\label{sec:asym_coe}

In this Section, we develop refine estimates of $U_n, \hat U_n$, which are crucial for the barrier argument Sections \ref{sec:shoot}, \ref{sec:Q_lower}. Although the relations \eqref{eq:recur_topN} and \eqref{eq:recur_top} seem to be complicated, we will show that the top two terms $a_{n, n}  U_n$ and $a_{n, n-1}  U_{n-1}$ are the dominated terms and determine the asymptotics of $U_n$. Other terms can be treated perturbatively. Using Lemma \ref{lem:pow_lead} and \eqref{eq:recur_topN_coe3}, one can show that 
\beq\label{eq:coe_grow_ansatz}
  - \f{a_{n, n-1}}{ a_{n, n}}  =  \f{n \kp}{ \kp - n} \cdot \f{ \D_{Y, 2} }{ \laml} 
+   O( |n-\kp|^{-1}) .
\eeq
We expect that the asymptotic growth rate of $U_n / U_{n-1}$ is given by the above rate and will justify it in \eqref{eq:induc_asym} in Lemma \ref{lem:asym} and Corollary \ref{cor:asym}, equation \eqref{eq:induc_asym2}.

We have the following estimate for binomial coefficients.

\begin{lem}\label{lem:binom}
For any $ n \geq 1$ and $ 1 \leq j \leq n /2$, we have $\binom{n}{j} \geq \f{ 4^j}{3 j^{1/2}}$.

\end{lem}

The proof is elementary and we refer to \cite[Lemma 5.5]{buckmaster2022smooth}.

We have the following estimates of the asymptotics of $\hat U_n$.
\footnote{In Lemma \ref{lem:asym} and its proof, we do not assume that $\kp \in (n, n+1)$.}

\begin{lem}\label{lem:asym}

Let $n_0$ be the parameter chosen in \eqref{eq:induc_para}. There exists $\kp_0$ large enough such that for any $\kp > \kp_0$ and $\kp \notin \mathbb{Z}$, the following statements hold true. 

\bseq\label{eq:lem_asym}
 For any $ 0\leq j \leq n <  \kp + 2$, we get 
\footnote{
 In the upper bound of \eqref{eq:induc_unif}, we use $|\hat U_n|$ instead of the more symmetric form $|\hat U_0 \hat U_n|$ since $\hat U_0$ can be $0$. In our case, we have $\hat U_0, U_0 \to 0$ \eqref{eq:sonic_Q} as $\g \to \ell^{-1/2}$.
}
\beq\label{eq:induc_unif}
|\hat U_j  \hat U_{n - j} | \leq \bar C_1  |  \hat U_n| ,
\quad \bar C_1 = 12.46.
\eeq

 For any $ n_0 \leq n < \kp$, we have 
\beq
U_n , \  \hat U_n> 0 . \label{eq:U_sign} 
\eeq

For any $ n_0 \leq n < \kp + 2$, we have 
\beq
  |\hat U_{n-1}|  < |\hat U_{n}| . \label{eq:induc_asym_b}
\eeq

Let $\hat \d = 0.049 $ be chosen in \eqref{eq:para3}, and denote 
\beq\label{eq:C_asym}
C_* = \lim_{ \g \to  ( \ell^{-1/2} )^+}  \f{\D_{Y, 2}}{ \laml} .
\eeq

 For any $n$ with $ n_0 \leq n < \kp + 2$, we get
\begin{align}
  \B| \hat U_{n } -   \f{ C_*}{4} \cdot \f{ n \kp}{ \kp - n } \hat U_{n-1} \B| 
 & \leq \hat \d \B|    \f{ C_*}{4} \cdot \f{ n \kp}{ \kp - n } \hat U_{n-1}  \B| , 
\label{eq:induc_asym} 
\end{align}
\eseq
\end{lem}

In \eqref{eq:induc_unif}, \eqref{eq:induc_asym_b}, \eqref{eq:induc_asym}, we consider the range of $n$: $ n < \kp +2$ larger than $n < \kp$ for later estimates in Section \ref{sec:shoot}. From \eqref{eq:induc_asym}, for $n > \kp$, $U_n$ changes sign, and \eqref{eq:U_sign} does not hold.

\vspace{0.1in}
\paragraph{\bf{Ideas and strategy}}

We expect that the asymptotics of $\hat U_n$ is given by \eqref{eq:induc_asym}, which quantifies the growth rate \eqref{eq:coe_grow_ansatz} for $U_n / U_{n-1}$. Since $|U_i|$ grows very fast in $i$, we can treat the remaining terms $U_i, i\leq n-2$ in \eqref{eq:recur_top}, \eqref{eq:recur_topN} perturbatively.

Recall $d_Y, d_U$ from \eqref{eq:poly_degUY}. To prove the induction, we introduce a few parameters 
\beq\label{eq:induc_para}
l_0 = \max(d_Y, d_U) = 2, \quad 
n_0 = 20, \quad  j_0 = 25, \quad  N = 30,  \quad  n_1 = 450 .
\eeq

 We require $n \geq n_0$ in \eqref{eq:U_sign}-\eqref{eq:induc_asym}  with $n_0$ not too small, since these estimates may not hold true for small $n$. We list a few conditions satisfying by the parameters \eqref{eq:induc_para} in the following lemma.
We verify them with computer assistance and will use them to prove Lemma \ref{lem:asym}. 
See more discussions on the computer assistance in Appendix \ref{app:comp}.

\begin{lem}[Computer-assisted]\label{lem:asym_claim}
Let $l_0, j_0, n_0, n_1$ be the parameters in \eqref{eq:induc_small}, $C_*$ in \eqref{eq:C_asym},
and $\bar C_1$ in \eqref{eq:induc_unif}. For $\g = \ell^{-1/2}$ (thus $\kp =\infty$ \eqref{eq:grad_smooth}), the following statements hold true.

(a) Inequality \eqref{eq:induc_unif} holds true for any 
$ 0 \leq n \leq n_1$, 
\eqref{eq:U_sign},\eqref{eq:induc_asym_b},\eqref{eq:induc_asym} hold true for any $ n_0\leq n \leq n_1$, and 
$\bar C_1 > \max(1, |\hat U_0|)$. 

(b) For any $j, l$ with $1 \leq j \leq j_0, 0\leq l \leq l_0$, we have
\beq\label{eq:induc_small}
  ( \f{C_*}{4}(1 -\hat  \d )  )^{j-1} (n_1 + 1 - N - j_0)^{j-1}  \geq  2^{j-1}  \frac{ | \hat U_{j+l-1} | }{ \max( |\hat U_l|, 1) }  .
\eeq

(c) The parameter $M_1$ defined below satisfies 
\beq\label{comp:induc_1}
M_1 \teq \sup_{ l \leq l_0}  \f{ |\hat U_{n_0}|}{ n_0!  \max( |\hat U_l|, 1) } (   \f{ C_*}{4} )^{l - n_0 } ,
\quad \B( \f{ 1}{j_0 + l_0 } \B)^{l_0}  \cdot 
 \f{1}{M_1}  
\cdot  \f{ ( \f{9}{5})^{j_0}}{ 3 j_0^{1/2 } }  > 1.
\eeq
\end{lem}

Let us motivate the choice of the parameters in \eqref{eq:induc_para}. First, we fix $l_0$ as in \eqref{eq:induc_para}. Next, we choose $n_0$ large enough so that conditions \eqref{eq:induc_asym}, \eqref{eq:induc_asym_b}, and \eqref{eq:U_sign} begin to hold. We then choose $j_0$ large enough to verify \eqref{comp:induc_1}. With $(j_0, l_0)$ fixed, \eqref{eq:induc_small} holds for any 
$n_1, N$ with $n_1 - N \geq n^*(j_0, l_0)$. We choose $N$ large and then take $n_1 = n_1(N) \geq n^*(j_0, l_0) + N$ to be sufficiently large. See the discussion below \eqref{eq:I_ijl_est}.

Using \eqref{eq:recur_top} and \eqref{eq:lam_limit}, we can obtain that $U_n$ is continuous in $\g$ as $\g \to \ell^{-1/2}$ for any $n \leq n_1$.  With Lemma \ref{lem:asym_claim} and continuity of $U_n$ in $\g$, 
 Lemma \ref{lem:asym} holds for all $n \leq n_1$ and $\kp$ large enough (equivalent to $\g - \ell^{-1/2}$ small). Below, we assume that the inductive hypothesis is true for the case of $ \leq n-1$ and we will prove the case of $n$ with $ n > n_1 > n_0 + 2 N $. In Section \ref{sec:induc_conseq}, we derive the consequence of the inductive hypothesis. In Section \ref{sec:induc_del}, we estimate the coefficients of the power series of $\D_{U}, \D_Y$. In Section \ref{subsec:induc_pf}, we estimate $\hat U_{n}$ and prove the induction.

\begin{remark}

The proof presented below does not take advantage of the specific forms of $\D_Y, \D_U$,
which are polynomials in $U$ with low degree $(\leq 2)$. We can prove it by checking the desired properties of $U_i, i \leq n_1$ with $n_1$ large enough with computer assistance and then use induction to prove the properties of $U_i$ with large enough $i$.

\end{remark}

\vspace{0.1in}
\paragraph{\bf{Constants in the estimates}}
Recall the constants $e_{l}$ from \eqref{eq:recur_topN_coe3}. 
In later proof, we will use the following constants, which have size of $O(1)$ compared to $\kp$ and $n$ \bseq\label{eq:induc_C}
\begin{align}
    b_{2, l} &= \bar C_1^{l} \max( |\hat U_1|,1), \ l \geq 0,  \quad b_{2, -1} =  1 ,
    \quad  b_3   = \max_{ i\leq N, 0\leq l \leq l_0}  \B| \f{\hat U_{i}}{\hat U_{i+l}} \B| , \label{eq:induc_C_b12}       \\
   q_n  &= \f{1}{4} (n - N) (1 - \hat \d)   C_* . \label{eq:induc_C_q} 
\end{align}
\eseq

We use $ b_{2, l}$ to denote the constants in the estimates of $U_{\mw{TR}, j}^l$ 
in \eqref{eq:power_Ul1} for $ l \geq 0$. The special value $b_{ 2,-1}$ is used to denote the constant in the exceptional case. The parameter $b_3$ is used to denote the constant in \eqref{eq:I_ijl_est1}. We use $q_n$ to denote the decay rate in the estimates of $\hat U_j, U_{\mw{TR, j}}^l$ in \eqref{eq:power_Ul2}.

\subsection{Consequence of the inductive hypothesis}\label{sec:induc_conseq}

Suppose that the inductive hypothesis holds true for the case of $\leq n - 1$ and 
$\kp + 2> n \geq n_1$. We show that for $(l, j, m)$ with 
\beq\label{eq:induc_decay_ass}
0\leq l \leq l_0 < j_0,
 \quad 1 \leq j, \quad j + l \leq m , \quad   n_1 - N \leq m \leq n < \kp + 2,
 \eeq
where $l_0, j_0, n_1$ are chosen in \eqref{eq:induc_para}, we have the inequality
\footnote{
We use $\max( |\hat U_l | , 1) |\hat U_{m-1}|$ instead of $|\hat U_l \hat U_{m-1}|$  in the upper bound since $\hat U_l$ would be $0$.
}
\bseq\label{eq:induc_decay}
\begin{align}\label{eq:induc_decay_a}
 |\hat U_{j + l - 1} \hat U_{m  - j} | & \leq  \max( |\hat U_l|, 1) |\hat U_{m -1 } | 2^{-\min(j-1, m-j-l)} . 
\end{align}
The estimate is trivial if $m - j = m-1$ or $j+l-1 = m-1$. Due to \eqref{eq:induc_decay_ass}, it remains to consider $m-j, j+l-1 \leq m-2 < \kp$.  Since \eqref{eq:induc_decay_a} is symmetric in $j+l-1$ and $m-j$,  we may without loss of generality assume 
\beq\label{eq:induc_decayb}
j+l-1 \leq m -j\ \leq m- 2. 
\eeq

Since $0 \leq l \leq l_0$,  $ n_0\leq  m < \kp + 2$, from the choices of $l_0, n_0$ in \eqref{eq:induc_para}, we get 
\beq\label{eq:range_j}
 j \leq  (m+1)/2 < \kp - 2.
\eeq
\eseq

To prove \eqref{eq:induc_decay}, we bound $  | \f{ \hat U_{m-1}}{\hat U_{m - j}} | $ from below 
and $  \f{ | \hat U_{j+l-1} | }{ \max( |\hat U_l | ,1 ) } $ from above.  Since $m \geq n_1-N \geq 2  n_0 + 2$ due to \eqref{eq:induc_para} and $1 \leq j \leq (m+1)/ 2$, using \eqref{eq:induc_asym} repeatedly, we get
\beq\label{eq:induc_rate_low}
\bal
\B| \f{ \hat U_{m-1}}{\hat U_{m-j}} \B| \geq 
  (    \f{ C_*(1 -  \hat \d)}{4} )^{j-1} \prod_{ m+1 - j\leq i\leq m -1} 
  \B|\f{i \kp}{\kp - i} \B|
  \geq 
    (    \f{ C_*(1 -  \hat \d)}{4} )^{j-1} 
  \f{ (m-1)!}{(m-j)!}  \prod_{ m+1 - j\leq i\leq m -2} 
  \B|\f{ \kp}{\kp - i} \B| ,
\eal
\eeq
where in the last inequality, we have used $|\f{\kp}{\kp - (m-1)}| > 1$ since 
\[
|\kp - (m-1)| <  \max(\kp - (m-1), m-1- \kp) < \kp  
\]
for $n_1 - N \leq m < \kp + 2$.  Since we consider $m \leq n < \kp + 2$, $\kp - (m-1)$ can be negative.

Recall $j_0, l_0, n_0$ from \eqref{eq:induc_para} and $M_1$ from \eqref{comp:induc_1}: 
\beq\label{eq:induc_M1}
M_1 = \sup_{ l \leq l_0}  \f{ |\hat U_{n_0}|}{ n_0!  \max( |\hat U_l|, 1) } (   \f{ C_*}{4} )^{l - n_0 } .
\eeq

The parameter $M_1$ can be seen as the accumulated error between the actual ratio 
$|\hat U_{n_0}  | / | \hat U_l |$ and the desired asymptotics in \eqref{eq:induc_asym} with $\kp$ much larger than $n_0, l$. 
Next, we estimate $  \f{ | \hat U_{j+l-1} | }{ \max( |\hat U_l|, 1) } $ from above and consider $j \geq j_0 + 1$ and $j \leq j_0$. 

\vspace{0.1in}
\paragraph{\bf{Case 1: $j \geq j_0 + 1$}}
Since $j_0 \geq n_0 \geq l_0 \geq l$ due to \eqref{eq:induc_para} and \eqref{eq:induc_decay_ass}, we get 
\[
\bal
 \f{ | \hat U_{j+l-1} | }{  \max( |\hat U_l|, 1) } 
 =  \B| \f{ \hat U_{j+l-1} }{ \hat U_{n_0}}  \cdot \f{\hat U_{n_0}}{  \max( |\hat U_l|, 1)} \B| 
\leq 
(1 + \hat \d )^{j + l-1- n_0} (  \f{ C_* }{4}  )^{ j -1}
\prod_{ n_0 +1 \leq i \leq j + l -1 }  \f{ i \kp }{ \kp- i} 
\cdot 
 n_0! M_1 . \\
\eal
\]
Each product is non-negative since $j+l-1 \leq m-2 < \kp$ \eqref{eq:induc_decayb}, \eqref{eq:induc_decay_ass}. Moreover, since $l + 2\leq l_0 + 2 < n_0+1$ \eqref{eq:induc_para} and 
$\f{\kp}{\kp -i} > 1$ for $i < \kp$, we obtain 
\[
 \f{ | \hat U_{j+l-1} | }{  \max( |\hat U_l|, 1) } 
  \leq ( \f{1}{4} (1 + \hat \d)  C_*  )^{ j -1}  
(j + l-1)!  \prod_{  l+ 2 \leq i \leq j + l -1 }  \f{ \kp }{ \kp- i} 
\cdot  M_1 . 
\]

Since $\f{\kp}{\kp-i}$ for $i\leq m- 2  < \kp$ \eqref{eq:induc_decay_ass}, \eqref{eq:induc_decayb} is increasing in $i$ and $j + l - 1\leq m - 2$, we yield 
\beq\label{eq:induc_rate_up}
 \f{ | \hat U_{j+l-1} | }{  \max( |\hat U_l|, 1) }  \leq ( \f{1}{4} (1 + \hat \d)   C_* )^{ j -1}  
(j + l-1)!  \prod_{ m+1 -j \leq i \leq m - 2 }  \f{ \kp }{ \kp- i} 
\cdot  M_1  .
\eeq

Combining the estimates \eqref{eq:induc_rate_low}, \eqref{eq:induc_rate_up}, setting
\[
I=\frac{ \max( |\hat U_l|, 1) \cdot |\hat U_{m -1 } |}{|\hat U_{j + l - 1} \hat U_{ m  - j} |} , 
\]
we have 
\[
I
\geq \f{1}{M_1} \f{(m-1)! } { (m-j)! (j+l-1) !} \f{(1 - \hat \d)^{j-1} }{ (1 + \hat \d )^{j-1} } 
= \f{1}{M_1} \binom{m-1}{j-1} \f{(1 - \hat \d)^{j-1} }{ (1 + \hat \d )^{j-1} }  \cdot 
\f{1}{ (j+l-1) (j+l-2) \cdots j} .
\]

We want to show that $I  \geq 2^{j-1}$. Since $l \leq l_0$ \eqref{eq:induc_decay_ass} and  $j-1 \geq j_0$ in this case, using $j+l_0-1 \leq ( j_0 + l_0)\frac{j-1}{j_0}$, we obtain 
\[
 (j+ l - 1) (j+l-2) \cdots j \leq 
(j + l_0 - 1)^{l_0}  \leq  \B( \f{ j_0 + l_0 }{j_0} \B)^{l_0} (j-1)^{l_0}\,.
\]

 In addition, since $ j-1 \leq \f{m-1}{2}$ \eqref{eq:induc_decayb}, Lemma \ref{lem:binom} implies
 \[\binom{m-1}{j-1} \geq  \f{4^{j-1}}{3 {(j-1)}^{1/2}} .\]

Combining the above estimates, we derive
\[\begin{aligned}
I &\geq 
\B( \f{ j_0}{j_0 + l_0 } \B)^{l_0} \f{1}{(j-1)^{ l_0}} \cdot 
 \f{1}{M_1} \cdot \f{4^{j-1}}{ 3(j-1)^{1/2}} \cdot  \f{(1 - \hat \d)^{j-1}}{ (1 + \hat \d )^{j-1} } 
 \\&= \B( \f{ j_0}{j_0 + l_0 } \B)^{l_0}  \cdot 
 \f{1}{M_1} \cdot \f{4^{j-1}}{ 3(j-1)^{1/2+ l_0}} \cdot  \f{(1 - \hat \d)^{j-1}}{ (1 + \hat \d )^{j-1} } .
\end{aligned}
\]

Due to the choices of $j_0, l_0$ in \eqref{eq:induc_para}, for any $j > j_0$, we have $j - 1 \geq j_0 >  2 l_0 + 1$. Since $\f{2 (1 - \hat \d)}{1 + \hat \d} > \f{9}{5}$ and 
 $ (\log \f{9}{5}) \cdot (j-1) \geq 2  (\log \f{9}{5}) (l_0 + \f{1}{2}) >l_0 + \f{1}{2}$, we obtain that $ \f{ ( \f{9}{5})^{j-1} }{ (j-1)^{1/2 + l_0}}$ is increasing in $j$. Thus, we obtain 
\beq\label{eq:induc_M12}
I\geq 2^{j-1} 
\B( \f{ j_0}{j_0 + l_0 } \B)^{l_0}  \cdot 
 \f{1}{M_1}  
\cdot  \f{ ( \f{9}{5})^{j_0}}{ 3 j_0^{1/2 + l_0} } 
= 2^{j-1} 
\B( \f{ 1}{j_0 + l_0 } \B)^{l_0}  \cdot 
 \f{1}{M_1}  
\cdot  \f{ ( \f{9}{5})^{j_0}}{ 3 j_0^{1/2 } }  .
\eeq

Recall $(l_0, j_0)$ from \eqref{eq:induc_para}. We use the condition \eqref{comp:induc_1} in Lemma \ref{lem:asym_claim} to obtain $I \geq 2^{j-1}$.

\vspace{0.1in}
\paragraph{\bf{Case 2: $j \leq j_0 $}} 
Applying  $|\f{\kp}{\kp-i}| > 1$ for $i < \kp + 2$, 
$ j \leq j_0$, and $ n_1- N \leq m$
to \eqref{eq:induc_rate_low}, we get
\[
 | \f{ \hat U_{m-1}}{\hat U_{m-j}} |  \geq 
  (  \f{1}{4}  C_*(1 -  \hat \d) )^{j-1}
 (m +1 - j_0)^{j-1} 
\geq 
  (  \f{1}{4}  C_*(1 -  \hat \d) )^{j-1}
(n_1 + 1 - N  - j_0)^{j-1} .
\]
Using the above estimate and \eqref{eq:induc_small}, we obtain \eqref{eq:induc_decay} 
\[
 \B| \f{ \hat U_{m-1}}{\hat U_{m-j}} \B|  \geq 2^{j-1}  \f{  | \hat U_{j+l-1} | }{ \max( |\hat U_l|, 1) }  .
\]

\subsection{Estimate of $\D_{\al, i}$}\label{sec:induc_del}

We will use \eqref{eq:recur_topN} to estimate the coefficients $\hat U_n$. Below, we estimate
$\D_{\al,i}(Y_0, U_0,.., U_{n- N}, 0.., 0) $. We introduce the truncated power series of $U$ up to $\xi^{n-N}$
\beq\label{eq:U_trunc}
U_{\mw{TR}} := \sum_{j\leq n - N} U_j \xi^j.
\eeq
Since we choose $N$ in \eqref{eq:induc_para} and fix $n$ in the following estimates, to simplify the notation, we drop the dependence of $U_{ \mw{TR}} $ on $n, N$. Note that the above power series involves $U_i$ rather than $\hat U_i$. We recall the relation between $U_i, \hat U_i$ from \eqref{eq:cat_num}.

Denote by $U_{ \mw{TR}, j}^l$ the coefficient of $\xi^j$ in the power series expansion of $(U_{ \mw{TR}} )^l$.  Next, we estimate $U_{ \mw{TR},j}^l$ for $l = 0, 1,..., \max(d_Y, d_U)$ \eqref{eq:poly_degUY} and $j \leq n $. 

For $l\geq 2$ and $j \leq n  $, we get 
\beq\label{eq:U_TR_j}
 U_{ \mw{TR},j}^l = \sum_{i_1 + i_2 + \cdots + i_l = j, \  i_s \leq n- N } U_{i_1} U_{i_2} \cdots U_{i_l} 
 = \sum_{i_1 + i_2 + \cdots + i_l = j, \ i_s \leq n- N } \hat U_{i_1} \hat U_{i_2} \cdots \hat U_{i_l} \mfr C_{i_1}\mfr C_{i_2} \cdots \mfr C_{i_l} .
\eeq

We use the inductive hypothesis to obtain two estimates of $\hat U_{i_1} \hat U_{i_2} \cdots \hat U_{i_l}$ with $l \geq 2$. 
The first estimate applies to $ n - N < j \leq n$ and the second applies to $j \leq n-1$.  

\vspace{0.1in}
\paragraph{\bf{First estimate with $n- N< j \leq n $.}} 
Without loss of generality, we assume $i_1 \leq i_2 \leq \cdots \leq i_l$. Since $i_l \geq 1, i_1 + i_2 + \cdots + i_l = j$ and $j \leq n$, we get $i_1, \cdots, i_{l-1} \leq n-1$. Using \eqref{eq:induc_unif} repeatedly, we get
\beq\label{eq:U_TR_est1}
\bal
& |\hat U_{i_1} \hat U_{i_2} \cdots \hat U_{i_{l-1}} |
\leq \bar C_1^{l-2} |  \hat U_{ i_1 + i_2 + \cdots + i_{l-1}} | . \\
\eal
\eeq

Due to $i_l = \max_{k< l} i_k$, the constraint on $i_k$ \eqref{eq:U_TR_j}, and $n \geq n_1 > N (1 + l_0)$ by the choices of $n_1, l_0, N$ \eqref{eq:induc_para}, we obtain that $i_k$ satisfies 
\beq\label{eq:U_TR_est2}
 n- N \geq i_l\geq \f{n-N}{l} \geq \f{n-N}{l_0} > N, 
 \quad i_1 + .. + i_{l-1} = j- i_l \leq n - N.
\eeq
It is easy to verify that the conditions \eqref{eq:induc_decay_ass} hold for 
$(l,j, m) \to (1, i_1 + i_2 + \cdots + i_{l-1},  j ) $. Thus, using \eqref{eq:induc_decay} with $(l,j, m) \to (1, i_1 + i_2 + \cdots + i_{l-1},  j ) $ and \eqref{eq:U_TR_est1}, we yield
\[
\begin{aligned}
 |\hat U_{i_1} \hat U_{i_2} \cdots \hat U_{i_{l-1}} \hat U_{i_l}| &\leq 
 \bar C_1^{l-2} |  \hat U_{ i_1 + i_2 + \cdots + i_{l-1}} 
 \hat U_{i_l} |  \\& \leq \bar C_1^{l- 2}  
 \max( | \hat U_1 |,1) |\hat U_{j-1}| \max( 2^{- (i_l-1) }, 2^{- (j -i_l-1)} ) .
 \end{aligned}
\]

Since $j$ satisfies $j-1 \geq n -N \geq n_1 - N > n_0$ due to \eqref{eq:induc_para}, we use 
\eqref{eq:induc_asym_b} to obtain $|\hat U_{j-1}| \leq | \hat U_{n-1}|$. From \eqref{eq:U_TR_est2}, we have 
\[
 i_l - 1 \geq N-1 \geq N -1 + j - n,
 \quad j - i_l - 1 \geq j - (n-N) - 1 = N - 1 + j- n,
\]

Thus the above estimate imply 
\bseq\label{eq:U_TR_est3}
\beq
 |\hat U_{i_1} \hat U_{i_2} \cdots \hat U_{i_{l-1}} \hat U_{i_l}| \leq \bar C_1^{l- 2} 
 2^{-( j + N - n - 1)}  \max( | \hat U_1 |,1) |\hat U_{n-1}| 
 = 2^{- (j + N - n -1)} b_{ 2, l-2} |\hat U_{n-1}| ,
\eeq
where $b_{2,l}$ is defined in \eqref{eq:induc_C}. 

\vspace{0.1in}
\paragraph{\bf{Second estimate with $ j \leq n-1$.}} 
Using \eqref{eq:induc_unif} and $i_1 + i_2 + \cdots + i_l = j \leq n-1$, we get another estimate
\beq
 |\hat U_{i_1} \hat U_{i_2} \cdots \hat U_{i_{l-1}} \hat U_{i_l} | 
 \leq \bar C_1^{l - 1} |  \hat U_{j} |  .
\eeq

Using the identity for the Catalan number \eqref{eq:cat_num_b} repeatedly, we get 
\beq
\sum_{i_1 + i_2 + \cdots + i_l = j, i_s \leq n - N} \mfr C_{i_1}\mfr C_{i_2} \cdots \mfr C_{i_l}
\leq 
\sum_{a + i_3 + \cdots + i_l = j + 1, i_s \leq n - N}
\mfr C_{a} \mfr C_{i_3} \cdots \mfr C_{i_l} 
\leq ... \leq \mfr C_{j+l-1}.
\eeq
\eseq

Plugging the above estimates \eqref{eq:U_TR_est3} in \eqref{eq:U_TR_j}, for $l \geq 2$, we obtain 
\bseq\label{eq:power_Ul1}
\begin{align}
| U_{\mw{TR},j}^l | & \leq \bar C_1^{l-1}  | \hat U_{j} | \mfr C_{j+l-1}, 
&& j \leq n-1, \label{eq:power_Ul1a} \\
| U_{ \mw{TR},j}^l | & \leq   b_{ 2, l-2}  2^{- (j+N-n-1)} |  \hat U_{n-1} | \mfr C_{j+l-1}, 
  && n - N <  j \leq n \label{eq:power_Ul1b}.
\end{align}
\eseq

  For $l = 1$, we just use the definition \eqref{eq:U_trunc} and \eqref{eq:cat_num} to obtain the $j-$th coefficient of $U_{ \mw{TR} } $: $U_{ \mw{TR},j}^l = U_{\mw{TR},j} = U_j \one_{j\leq n -N}
= \mfr C_j  \hat U_j \one_{j\leq n -N}$. Since we define $ b_{2, -1} = 1$ \eqref{eq:induc_C_b12}, 
 the above estimates hold trivially for $l = 1$.

 \begin{remark}
 The estimate \eqref{eq:power_Ul1b} applies to $j=n$. 
 We will \textit{only} use this estimate in Section \ref{sec:induc_est_J23} to bound $J_2, J_3$ \eqref{eq:induc_RHS_J}, which involves $U_{ \mw{TR},n}^l$.

 \end{remark}

For $n - N < j \leq n - 1$ and any $n > n_1$, using \eqref{eq:induc_asym}, we get 
\beq\label{eq:power_topN2}
| \hat U_j | <   ( \f{1}{4} (n - N) (1 - \hat \d )  C_* )^{-( n - 1 - j)} | \hat U_{n - 1} |
= q_n^{- (n - 1 - j)} | \hat U_{n - 1} |
< q_{n_1}^{- (n - 1 - j)} | \hat U_{n - 1} |,
\eeq
where we have used $q_n > q_{n_1}$ for $n > n_1$ in the last inequality since $q_n$ defined in \eqref{eq:induc_C} is increasing in $n$.  Moreover, $q_n$ is of order $n$. Combining \eqref{eq:power_topN2} and \eqref{eq:power_Ul1a}, we obtain another estimate of $|U_{ \mw{TR},j}^l |$, which along with \eqref{eq:power_Ul1b} imply
\beq\label{eq:power_Ul2}
|U_{ \mw{TR},j}^l | < \mfr C_{j+l-1} |\hat U_{n-1}|
\min\B(  \bar C_1^{l-1}
 q_{n_1}^{-( n-1-j)},  \  
b_{2, l-2}  2^{- (j+N-n-1)} \B) ,  
\eeq
for $n - N < j \leq n - 1$.

\subsection{Proof of the induction}\label{sec:induc_pf}

We decompose the right side (RS) of \eqref{eq:recur_topN} into three parts as follows: the first term, the summation over $i$ with $ N+1 \leq i \leq n-1$ and with $i = n$
\beq\label{eq:induc_RHS_J}
\bal
\mw{RS} & = J_3 - J_1 - J_2 ,  \\ 
J_1 &= 
  \sum_{ N+1 \leq i \leq n-1} (n+1 - i)  U_{n+1 - i} \D_{Y,  i}( Y_0, U_0, ...,  U_{n- N}, 0,0 ) , \\
\quad 
J_2 &=   U_1 \D_{Y,n}(Y_0, U_0, ...,  U_{n- N}, 0,0 ) , \\ 
J_3 &=  \D_{U,n}(Y_0, U_0, ...,  U_{n- N}, 0,0 ) .
\eal
\eeq

We estimate $J_i$ in Sections \ref{sec:induc_est_J1}, \ref{sec:induc_est_J23}, and the lower order terms on the left side  of \eqref{eq:recur_topN} in Section \ref{sec:induc_est_err}. Our goal is to show that they are small compared to $n |\hat U_{n-1} \mfr C_{n-1}|$, e.g.,
\[
|J_i| \ll n |\hat U_{n-1} \mfr C_{n-1}|.
\]
 We combine these estimates to prove the induction in Section \ref{subsec:induc_pf}.

\subsubsection{Estimate of $J_1$}\label{sec:induc_est_J1}

Recall $\D_Y(Y, U_{ \mw{TR} } ) = \sum_{l \leq d_Y } G_l(Y) (U_{ \mw{TR} })^l $ from \eqref{eq:ODE_UY}. Expanding the coefficients of $\D_{Y,i}$ using \eqref{eq:power_prod}, we get
\beq\label{eq:induc_J1_est1}
\bal
|J_1| & \leq 
\B| \sum_{N + 1 \leq i \leq n-1} 
  \sum_{ 0 \leq l \leq d_Y}  \sum_{j \leq \min(i, \deg G_l )}
(n+1 - i)  U_{n+1 - i} ( G_l)_j U_{ \mw{TR}, i -j}^l \B|, 
\eal
\eeq
where $a_i$ denotes the $i$-th order coefficient of in the power series of $a$. For $l=0$, since $(U_{ \mw{TR} })^0 = 1$ and $N + 1 > d_G \geq \deg G_l$ from the definitions \eqref{eq:induc_para} of $N$ 
and \eqref{eq:poly_deg} for $d_G$, we get 
\[
(G_l)_j U_{ \mw{TR}, i-j}^l  \one_{j \leq \deg G_l ,  N + 1 \leq i}
= (G_l)_j \one_{i=j}\one_{j \leq \deg G_l,  N + 1 \leq i}
= (G_l)_i \one_{ N+1 \leq i \leq \deg  G_l } = 0.
\]

Therefore, we may restrict the summation to $l \geq 1$ in \eqref{eq:induc_J1_est1}. Next, we estimate
\beq\label{eq:I_ijl}
I_{i, j,  l} \teq |U_{n+1-i}  U_{ \mw{TR},i-j}^l |
= \mfr C_{n+1-i} |\hat U_{n+1-i}  U_{ \mw{TR},i-j}^l |
\eeq
and consider $(i, j)$ with $ n- N <  i - j \leq n-1 $ and $i-j \leq n- N, i \geq N+1$ separately. 
We have $i \geq N+1$ due to the constraint in \eqref{eq:induc_J1_est1}. Since we fix $n$ below, 
we drop the dependence on $n$.

\vspace{0.05in}
\paragraph{\bf{Case 1:  $n- N <  i - j \leq n-1 $ }}

In this case, using \eqref{eq:power_Ul2}, we get
\[
\bal
I_{i, j, l} &\leq
 \mfr C_{n + 1 -i} \mfr C_{i-j + l-1 }
 | \hat U_{ n + 1 - i } \hat U_{n-1}| \min( \bar C_1^{l-1} q_{n_1}^{-(n-1 - (i-j))}, b_{2, l-2} 2^{-( i - j + N - n -1)} ) ,\\
 \eal
\]
 In this case, since 
 $n+1 - i \leq N$ and $n+ 1  - (i-j) \leq N + l_0$ ($j \leq d_G \leq l_0$ \eqref{eq:induc_para}), we can bound 
 $|\hat U_{n+1 - i} / \hat U_{n+1 - i + j}|$ by some constant $b_3$ \eqref{eq:induc_C} independent of $n$ and obtain 
\[
\bal
I_{i, j, l} & \leq  b_3 | \hat U_{ n + 1 -(i-j) } | \mfr C_{n+1 -i} \mfr C_{i- j + l-1 }
 | \hat U_{n-1}| \min( \bar C_1^{l-1} q_{n_1}^{-(n-1 - (i-j))}, b_{2, l-2} 2^{-( i -j + N - n -1)} ) . \\
 \eal
\] 

Denote 
\[
 m = n-1 - (i- j) .
\]

From the bound of $i- j$, we get  $0\leq m \leq N-2$ and can further estimate $I_{i, j, l}$ as follows 
\beq\label{eq:I_ijl_est1}
\bal
I_{i,j,l}& =   b_3 | \hat U_{m + 2}|   \mfr C_{n+1 -i} \mfr C_{i-j + l-1 }
 | \hat U_{n-1}|   \min( \bar C_1^{l-1} q_{n_1}^{-m}, b_{2, l-2} q^{-( N -2-m)} )  \\
& \leq  
  b_3 \mfr C_{n + 1 -i} \mfr C_{i- j + l-1 }
 | \hat U_{n-1}|  \max_{ 0\leq m \leq N-2}  
\B( | \hat U_{m + 2}|  \min( \bar C_1^{l-1}  q_{n_1}^{-m}, b_{2, l-2} 2^{-( N -2-m)} )  \B) .
\eal
\eeq

The bound in the maximum is very small by choosing $N$ relatively large and then $n_1$ large.

\vspace{0.05in}
\paragraph{\bf{Case 2: $  i- j\leq n - N, i \geq N+1$.}}

In this case, since $i \geq N+1$, which implies $i-j > N - l_0 >0 $ \eqref{eq:induc_para},  applying \eqref{eq:power_Ul1} to $U_{ \mw{TR},i- j}^l$, we estimate $I_{i,j,l}$ \eqref{eq:I_ijl} as
\bseq\label{eq:I_ijl_case2}
\beq
I_{i,j,l} =\mfr C_{n+1-i} |\hat U_{n+1-i}  U_{ \mw{TR},i-j}^l |
\leq 
 \mfr C_{n + 1 -i} 
  \mfr C_{i- j + l-1} 
\bar C_1^{ l -1}
 | \hat U_{ n + 1 - i }
 \hat U_{ i - j} | .
\eeq

If $ j \leq 1$, using \eqref{eq:induc_decay} with $l = 2-j$, $i \geq N+1$, and $i- j \leq n -N$  again, we obtain
\beq
\bal
 | \hat U_{ n + 1 - i }
 \hat U_{ i - j } | 
 & \leq | \hat U_{n  - 1} | \max( | \hat U_{2- j} | , 1)
 2^{ - \min( n-1-(n+1 -i), n-1- (i- j) )} \\
 & \leq  | \hat U_{n  - 1}| \max( |\hat U_{2- j} | , 1)  2^{-(N- 1)} .
 \eal
\eeq

If $ 2\leq j \leq \deg G_l $, using \eqref{eq:induc_decay} with $l=0$, similarly, we get 
\[
\bal
 | \hat U_{ n + 1 - i }
 \hat U_{ i - j} | 
 & \leq |\hat U_{n+1 - j}| \max( |\hat U_{ 0} | , 1)  2^{ - \min(i-j, n+1 - i)} . 
\eal
\]
Since $\deg G_l \leq d_G$ (see \eqref{eq:poly_deg}), $i$ satisfies $   i - j \leq n- N$ and $i \geq N+1$, we obtain 
\[
i - j \geq N+ 1 - d_G,  \quad n+ 1 - i \geq N - j + 1 \geq N + 1 - d_G,
\quad \min(i- j, n+1 - i) \geq N +1 - d_G. 
\]
Since $n$ satisfies $n-1 \geq n+1 - j \geq n_1 + 1 -l_0 >  n_0$, we 
use \eqref{eq:induc_asym_b} to obtain  $|\hat U_{n+1- j} | \leq |\hat U_{n-1}|$. Using the above three estimates, we simplify the bound as
\beq
 | \hat U_{ n + 1 - i }
 \hat U_{ i - j} |  \leq | \hat U_{n -1}| \max(| \hat U_0|, 1)  2^{ - (N+ 1 - d_G)}  .
\eeq
\eseq

Combining the estimates in \eqref{eq:I_ijl_case2}, we establish 
\beq\label{eq:I_ijl_est2}
I_{i,j,l}  \leq 
2^{-N}  \max\B(  \max( |\hat U_0|, 1) 2^{ d_G- 1}, 2 \max_{m\leq 2} (\max( |\hat U_m| , 1) ) \B) 
\mfr C_{n+1 -i}  \mfr C_{i- j + l -1} | \hat U_{n-1} | .
\eeq

\paragraph{\bf{Summary}} Summing two estimates of $I_{ijl}$ \eqref{eq:I_ijl_est1}, \eqref{eq:I_ijl_est2} for two cases, we establish
\beq\label{eq:I_ijl_est}
 |I_{i,j,l}| \leq \IC_l  \mfr C_{n+1 -i}  \mfr C_{i-j + l -1} | \hat U_{n-1} | , \\
\eeq
for any $ l \geq 1$, where $\IC_l$ is defined as 
\beq\label{eq:induc_C_C}
\bal
 \IC_l & \teq  b_3  \max_{ 0\leq m \leq N-2}  
\B( | \hat U_{m + 2}|  \min( \bar C_1^{l -1} q_{n_1}^{-m}, b_{2, l -2} 2^{-( N -2-m)} )  \B)  \\
 & \quad  + 2^{-N}  \max\B(  \max( |\hat U_0|,1) 2^{ d_G- 1}, 2 \max_{m\leq 2} ( \max( |\hat U_m| ,1) ) \B)  ,
 \eal
 \eeq
 and $b_{2, l -2}, q_{n_1}$ are defined in \eqref{eq:induc_C}. The first term in $\IC_l$ corresponds to the bound \eqref{eq:I_ijl_est1} for Case 1, and the second term for the bound \eqref{eq:I_ijl_est2} in Case 2.

Since $q_{n_1}$ defined in \eqref{eq:induc_C_q} is of order $n_1$, by first choosing $N$ large so that the second part in $\IC_l$ is small and then $n_1$ large 
\footnote{
 Choosing $n_1$ large means that we verify the estimates in Lemma \ref{lem:asym} with computer assistance up to the case $n \leq n_1$ with large $n_1$. 
}
so that the first part is small, we can make $\IC_l$ very small.  Therefore, combining the above estimates, we can estimate $J_1$ as follows 
\[
|J_1 |  \leq n \sum_{N+ 1 \leq i\leq n-1} \sum_{1\leq l \leq d_Y, j \leq \deg G_l } | (G_l)_j |
\IC_l \mfr C_{n + 1 -i}  \mfr C_{i-j + l-1}  |\hat U_{n-1} | .
\]
Summing over $i$ using \eqref{eq:cat_num_b} for $\mfr C_i$, we yield 
\bseq\label{eq:induc_est_J1}
\beq
|J_1|  \leq  \sum_{1\leq l \leq d_Y, j \leq \deg G_l } n | (G_l)_j |  \IC_l  |\hat U_{n-1}| \mfr C_{n - j + l + 1}   \leq n |\hat U_{n-1}| \mfr C_{n + d_Y + 1} 
\sum_{ 1\leq l \leq d_Y, j \leq \deg G_l } | (G_l)_j |  \IC_l   \\
\eeq
Using $\mfr C_{i+1} \leq 4 \mfr C_i$ \eqref{eq:cat_num_c} for any $i \geq 0$, we further bound 
$ \mfr C_{n + d_Y+1} \leq 4^{d_Y+2} \mfr C_{n-1}$ and obtain 
\beq
| J_1| \leq n |\hat U_{n-1}|  \mfr C_{n -1} \cdot 4^{d_Y+2}
\sum_{ 1\leq l \leq d_Y, j \leq \deg G_l } | (G_l)_j |  \IC_l  
\teq n |\hat U_{n-1}| \mfr C_{n-1} \cdot C_{J_1}
\eeq
\eseq
Since $\IC_l$ can be made sufficiently small by choosing $N, n_1$ in order, $C_{J_1}$ can be made very small.

\subsubsection{Estimate of $J_2, J_3$}\label{sec:induc_est_J23}

Recall $J_2, J_3$ from \eqref{eq:induc_RHS_J}. We first estimate $J_2$ .
Using the expansion of $\D_Y(Y, U_{ \mw{TR} })$ from \eqref{eq:ODE_UY}, we get 
\beq\label{eq:J2}
|J_2| = | U_1 \D_{Y, n}|
= |  U_1 \sum_{ 0\leq l \leq d_Y, j \leq \deg G_l  } (G_l)_j U^l_{ \mw{TR}, n - j}| .
\eeq
For $l=0$, since $U_{ \mw{TR} }^0= 1$ and $n$ satisfies $n\geq n_1 >  \max \deg G_i$ \eqref{eq:induc_para} and \eqref{eq:poly_deg}, we obtain
\[
  (G_0)_j (U_{ \mw{TR} })^0_{n-j} = (G_0)_j \one_{j=n} = (G_0)_n =0 .
\]
Thus, we may restrict the summation to $l\geq 1$. Applying \eqref{eq:power_Ul1b} to $U_{ \mw{TR},n-j}^l$, we obtain 
\bseq\label{eq:induc_est_J2}
\beq
\bal
|J_2| &  \leq 
\B|  U_1 \hat U_{n-1}  \sum_{ 1 \leq  l \leq d_Y, j \leq \deg G_l } 
(G_l)_j
b_{2, l-2} 2^{- ( n - j + N - n -1 )} \mfr C_{n-j + l-1} \B| \\
& \leq 
|\hat U_{n-1}| \mfr C_{n + d_Y - 1} 2^{-N}
| U_1 | \sum_{ 1 \leq l \leq d_Y, j \leq \deg G_l }  | (G_l)_j | b_{ 2, l- 2} 2^{j+1}  .\\
\eal
\eeq
Since $\mfr C_{i+1} / \mfr C_i \leq 4$ for any $ i \geq 0$ (see \eqref{eq:cat_num_b}),  
we further bound $\mfr C_{n + d_Y - 1}  \leq \mfr C_{n  - 1} 4^{d_Y}$ to establish 
\beq
|J_2| 
\leq  C_{J_2}  |\hat U_{n-1} \mfr C_{n-1}|, \quad 
C_{J_2 } =   2^{-N + 2 d_G}
| U_1 | \sum_{ 1 \leq l \leq d_Y, j \leq \deg G_l }  | (G_l)_j | b_{2, l - 2} 2^{j + 1} .
\eeq
\eseq

Similarly, we estimate $\D_{U,n}$ in \eqref{eq:recur_topN} using the expansion \eqref{eq:ODE_UY} as follows 
\bseq\label{eq:induc_est_J3}
\beq
\bal
 J_3 & = |\D_{U, n }(  U_0,  U_1, ..  U_{n-N}, 0, .., 0)|\\&
 \leq 
 |\hat U_{n-1}| \mfr C_{n + d_Y - 1} 2^{-N}
 \sum_{ 1 \leq l \leq d_U, j \leq \deg F_l }  | (F_l)_j | b_{2, l - 2} 2^{j +1}   \leq C_{J_3} |\hat U_{n-1} \mfr C_{n-1}| , \\
 \eal
 \eeq
 where $C_{J_3}$ is given by
 \beq
 C_{J_3} =  2^{-N +2 d_Y}
 \sum_{ 1 \leq l \leq d_U, j \leq \deg F_l }  | (F_l)_j | b_{ 2, l - 2} 2^{ j +1} .
\eeq
\eseq

\subsubsection{Estimate the LHS of \eqref{eq:recur_topN}}\label{sec:induc_est_err}
Below, we estimate the left side (LS) of \eqref{eq:recur_topN}. 
We decompose the left side as follows and treat the terms of $U_m, m \leq n-2$ perturbatively
\beq\label{eq:induc_LHS_err}
\mw{LS} = \sum_{ n -N < m \leq n} a_{n, m}  U_m
= a_{n, n}  U_n
+ a_{n, n-1}  U_{n-1} + \cE_n ,   
\quad \cE_n := \sum_{ n -N < m \leq n-2} a_{n, m}  U_m .
\eeq

Using the estimate \eqref{eq:power_topN2} for $\hat U_m$, we estimate the error term $\cE_n$ as follows 
\[
|\cE_n |
\leq \sum_{ n- N < m \leq n - 2} |a_{n, m}| \mfr C_{m} q_{n_1}^{-(n-1-m)}   |\hat U_{n-1}|.
\]
Below, we further estimate $a_{n, m}$ and $\mfr C_m$. For $n > n_1$,  using \eqref{eq:recur_topN_coe3} we get
\[
|a_{n, m}| = |e_{ n-m} + m \D_{Y, n-m + 1}|  \leq n  ( | e_{n-m} | n_1^{-1} + |\D_{Y, n-m + 1}| ) .
\]
Since $m$ satisfies $n > m $ and $m > n -N \geq n_1 - N  > 4$ \eqref{eq:induc_para},  using  \eqref{eq:cat_num}, we obtain $\mfr C_{m} |\hat U_{n-1}| < 3^{m+ 1-n } \mfr C_{n-1} |\hat U_{n-1}| = 3^{ m+ 1 - n} |U_{n-1}|$. Combining these estimates for $\cE_n, a_{n, m}$ and denoting $l = n-m$, we conclude 
\bseq\label{eq:induc_est_err}
\beq
\bal
|\cE_n | & \leq  n | U_{n-1}| \sum_{ 2\leq  n -m \leq N-2 }
( |e_{ n-m}|  n_1^{-1}+   |\D_{Y, n-m+1}|  ) (3q_{n_1} )^{- (n-m-1) }  .
 \\
\eal
\eeq
Using  $l = n-m$, we can rewrite the above summation and obtain the following estimates 
\beq
|\cE_n | \leq C_{\cE}(n_1) \cdot n | U_{n-1}|, \quad 
C_{\cE}(n_1) = \sum_{ 2\leq  l \leq N-2 }
( |e_{ l}|  n_1^{-1}+   |\D_{Y, l + 1}|  ) (3q_{n_1} )^{- (l-1) }  .
\eeq
\eseq
For fixed $N$, since $l-1 \geq 1$, $C_{\cE}$ is of order $n_1^{-1}$.

\subsubsection{Proof of the induction}\label{subsec:induc_pf}


Now, we are in a position to prove the induction and Lemma \ref{lem:asym}, which is based on the following lemma verified with computer assistance. 

\begin{lem}[Computer-assisted]\label{lem:induc_claim}
Recall $\hat \d = 0.049$ \eqref{eq:para3}. For the parameters chosen in \eqref{eq:induc_para} and constants $C_{J_i}, C_{\cE}$ defined in 
\eqref{eq:induc_est_J1}, \eqref{eq:induc_est_J2}, \eqref{eq:induc_est_J3}, \eqref{eq:induc_est_err} with $\g = \g_*=  \ell^{-1/2}$, we have 
\bseq\label{eq:induc_claim}
\begin{align}
 \B( n_1^{-1} (C_{J_2} + C_{J_3} + |e_{1 } | + | \D_{Y, 2} | ) +  C_{J_1} 
+  C_{\cE}(n_1) \B) \B|_{\g = \g_* }  &  < \hat \d , \\
  \max( |\hat U_1|, 1) &  < q_{n_1} , \\
 \f{3}{2 (4 n_1 -2)} \cdot C_* + \f{0.05}{\laml( \g_*)}   & <  \f{C_*}{4} \cdot \hat \d . 
\label{eq:induc_claim_c} 
\end{align}
\eseq
\end{lem}

\begin{proof}

\textbf{Estimate \eqref{eq:induc_asym}.}
Recall $a_{n, n-1} = e_1 + (n-1) \D_{Y, 2}$ \eqref{eq:recur_topN_coe3b}. 
Combining the estimates \eqref{eq:induc_RHS_J}, \eqref{eq:induc_est_J1}, \eqref{eq:induc_est_J2}, \eqref{eq:induc_est_J3} of $J_i$ for the right side of \eqref{eq:recur_topN} 
and the estimates \eqref{eq:induc_est_err},  \eqref{eq:induc_LHS_err}  for the left side of \eqref{eq:recur_topN}, we can estimate the relation between $\hat U_n, \hat U_{n-1}$ as follows 
\[
\bal
| a_{n, n}  U_n 
+  n \D_{Y, 2}  U_{n-1}| 
& \leq (|e_{1 } |+ | \D_{Y, 2}|)  U_{n-1} + | a_{n, n}  U_n 
+  a_{n,n-1}  U_{n-1}| \\
& =  (|e_{1 } |+ | \D_{Y, 2}|)  U_{n-1}  
+ | J_3 - J_1 - J_2  - \cE_n| \\
& \leq n | U_{n-1}| 
\B(  n_1^{-1} (C_{J_2} + C_{J_3} + |e_{1 } | + | \D_{Y, 2} | ) +  C_{J_1} 
+  C_{\cE}(n_1) \B) . \\
\eal
\]

Using continuity of the functions in $\g$ and Lemma \ref{lem:induc_claim}, for $\kp$ sufficiently large, we obtain 
\[
| a_{n,n} U_n 
+  n \D_{Y, 2}  U_{n-1}| < 0.05 n |U_{n-1}|.
\]

From Lemma \ref{lem:pow_lead} and \eqref{eq:kappa}, we obtain $a_{n, n} = n \lams - \laml = \f{n - \kp}{ \kp} \laml$. Using \eqref{eq:cat_num} and diving the above estimate by $|a_{n,n} \hat C_n \hat U_{n-1}| $, we compute 
\[
  \B| \f{\hat U_n}{\hat U_{n-1}} - \f{n \kp \D_{Y, 2}}{  (n - \kp) \laml }  \f{\mfr C_{n-1}}{\mfr C_n} \B| <  0.05   \f{\mfr C_{n-1}}{\mfr C_n} \f{ n \kp}{ (n - \kp) \laml } .
\]
 Using triangle inequality, we obtain
\[
  \B| \f{\hat U_n}{\hat U_{n-1}} - \f{C_*}{4} \f{n \kp  }{  (n - \kp) } \B|
 \leq C(\g)   \f{n \kp  }{  (n - \kp) } ,
 \quad 
 C(\g) = \f{\mfr C_{n-1}}{ \mfr C_n} |  \f{\D_{Y,2}}{  \laml} -  C_* |
 + | \f{\mfr C_{n-1}}{ \mfr C_n} - \f{1}{4}|  | \f{\D_{Y,2}}{  \laml} |
 + \f{0.05}{\laml}. 
\]

From \eqref{eq:cat_num}, we have $ | \f{\mfr C_{n-1}}{\mfr C_n} - \f{1}{4} | = | \f{n+1}{ 4 n -2} - \f{1}{4} | \leq \f{3}{2 (4 n -2)} $. For $ n \geq n_1$, 
using \eqref{eq:induc_claim_c}, the limit \eqref{eq:C_asym} and continuity of $\laml(\g)$ in $\g$, we obtain 
\[
\limsup_{\g \to \g_*} C(\g) \leq 
 \f{3}{2 (4 n_1 -2)} \cdot C_* + \f{0.05}{\laml( \g_*)}    < \f{C_*}{4} \hat \d . 
\]
Thus, for $\kp$ sufficiently large (equivalent to $\g $ close to $\g_*$), we establish  $C(\g) <\f{C_*}{4} \hat \d $ and \eqref{eq:induc_asym}.

\textbf{Estimates \eqref{eq:U_sign}.} The sign of $\hat U_n$ follows from 
$\hat U_{n-1} > 0$ in \eqref{eq:U_sign} and the relation \eqref{eq:induc_asym}.

\textbf{Estimates \eqref{eq:induc_unif} and \eqref{eq:induc_asym_b}}.
Using the new asymptotics for $\hat U_n$ and $q_n$ defined in \eqref{eq:induc_C_q} and $q_n > q_{n_1}$ for $n > n_1$, we obtain $|\hat U_{n-1} | < q_n^{-1} |\hat U_n| \leq q_{n_1}^{-1} |\hat U_n| $. Since $q_{n_1} > 1$ from \eqref{eq:induc_claim}, we verify 
$|\hat U_{n-1} | <  |\hat U_n|$ in \eqref{eq:induc_asym_b}. Combining this estimate, \eqref{eq:induc_decay} with $l = 1$, and using $q_{n_1} > \max( |\hat U_1|, 1)$ from Lemma \ref{lem:induc_claim}, we establish
\[
   |\hat U_j \hat U_{n-j}| \leq \max( |\hat U_1|, 1) |  \hat U_{n-1}| 
   < q_{n_1} |  \hat U_{n-1}|  < |\hat U_n|. 
\]
for any $1\leq j\leq n-1$. Since $\bar C_1 \geq \max( |\hat U_0|, 1)$ from item (a) in Lemma \ref{lem:asym_claim}, we verify \eqref{eq:induc_unif}.

We prove all the estimates for $n$ in the induction and thus establish Lemma \ref{lem:asym}. 
\end{proof}

\subsection{Convergence of power series}

In this section, we establish the following  estimates of the power series coefficients $\hat U_n$ and establish the convergence of the power series. 

\begin{prop}\label{prop:analy}
Recall the definitions of $\g_m, C_{\kp}, \G$ from \eqref{eq:kappa_bi}, \eqref{eq:kappa_n}. Let $\al = ( 4 \max( d_U, d_Y) )^{-1} $ with $d_U, d_Y$ defined in \eqref{eq:poly_degUY}, $m > C_{\kp}$ be a positive integer, and $I \in (\g_{m+1}, \g_m)$ be a closed interval. There exists a constant $D$ depending on $I$ such that the renormalized power series coefficient \eqref{eq:cat_num} with parameter $\g$ satisfies 
\beq\label{eq:local_Ui}
| \hat U_n(\g) | \leq D^{\max(\al, n -2)}.
\eeq
In particular, for any $\kp \in \kp(I) \subset (m, m + 1 )$, the solution $U^\rn{\kp}(Y) = \sum_{n \geq 0} U_n(\G(\kp)) Y^n$ with $\g = \G(\kp)$ defined in \eqref{eq:kappa_bi} and $U_n = \mfr C_n \hat U_n$ is analytic in $ |Y| < \f{1}{4 D}$. Moreover, $U^\rn{\kp}(Y)$ is continuous in $\kp \in \kp(I) \subset (m, m + 1 )$ and $|Y|\leq (8 D)^{-1}$. 

\end{prop}

Here, we use the notation $U^\rn{\kp}$ rather than $U^\rn{\g}$ to indicate the dependence on the parameter $\g$ (or equivalently $\kp$). This convention will simplify our notations in Sections \ref{sec:shoot}, \ref{sec:Q_lower} when referring to this local analytic solution.

\begin{proof}

Firstly, for any $\g \in I$, using \eqref{eq:recur_top} and \eqref{eq:kappa}, we obtain 
\[
a_{n,n} = n \lams - \laml = ( \f{n}{\kp} -1 ) \laml = \f{n - \kp}{\kp} \laml. 
\]
For $\kp = \kp(\g) $ with $\g \in I \subset (\g_{m+1}, \g_m)$, 
from the definition of $\g_{\cdot}$ \eqref{eq:kappa_n},  we obtain $\kp \in J \subset (m, m+1)$ for a close interval $J = \kp(I)$. 
Thus, we obtain $|n - \kp| \geq c_I n$ with some $c_I > 0$ uniformly for any $n \geq 0$ and $\g \in I$. 
From \eqref{eq:lam_sign}, \eqref{eq:grad_lam_two}, we obtain $ \laml \geq 
\f{1}{2}(\lams + \lam-) = \f{1}{2}(c_1 + c_4) \gtr 1$. Combining these estimates, we obtain 
\beq\label{eq:an_low}
|a_{n, n}| \geq c_I \max(n, 1) , 
\eeq
uniformly for $\g \in I$. In particular, $|a_{n,n}|^{-1}$ is bounded uniformly for $\g \in I$.

Next, we prove \eqref{eq:local_Ui} by induction. We fix an arbitrary $\g \in I$ and drop the dependence of $U_i(\g)$ on $\g$ to simplify the notation. For $n_L > 2 \kp$ to be determined, we choose $D = D(n_L)$ large so that \eqref{eq:local_Ui} holds for $ n \leq n_L $. 
Suppose that \eqref{eq:local_Ui} holds for the case of $ \leq n -1 $ with $n \geq n_L$. 
We use \eqref{eq:recur_top} to bound $U_n$. We follow the arguments in Sections \ref{sec:induc_del}, \ref{sec:induc_pf} with $N = 1$.


Following \eqref{eq:U_trunc} with $N=1$, we introduce the truncated power series $U_{\mw{TR}} := \sum_{j\leq n - 1} U_j \xi^j.$ For $2 \leq l \leq \max(d_U, d_Y)$ and $j \leq n$,  the $j-th$ power series coefficient of $U_{\mw{TR}}^l$ is given by 
\eqref{eq:U_TR_j} 
\beq\label{eq:U_TR_j_N1}
 U_{ \mw{TR},j}^l 
 = \sum_{i_1 + i_2 + \cdots + i_l = j, \ i_s \leq j - 1 } \hat U_{i_1} \hat U_{i_2} \cdots \hat U_{i_l} \mfr C_{i_1}\mfr C_{i_2} \cdots \mfr C_{i_l} .
 \eeq

Using the inductive hypothesis \eqref{eq:local_Ui}, we get 
\beq\label{eq:U_TR_N1_coe}
 |\hat U_{i_1} \hat U_{i_2} \cdots \hat U_{i_l}|
 \leq D^S, 
 \quad S = \sum_{s \leq l} \max(\al, i_s - 2). 
\eeq

Without loss of generality, we assume $i_1 ,.., i_k >  2$ and $i_{k+1} ,.., i_l \leq 2 $ 
for some $k$ with $0\leq k \leq l$. If $k = 0, 1$ and $j \geq 3$, using the definition of $\al$ in Proposition \ref{prop:analy} and $i_1 \leq j-1 \leq n-1$, we get 
\[
S = \max( \al , i_1-2) +  (l-1) \al 
\leq j - 3 + l \al \leq j - 2.
\]

If $k \geq 2$ and $j \geq 3$, we obtain 
\[
S \leq \sum_{s \leq k} (i_s - 2) + (l-k) \al
 \leq j - 2 k + l \al  < j - 2.
\]

If $j \leq 2$, we obtain $i_1, .., i_l \leq 1$ and 
\[
S \leq l \al \leq 1/4.
\]

Plugging the above estimates in \eqref{eq:U_TR_N1_coe}, applying the identities \eqref{eq:cat_num_b} repeatedly,  we obtain
\beq\label{eq:U_TR_N1_coe2}
 |U_{ \mw{TR},j}^l |
\leq D^{  \max( j- 2 , 1/4)} 
  \sum_{i_1 + i_2 + \cdots + i_l = j, \ i_s \leq j - 1 }  \mfr C_{i_1}\mfr C_{i_2} \cdots \mfr C_{i_l} 
\les   D^{  \max( j- 2 , 1/4)}   \mfr C_{ j + l-1}.
\eeq
for $l\geq 2$ and $j \leq n$. Since $U_{ \mw{TR},j}^1 =  U_j \one_{ j\leq n-1}$ and
$U_{ \mw{TR},j}^0 = \one_{j=0}$, \eqref{eq:U_TR_N1_coe2} also hold for $l=0, 1$. 

Next, we bound the right hand side of \eqref{eq:recur_top}. Following Sections \ref{sec:induc_pf}, we only need to bound $J_i$ in \eqref{eq:induc_RHS_J} with $N = 1$. For $J_1$ \eqref{eq:induc_RHS_J}, using the expansion \eqref{eq:induc_J1_est1} with $N=1$, we get
\[
\bal
|J_1| & \leq 
\B| \sum_{2 \leq i \leq n-1} 
  \sum_{ 0 \leq l \leq d_Y}  \sum_{j \leq \min(i, \deg G_l)}
(n+1 - i)  U_{n+1 - i} ( G_l)_j U_{ \mw{TR}, i -j}^l \B| .
\eal
\]
Applying \eqref{eq:U_TR_N1_coe2} for $U_{ \mw{TR}, i -j}^l$, 
inductive hypothesis \eqref{eq:local_Ui} for $U_{n+1 - i}=  \mfr C_{n+1-i} \hat U_{n+1-i}$, we obtain 
\[
|J_1|  \les n 
\sum_{2 \leq i \leq n-1} 
  \sum_{ 0 \leq l \leq d_Y}  \sum_{j \leq \min(i, \deg G_l)}
  \mfr C_{n+1 - i} \mfr C_{i - j + l -1}
  D^{\max( n+1-i-2, \f{1}{4} ) + \max( i -j-2, \f{1}{4} )}
\]

Since $n +1 - i \leq n-1, i-j \leq n-1$, and $n$ is large, it is easy to get that the exponent is bounded by 
$n - 3 + \f{1}{4}$.
Using \eqref{eq:cat_num}, we get 
\[
|J_1| \les n D^{n-3 + 1/4} \mfr C_{n + d_Y + 1} \les  n D^{n-3 + 1/4} \mfr C_n.
\]

For $J_2$ \eqref{eq:induc_RHS_J}, using \eqref{eq:J2}, \eqref{eq:U_TR_N1_coe2}, and \eqref{eq:cat_num_c}, we get 
\[
|J_2| = |  U_1 \sum_{ 0\leq l \leq d_Y, j \leq \deg G_l} (G_l)_j U^l_{ \mw{TR}, n - j}| 
\les D^{n-2} \mfr C_{n + d_Y -1}
\les D^{n-2} \mfr C_{n}. 
\]

For $J_3$ \eqref{eq:induc_RHS_J}, we bound it similarly 
\[
|J_3| \les D^{n-2} \mfr C_{n}. 
\]

Plugging the above estimates and using \eqref{eq:recur_top} and \eqref{eq:an_low}, we obtain 
\[
 | U_n|  = |\mfr C_n \hat U_n| \leq 
C(I) \f{1}{n} ( |J_1| + |J_2| + |J_3|) \leq 
 C(I) \f{1}{n} ( n D^{n-3+1/4} \mfr C_n + D^{n-2} \mfr C_n )
\]
for some constants $C(I)$ only depending on the interval $I$. We choose $n_L, D$ with
\[
 n_{L} > \max(2 \kp, 2 C(I)), \quad 
 D^{\al} > \max_{i\leq n_L}( |U_i|), \quad 
 D^{1/4} > \max( 4 C(I) ,1) .
 \]
Then for any $n \geq n_{L, 2}$, we prove $|\hat U_n| \leq D^{n-2}$ and the induction argument.

  From \eqref{eq:cat_num}, we have $\mfr C_n \les 4^n$, which along with \eqref{eq:local_Ui} implies $| U_n | \les 4^n D^{n-2}$ for any $n \geq 3$. Thus the power series $U^\rn{\kp}$ converges for any $Y$ with $|Y| < (4D)^{-1}$. 

\vspace{0.05in}
\paragraph{\bf{Continuity}}

Lastly, we prove the continuity of $U^\rn{\kp}$ in $ (\kp, Y) \in  \Om$, where \\ 
$\Om \teq J \times  [-(8D)^{-1}, (8D)^{-1} ] $ and $J = \kp(I) \subset (m, m+1)$. Fix an arbitrary $\e^{\pr} > 0$. We decompose 
\[
U^\rn{\kp}(Y) = \sum_{i\leq m_1} U_i( \g) Y^i + \sum_{ i > m_1} U_i( \g ) Y^i
\teq P(Y, \kp) + Q(Y, \kp), \quad \g = \G(\kp) \in I.
\]
where $\G(\kp)$ is the inverse map  defined in \eqref{eq:kappa_bi}. From the above estimate of $U_i$, by choosing $m_1$ large, we obtain $|Q(Y, \kp)| < \e^{\pr} / 4 $ uniformly in $(\kp, Y) \in \Om$. From \eqref{eq:an_low} and \eqref{eq:recur_top}, each term $U_i(\g)$ with $i \leq m_1$ is uniformly continuous in $\g \in I$.
Since the map $\g = \G(\kp)$ is uniformly continuous in $\kp$ for $\kp \in J$, we obtain that $P(Y, \kp)$ is uniformly continuous in $Y, \kp$. 
In particular, there exists small $\d >0$ such that 
$ |U^\rn{\kp}(Y)  - U^\rn{\kp^{\pr}}(Y^{\pr})|  < \e^{\pr}$ for any $|(\kp, Y) - (\kp^{\pr}, Y^{\pr}) | < \d $. 
\end{proof}

\subsection{Refined asymptotics}

As a consequence of \eqref{eq:induc_asym} in Lemma \ref{lem:asym} and \eqref{eq:cat_num}, we have the following estimates of the original coefficients $U_n$.
\begin{cor}\label{cor:asym}
Let $C_*, \d$ be the constants defined in \eqref{eq:C_asym},\eqref{eq:para3}. There exists $n_2 , \kp_1$ large enough with $n_2 < \kp_1$ such that for any $\kp, n$ with $\kp  > \kp_1$ and $ \kp + 2> n > n_2$, we have
\beq\label{eq:induc_asym2}
\bal
  \B|  U_n -    C_*  U_{n-1} \f{ n \kp}{\kp - n } \B| 
 & \leq  \d \B|   C_*  U_{n-1 } \f{ n \kp}{ \kp - n } \B| , \quad  \d = 0.05 , \\ 
 \eal
\eeq
\end{cor}

  \begin{proof} 
Using $U_n = \hat U_n \mfr C_n$ \eqref{eq:cat_num}, we decompose 
\[
R_n \teq  \f{U_{ n}}{   C_* U_{n-1} \f{ n \kp}{\kp - n } }
= \f{\mfr C_{ n }}{  \mfr C_{n-1} }  \cdot \f{  \hat U_n }{   C_* \hat  U_{n-1} \f{ n \kp}{\kp - n } }
= \f{\mfr C_{ n}}{ 4 \mfr C_{n-1} } \cdot \hat R_n,
\quad \hat R_n = \f{   \hat U_{n}}{   C_* \hat  U_{n-1} \f{ n \kp}{ 4( \kp - n ) } }. 
\]

From the asymptotics \eqref{eq:cat_num_b}, we obtain $\lim_{i\to \infty} \f{ \mfr C_n }{  4  \mfr C_{n-1} }  = 1$. Since $ | \hat R_n -1 | < \hat \d$ from \eqref{eq:induc_asym}, and $ \hat \d < \d$ \eqref{eq:para3}, taking $n_2$ sufficiently large so that $\f{ \mfr C_{n} }{  4  \mfr C_{n-1} } $ is sufficiently close to $1$, we prove $R_n \in [1 -\d, 1 + \d]$ for any $n > n_2$.
\end{proof}

Using Lemma \ref{lem:asym} and Corollary \ref{cor:asym}, we have the following refined estimates of $U_i$, which will be crucially used to estimate barrier functions.

\begin{lem}\label{lem:asym_refine}


Let $\d, n_2$ be the parameters defined in \eqref{eq:para3} and Corollary \ref{cor:asym}. 
There exists $\bar C > n_2$ large enough, such that for any $(n ,\kp)$ with $ n > \bar C, \kp \in (n, n+1)$, the following statement holds.

For any $(m, l)$ with $  l \leq m \leq n $ and $m + l \geq n$ and $ q = \f{2 (1-\d)}{1 + \d} > \f{3}{2}$, we have
\beq\label{eq:Um_convex}
|U_m U_l | \leq C q^{ - (n-m) } U_n \max( |U_{k}|, 1) ,
\quad k = l + m - n.
\eeq

For any $(m, l)$ with $0 \leq l\leq m \leq n - 1$, we have
\begin{subequations}\label{eq:Um_asym_ref}
\begin{align}
& U_m \leq C_l  ( (1 + \d) C_* \kp)^{  m -l } , \quad  l + m \leq n-1,  \label{eq:Um_asym_ref1} \\
& U_m \leq C_l ( C_* \kp)^{ m-l } 4^{-m}, \quad  m \leq n / 8, 
\label{eq:Um_asym_ref2}  \\
& U_m \leq C_q ( (1 + \d) C_* \kp)^m  q^{\min( n-1-m, m)}, \quad q \in ( \f{1}{4}, \f{1}{2})
\label{eq:Um_asym_ref3}. 
\end{align}
\end{subequations}

For any $m \in \BZ $ and $\tau \in \R $ with $ \tau \in ( \f{1}{2}, 1- 10^{-2}), 2 n -1 \geq m \geq \tau n + n$, then for
\bseq\label{eq:Usq_grow}
\beq
T_m \teq \sum_{ i+ j = m, i, j \leq n} U_i U_j.
\eeq
we have the estimate
\beq\label{eq:Usq_grow_b}
T_{m+1} > \B( \f{\tau (1 -\d)}{ ( 1 - \tau) } - O_{\tau, \d}(n^{-1}) \B) ( C_* \kp) T_m\,.
\eeq
\eseq

We have 
\bseq\label{eq:Um_grow2}
\begin{align}
   U_n^{m/n} & \geq C 2^{\min(n-m, m ) }U_m, && m \leq n-1 ,  \label{eq:Um_grow2a} \\
       U_n^{m/n} & \geq  C_{n-m} n^{2/3} U_m ,   && 2n / 3 \leq m \leq n-1 , \label{eq:Um_grow2b} \\
   U_n  & \geq C |\kp-n|^{-1} ((1-\d) C_* n)^{n} , \label{eq:Um_grow2c} 
\end{align}
\eseq

\end{lem}

Estimate \eqref{eq:Um_convex} is a refinement of \eqref{eq:induc_unif}. 
In \eqref{eq:Um_asym_ref}, we compare $U_m$ with the reference scale $(C_* \kp)^m$ and show that the former is much smaller if $m$ is away from $n$ or $0$, and it is not much larger if $m$ is close to $n$. 
Estimate \eqref{eq:Um_asym_ref3} will be useful for $ m \in [c_1 n, c_2 n] $ with $c_1, c_2$ being some absolute constants. Such an estimate is not covered by \eqref{eq:Um_asym_ref1},
\eqref{eq:Um_asym_ref2} as we cannot choose $l$ comparable to $n$ in \eqref{eq:Um_asym_ref1},
\eqref{eq:Um_asym_ref2} (otherwise \eqref{eq:Um_asym_ref1},
\eqref{eq:Um_asym_ref2} are trivial). In \eqref{eq:Usq_grow}, we show that $T_m / (C_* \kp)^m$ grows exponentially to estimate the power series of $U^2$ in later proof of Proposition \ref{prop:bar_final}. We will only apply \eqref{eq:Usq_grow} with $\tau$ away from $1$. In \eqref{eq:Um_grow2}, we show that $U_n$ is very large relative to $U_m$ and $C_* n$. 
In the following proof, we will treat the parameters \eqref{eq:induc_para}, e.g. $n_2$ in Corollary \ref{cor:asym} and the bound of $U_i, i \leq n_2$ as absolute constants.

\begin{proof}
\textbf{Proof of \eqref{eq:Um_convex}.} 
For $n$ large enough, we get $m \geq \f{n}{2}>n_2$. Since $\kp \cdot \f{i}{\kp -i} \asymp 1$ for $ i\leq n_2$, using \eqref{eq:induc_asym2}, we obtain 
\beq\label{eq:Um_growth}
\f{U_n}{U_m} \gtr ((1 - \d) C_* \kp)^{n-m} \prod_{ m < i \leq n} \f{ i}{ \kp -i},
\quad \f{ |U_l|}{ \max( |U_{k} |, 1 ) } \les 
  ((1 + \d) C_* \kp)^{l-k} \prod_{ k < j \leq l} \f{ j}{ \kp -j}.
\eeq


Since $ n - m = l - k$ \eqref{eq:Um_convex}, using a change of variables $(i, j) \to ( m + s, k + s)$, and combining the above two estimates, we obtain 
\[
 \f{U_n   \max( |U_{k} |, 1 )}{ | U_l| U_m}
 \gtr   \B( \f{1 - \d}{1 + \d} \B)^{n-m} \prod_{0< s \leq n -m } 
 \f{ (m+s) (\kp - k -s) }{
 (\kp - m -s)( k+ s) } .
\]
 We estimate the product. Since 
$ k +n= l + m $, 
$k \leq l \leq m \leq n $, $s \geq 1$, and $\kp <n+1$, we obtain
\[
\bal
m - k - (\kp -m -s)
&= 2 m + s - k - \kp
\geq 2 m - k - n = 2 m - l - m = m -l \geq 0, \\
 \kp - ( k + s) & > n - (k + n- m)= m -k \geq 0.
 \eal
\]

Thus, for $0< s \leq n-m$, we estimate each term in the product as 
\[
 \f{ (m+s) (\kp - k -s) }{
 (\kp - m -s)( k+ s) }
 = 1 + \f{ (m-k)\kp}{(\kp - m -s)(k+s)} > 1 + 1 = 2.
\]
Combining the above estimates, we prove \eqref{eq:Um_convex}.


\noindent\textbf{\bf{Proof of \eqref{eq:Um_asym_ref1}.}} 
The estimate is trivial for $m\leq n_2$. 
For $ m > n_2$, using \eqref{eq:induc_asym2} and $ \f{ i}{\kp -i} \cdot \kp \asymp 1$ for $i\leq n_2$, we obtain 
\beq\label{eq:Um_asym_pf1}
\bal
 U_m  & \leq C ( (1 + \d)  C_* \kp)^{ m - l } 
 |U_{n_2}| \cdot  \f{ m (m-1) \cdots (l + 1) }{ (\kp -m)(\kp-m-1) \cdots (\kp- l - 1) }  .
\eal
\eeq

Since we assume $l  + m \leq n-1$ and $ n< \kp$, we get $ m - i < \kp - l - 1 - i$ for $ i =0, 1, .., m- l -1 $, which along with \eqref{eq:Um_asym_pf1} implies \eqref{eq:Um_asym_ref1}. 

\noindent \textbf{\bf{Proof of \eqref{eq:Um_asym_ref2}.}} We only need to consider $ l \leq m \leq n/8$. In this case, for all $i \leq m$, we have $ \f{i}{\kp- i} \leq \f{m}{\kp-m} < \f{1}{7}$. Since $(1 + \d) \f{1}{7} < \f{1}{4}$ \eqref{eq:para3}, using the above estimate, we yield
\[
U_m \leq C_l ((1 + \d)  C_* \kp)^{ m -l  } 7^{-m}
\leq C_l (  C_* \kp)^{ m -l } 4^{-m} .
\]

\noindent \textbf{\bf{Proof of \eqref{eq:Um_asym_ref3}.}}
Choosing $l = n_2$ in the above estimate \eqref{eq:Um_asym_pf1} and using $\kp > n, n_2 ! \les 1$, and  $\kp / (\kp -i) \asymp 1$ for $i \leq n_2$,  we further obtain
\beq\label{eq:Um_growth_up}
U_m \leq C ( (1 + \d )  C_* \kp)^m \prod_{ 1 \leq i \leq m} \f{i}{\kp -i} 
\leq  C ( (1 + \d )  C_* \kp)^m \binom{n-1}{m}^{-1}. 
\eeq
Using the above estimate and Lemma \ref{lem:binom}, we prove \eqref{eq:Um_asym_ref3}.

\noindent \textbf{\bf{Proof of \eqref{eq:Usq_grow}.}}
If $m = 2n-1$, we get 
\[
T_{2n} = U_n^2 \geq  C n \kp U_{n-1} U_n ,\quad     T_{2n-1} = 2 U_{n-1} U_n.
\] 
Since $\d$ is given \eqref{eq:para3} and $\tau < 1 -10^{-2}$, for $n$ large enough, we obtain \eqref{eq:Usq_grow}.  

Next, we consider $m \leq 2n-2$. Denote $l = m - n$. From the assumption of $m$ above \eqref{eq:Usq_grow}, we obtain $ n-1 \geq l \geq \tau n > \f{n}{2}$. By requiring $n_q$ large, since $n > n_q$, we also have $l > n_2$. 
Since $\kp < n+1$, using \eqref{eq:induc_asym2}, we get 
\[
  \f{ U_{l +1 } }{U_l } > (1 - \d) ( C_* \kp) \f{l+1}{\kp-l-1}
  \geq \f{\tau n}{ n - \tau n}  (1 - \d) ( C_* \kp)
  = b   C_* \kp,
  \quad 
  b = \f{\tau (1 - \d)}{ (1 - \tau) } .
\]

As a result, we yield 
\bseq\label{eq:est_Tm}
\beq
\sum_{i+j = m+1, i, j\leq n-1} U_i U_j 
> b 
 C_* \kp \sum_{i+j = m+1, i, j\leq n-1} U_{i-1} U_j 
\geq b 
 C_* \kp \sum_{i+j = m, i, j\leq n-2} U_{i} U_j .
\eeq

Since $U_{n-1} \les n^{-2} U_n$ \eqref{eq:induc_asym2} for $\kp \in (n, n+1)$,  using the above estimate again, we get 
\beq
\bal
2 U_{m-n+1} U_n & >  2 b  C_* \kp   U_{m-n} U_n,  \\
n^{-1} T_{m+1} & \gtr  n^{-1} U_{m-n+1} U_n
\gtr 
 \kp U_{m-n+1} U_{n-1} .
 \eal
\eeq
\eseq

Summing the estimates in \eqref{eq:est_Tm}, we prove 
\[
T_{m+1} + C  n^{-1} T_{m+1} > b C_* \kp \sum_{i+j=m, i, j \leq n-2} U_i U_j 
+ 2 b C_* \kp (  U_{m-n} U_n  + U_{m - n +1} U_{n-1} ) = b  C_* \kp T_m .
\]
Dividing $1+ C  n^{-1}$ on both sides, we prove \eqref{eq:Usq_grow}.

\noindent \textbf{\bf{Proof of \eqref{eq:Um_grow2c}}.} 
Next, we prove \eqref{eq:Um_grow2c}, \eqref{eq:Um_grow2a}, and \eqref{eq:Um_grow2b} in order. 
Recall $n_2$ chosen in Corollary \ref{cor:asym}. Using \eqref{eq:induc_asym2} and $\kp - m < n+1-m$ for $m=1,2,..,n-1$, we prove 
\[
\bal
U_n &\geq C ( (1- \d) C_* \kp)^{n- n_2} \prod_{n_2 +1 \leq l \leq n} \f{l}{\kp-l}
\geq C ( (1- \d) C_* \kp)^{n}\prod_{1 \leq l \leq n} \f{l}{\kp-l} \\
& \geq C (\kp-n)^{-1}  ( (1- \d) C_* \kp)^{n} \prod_{1 \leq l \leq n} \f{l}{ n+1-l}
= C (\kp-n)^{-1}  ( (1- \d) C_* \kp)^{n} .
\eal
\]

\noindent \textbf{\bf{Proof of \eqref{eq:Um_grow2a}.}}
 If $m \leq n_2$, the estimate \eqref{eq:Um_grow2a} follows from \eqref{eq:Um_grow2c}. Next, we consider $n_2< m \leq n-1$.  Using estimates similar to \eqref{eq:Um_growth} and  $\f{i}{\kp -i} > \f{i}{n+1 - i}$ for $\kp \in (n, n+1)$,
  we get
  \beq\label{eq:Um_growth_low}
  \bal
\f{U_n}{U_m} & \geq C ( (1 - \d)  C_* \kp )^{ (n-m) } 
 \prod_{ m+1\leq i \leq n} \f{i}{ n+1-i} 
\geq C   ( (1 - \d)  C_* \kp )^{ (n-m) }
 \binom{n}{n-m} .
 \eal
  \eeq

Since $ m \leq n-1$, combining \eqref{eq:Um_growth_low} and the upper bound of $U_m$ \eqref{eq:Um_growth_up}, we estimate
\[
    \f{U_n^{m/n}}{U_m}
    =  (\f{ U_n}{U_m})^{ \f{m}{n} } (U_m)^{ - \f{n-m}{n} } 
\geq C \f{ ( (1 -\d) C_* \kp)^{ (n-m) m/n} }{ ( (1 + \d) C_* \kp)^{m (n-m)/n} }
 \binom{n}{n-m}^{m/n} 
 \binom{n-1}{m}^{ (n-m)/n}.
\]
The common factor $C_* \kp$ in the denominator and numerator is cancelled.

Denote $\lam = \f{1-\d}{1 + \d}< 1$. We have $ 4 \lam > 2$. We further consider $m \leq \f{n-1}{2}$ and $m > \f{n-1}{2}$. (a) For $ m \leq \f{n-1}{2}$, which gives $m< \f{n}{2}$, applying Lemma \ref{lem:binom} to $ \binom{n-1}{m},  \binom{n}{m} =  \binom{n}{n-m}$, we get 
\[
 \f{ U_n^{m/n}}{ U_m} \geq C m^{-1/2} \lam^{ \f{ (n-m)m}{n}} 4^{ \f{ m^2}{n}  + \f{ m(n-m)}{n}} 
 \geq 
 C m^{-1/2} \lam^{ \f{ (n-m)m}{n} } 4^{ m}  
 \geq C m^{-1/2} (4 \lam )^m> C 2^m.
\]
For $n >m > \f{n-1}{2}$, which implies $m \geq \f{n}{2}$, using $\binom{n-1}{m} = \binom{n-1}{n-1-m}$ and Lemma \ref{lem:binom}, we get 
\[
U_n^{m/n} / U_m \geq C (n-m)^{-1/2}  \lam^{ \f{(n-m)m}{n}}
4^{(n-m)} 
\geq C (n-m)^{-1/2} (4 \lam)^{n-m}  > C 2^{n- m} .
\]
We prove \eqref{eq:Um_grow2a}.

\noindent \textbf{\bf{Proof of \eqref{eq:Um_grow2b} }}
For \eqref{eq:Um_grow2b}, since $m \geq \f{2n}{3} > n_2$, applying \eqref{eq:Um_growth_low} 
and then using $ \f{l}{n+1- l} > 2$ for $l \geq m + 1 \geq 1 + \f{2 n}{3}$, and $2 ( 1 -\d) > 1 + \d$ \eqref{eq:para3}, we get 
\[
 \f{U_n}{U_m} \geq C ( (1 - \d) C_* \kp)^{n-m} \prod_{ m+1 \leq l \leq n} \f{ l}{ n + 1- l} 
 \geq C ( (1 - \d) C_* \kp)^{n-m} 2^{n-m} n
 > C ( (1 + \d) C_* \kp)^{n-m} n.
\]
 Combining the above estimates and using $1 > \f{m}{n} \geq \f{2}{3}$, we prove
\[
\f{ U_n^{m/n}}{U_m} >  \f{   ( C n ( (1+\d) C_* \kp)^{n-m} U_{m} )^{ \f{m}{n} }}{U_m} > C n^{2/3}   \f{   (  ( (1+\d) C_* \kp)^{n-m}  )^{ \f{m}{n} }}{  ( (1 + \d) C_* \kp )^{ (1 -\f{m}{n})  m } }
= C n^{2/3}.
\]
\end{proof}

\section{Smooth solution connecting $P_O$ and $P_s$}\label{sec:shoot}

In this section, we study the phase portrait of $(Y, U)$ above the point $Q_O$, which corresponds to the region $Z \in [0, Z_0]$ in the original ODE \eqref{eq:ODE}. Our goal is to prove the following result. 
\begin{prop}\label{prop:Qs_QO}
There exists $C >0$ large enough such that for any $n$ with $n > C$, there exists $\kp_n \in (n, n+1)$ and $\e_1 > 0$, such that the following statement holds true. 
The ODE \eqref{eq:ODE} admits a solution $V^\rn{\kp_n}(Z) \in C^{\infty}([0, Z_0 + \e_1 ])$ with 
$V^\rn{\kp_n}(0)= 0,  V^\rn{\kp_n}( Z_0 )=  V_0$, and 
\beq\label{eq:Qs_QO_prop1}
 \quad  V^\rn{\kp_n}( Z) < Z, \quad V^\rn{\kp_n}( Z) \in (-1, 1)
\eeq
for any $Z \in (0, Z_0 + \e_1]$, and $V^\rn{\kp_n} (Z) = Z g (Z^2)$ for some function $g \in C^{\infty}[0, Z_0 + \e_1 ])$. Moreover, it agrees with the local analytic function $U^{(\kp_n)}(Y)$ near $Y = 0$ with $(0, U^{(\kp_n)}(0))= Q_s$ constructed in Proposition \ref{prop:analy} in the following sense 
\beq\label{eq:sonic_agree}
 (Z, V^\rn{\kp_n}(Z)) = (\cZ, \cV)( Y, U^{(\kp_n)}(Y))
\eeq
for $|Y| < \e_1$ with small $\e_1 > 0$, where $(\cZ, \cV)$ are the maps defined in \eqref{eq:UY_to_VZ}.

\end{prop}


Throughout this section, we use $n$ to denote an integer rather than a dummy index and will first consider $\kp \in (n-1, n)$ or $\kp \in (n, n+1)$. In Section \ref{sec:shoot_Qs_QO}, we will consider $\kp \in (n, n+1)$ and prove Proposition \ref{prop:Qs_QO} using a shooting argument.

\vspace{0.1in}
\paragraph{\bf{Ideas and barrier argument}}
We first construct the far-field lower $B_l^f$ and upper barriers $B_u^f$ with  $ B_u^f \leq B_l^f$ (see \eqref{eq:bar_far_l} and \eqref{eq:bar_far_u} for the definitions). If $(Y_*, U(Y_*))$ is outside the region 
\beq\label{eq:bar_region}
\Om_{B}^f = \{ (Y, U) : B_u^f(Y) < U < B_l^f(Y) , \  0 < Y < Y_O \},
\eeq
the solution $(Y, U(Y))$ will remain outside $\Om_B^f$ for any $Y \in (Y_*, Y_O) $ (see Section \ref{sec:bar_far} and Proposition \ref{prop:bar_f}). We define the boundaries of $\Om_B^f$
\beq\label{eq:bar_region_edge}
E_{B, 1} = \{  (Y, B_u^f(Y)): Y \in (0,  Y_O ) \}, 
\quad E_{B, 2}  = \{  (Y, B_l^f(Y)): Y \in (0,  Y_O ) \} .
\eeq
See Figure \ref{fig:coordinate2} for illustrations of $B_u^f, B_l^f, \Om_B^f$, and the solution curve.

We construct the local upper and lower barriers $\BB_u^{\mw{ne}}(Y), \BB_l^{ \mw{ne}}(Y)$ 
based on $\BB_\sn{n}^\mw{ne}$ \eqref{eq:bar_loc}.

In Propositions \ref{prop:bar_loc}, \ref{prop:bar_loc_valid}, we will show that $\BB_u^{ \mw{ne} }(Y)$ is a upper barrier for $U(Y)$, which is valid for $Y \in (0, t_{bar})$. Then, 
in Propositions \ref{prop:bar_inter}, \ref{prop:bar_exit}, we show that $\BB_u^{ \mw{ne} }(Y)$ intersects  $B_u^f$ at some $t_I < t_{bar}$ for some $\kp \in (n-\e, n)$ with large $n$ and small $\e$. This implies that the solution exits the region $\Om_B^f$ via $E_{B, 1}$ for this $\kp$. We perform a similar argument for the lower barriers $\BB_l^{ \mw{ne} }, B_l^f$ and some $\kp \in (n, n+\e)$ with large $\kp$ and small $\e$. See Sections \ref{sec:bar_loc}, \ref{sec:Qs_QO_inter}.

In Section \ref{sec:shoot_Qs_QO}, using continuity, we construct a smooth solution from $Q_s$ to $(Y, U(Y))$ with $Y$ close to $Y_O$, and the solution agrees with the local analytic solution $(Z, V(Z))$ to \eqref{eq:ODE} near $Z=0$ via the map $(\cY, \cU)$ \eqref{eq:sys_UY}. Using the inverse map $(\cV, \cZ)$ \eqref{eq:UY_to_VZ} proves Proposition \ref{prop:Qs_QO}.

Here, \textit{u, l, ne, f} are short for \textit{upper, lower, near}(for local), \textit{far}, respectively.

\begin{figure}[t]
  \centering
  \includegraphics[width=0.6\linewidth]{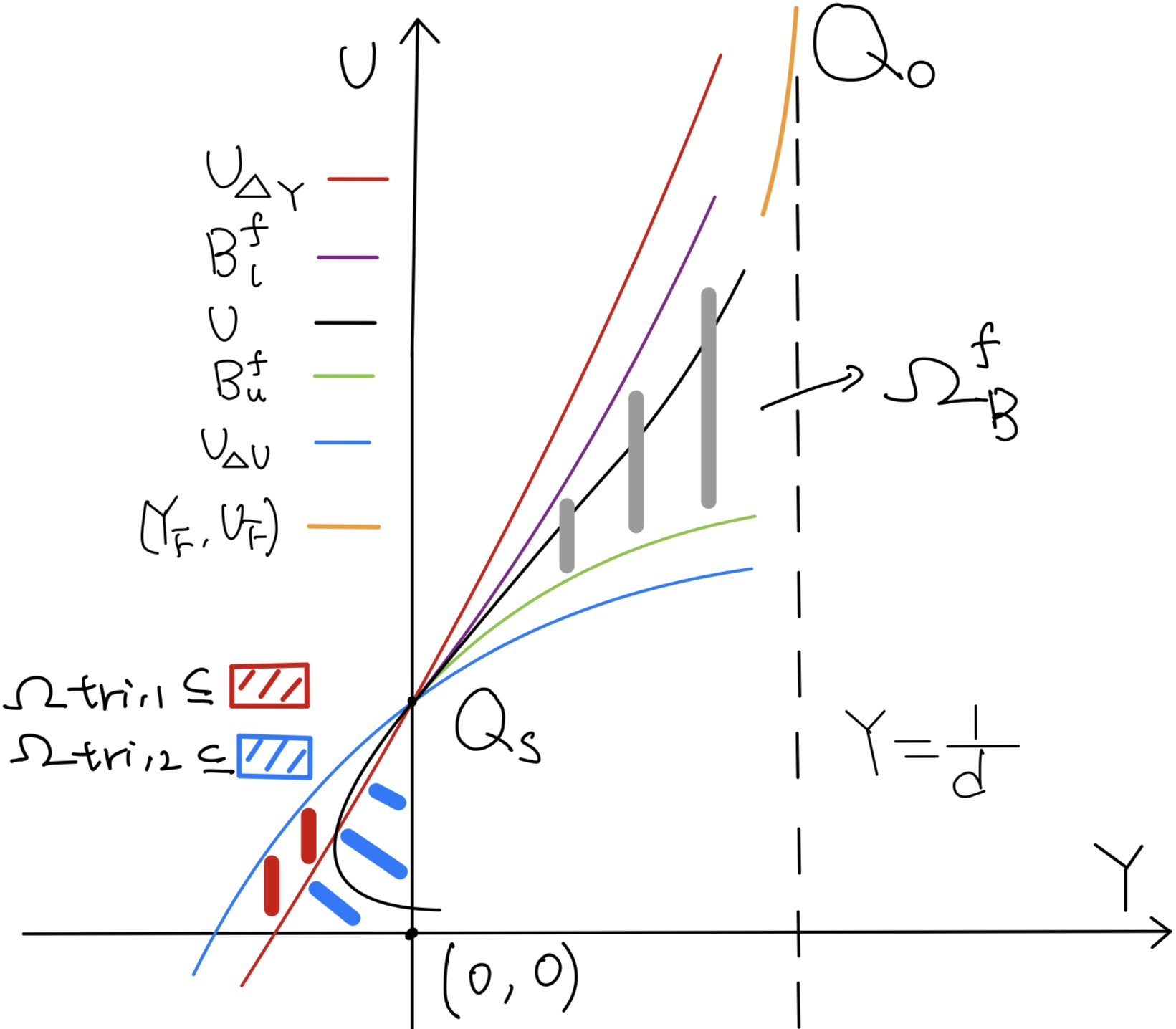}
  \caption{Illustrations of phase portrait of the $(Y, U)$-ODE \eqref{eq:ODE}.
The black curve represents the $C^{\infty}$ solution curve, 
$(Y_F, U_F)$ (orange) defined in \eqref{eq:QO_glue} is the solution curve near $Q_O= (1/d, \infty)$, 
$B_l^f$ (purple) and $B_u^f$ (green) are barrier functions defined in \eqref{eq:bar_far_l}, \eqref{eq:bar_far_u},  $U_{\D_Y}$ (red) and $U_{\D_U}$ (blue) defined in \eqref{eq:root_UY} are roots of $\D_Y, \D_U$.  
The domain $\Om_B^f$ is defined in \eqref{eq:bar_region} and $\Om_{\mw{tri},i}$ in \eqref{eq:dom_tri}. 
   }
 \label{fig:coordinate2}
\end{figure}

\subsection{Far-field barriers}\label{sec:bar_far}

We construct the far-field upper barrier 
\bseq\label{eq:bar_far}
\beq\label{eq:bar_far_u}
B_u^f(Y) = U_0 +  U_1 Y  + 2 Y^2 ,
\eeq
and the far-field lower barrier 
\beq\label{eq:bar_far_l}
B^f_{l}(Y) = \f{ e_1 Y + e_2 Y^2 - U_0}{ d Y - 1 },
\eeq
with $e_1, e_2$ satisfying
\begin{align}
\f{d}{dY} B_{l}^f(Y) |_{Y = 0}  & = U_1, \label{eq:bar_far_l_cond1} \\
\f{d^2}{dY^2} 
\B( (B^{f}_l)^{\prime}(Y) \D_Y(Y, B^f_l(Y)) - \D_U(Y, B^f_l(Y))   < 0  \B) \B|_{Y = 0} &= - 40 . 
\label{eq:bar_far_l_cond2} 
\end{align}
\eseq

For \eqref{eq:bar_far_l_cond1},  a direct calculation obtains the derivative and solves $e_1$:
\[
 d U_0 - e_1 = U_1, \quad e_1 = d U_0 - U_1 .
\] 
We impose  \eqref{eq:bar_far_l_cond2} to show that the quantity in the bracket in  \eqref{eq:bar_far_l_cond2} is negative. See \eqref{eq:bar_far_ds_a}.  It is easy to see that \eqref{eq:bar_far_l_cond2} is linear in $e_2$, and we can solve $e_2$ easily with symbolic computation. 

By definition, $B_l^f, B_u^f$ agree with $U(Y)$ near $Y = 0$ up to $O(Y^2)$. Moreover, using \eqref{eq:eqn_U10}, \eqref{eq:eqn_U1}, and \eqref{eq:grad_Q}, near $Y=0$, we obtain 
\[
 (B^{f}_{\al})^{\prime}(Y) \D_Y(Y, B^f_{\al}(Y)) - \D_U(Y, B^f_{\al}(Y)) = O(Y^2),\quad \al \in \{ l, u\}.
\]

Recall the roots $U_{\D_Y}, U_{\D_Z}$ from \eqref{eq:root_UY}, and $U_g$ from \eqref{eq:root_DZ}.
In Propositions \ref{prop:bar_f}, \ref{prop:bar_loc_valid}, we will show the following relative positions of curves near $Q_s$
\beq\label{eq:rela}
 U_{\D_U}(Y) < B_u^f(Y) < U(Y) < B_l^f(Y) < U_{\D_Y}(Y),
 \quad U_g(Y) < B_u^f(Y),
\eeq
for $0 < Y \ll 1$. See Figure \ref{fig:coordinate2} for illustrations of the relative positions
of these curves.

We have the following results.

\begin{prop}[Computer-assisted]\label{prop:bar_f}

There exists a large $C$ such that for any $\kp > C$, the following statements hold true.
The functions $B_l^f, B_u^f$ \eqref{eq:bar_far} satisfies $B_l^f(0) = B_u^f(0) = 0$ and 
\bseq\label{eq:bar_far_prop}
\begin{align}
 & \pa_Y^2 B_u^{f}(0) <  \pa_Y^2 U(0) < \pa_Y^2 B_l^f(0), && \label{eq:bar_far_valida} \\
 & B_l^f(Y) < U_{\D_Y}(Y) , &&0<Y\leq  Y_O , \label{eq:bar_far_validc} \\
&  U_{\D_U}(Y)  <  B_u^f(Y) < B_l^f(Y) , &&  0<Y\leq  Y_O  , \label{eq:bar_far_validd}   \\
  &(B^{f}_l)^{\prime}(Y) \D_Y(Y, B^f_l(Y)) - \D_U(Y, B^f_l(Y))     < 0, 
&&0<Y\leq  Y_O
   \label{eq:bar_far_ds_a}  \\
  &  (B^{f}_u)^{\prime}(Y) \D_Y(Y, B^f_u(Y)) - \D_U(Y, B^f_u(Y))    > 0, 
  &&0<Y\leq  Y_O \label{eq:bar_far_ds_b} 
  \end{align}
\eseq
where $Y_O = \f{1}{d}$. Moreover, the function $U_g$ \eqref{eq:root_DZ} and $U_{\D_U}$ \eqref{eq:root_UY} satisfies 
\beq\label{eq:root_ineq}
  \pa_Y U(0) > c > 0, \quad  \pa_Y U(0) -  \pa_Y U_{\D_U}(0) > c > 0, 
  \quad 
    \pa_Y U(0) -  \pa_Y U_g(0) > c > 0 ,
\eeq
for some constant $c$ independent of $\kp$. The map $\cZ$ \eqref{eq:UY_ODE} along the solution curve satisfies 
\beq\label{eq:dZ_dY}
 \f{d \cZ(Y, U(Y)}{ d Y} \B|_{Y= 0} 
=  \f{d \cZ(Y, U_0 + U_1 Y ) }{ d Y} \B|_{Y= 0}  < 0.
\eeq


As a result, we have $\Om_B^f \subset \cR_{YU}$ for $\cR_{YU}$ defined in Lemma \ref{lem:bijec}, and
\begin{align}
 U_g(Y) < B_u^f(Y) , \quad  \D_Z( (\cZ, \cV)(Y, U) ) &> 0  , &&  (Y, U) \in \Om_B^f ,   \label{eq:DelZ_sign}  \\
 \D_Y(Y, U) & > 0  , &&  (Y, U) \in \Om_B^f . \label{eq:DelY_sign}  
  \end{align}
and $B_{u}^{f}, B_l^f$ are an upper, and lower barrier for the ODE with $Y \in [0, Y_O)$,  respectively.

\end{prop}

The proof involves computer assistance, and we refer to Appendix \ref{app:comp} for further details. 

\begin{proof}

The properties $B_l^f(0) = B_u^f(0) = U_0 = \e$ follow from the definition \eqref{eq:bar_far} and \eqref{eq:sonic_Q}. 

For \eqref{eq:dZ_dY}, since $U(Y)$ is smooth near $Y = 0$ and 
\[
|U(Y) - U_0 - U_1 Y | \les C_U Y^2, \quad |U^{\pr}(Y)  - U_1  | \les C_U Y,
\]
using chain rule, we prove the first identity in \eqref{eq:dZ_dY}. The second term in \eqref{eq:dZ_dY} is a scalar and only depends on $U_0, U_1$. We do not expand it below as we can derive it symbolically.

The terms in inequality \eqref{eq:bar_far_valida}, \eqref{eq:root_ineq}  are explicit constants.

Inequality \eqref{eq:bar_far_validc} or \eqref{eq:bar_far_validd}, \eqref{eq:dZ_dY} compares two rational functions with explicit formulas.

The term in inequality \eqref{eq:bar_far_ds_a} or \eqref{eq:bar_far_ds_b} is a polynomial. 

\noindent \textbf{Methods.}
Recall that choosing $\kp$ sufficiently large is equivalent to choosing $\g$ close to $\ell^{-1/2}$ \eqref{eq:grad_smooth}, \eqref{eq:gamma_star}. Each inequality in \eqref{eq:bar_far_prop} and  \eqref{eq:root_ineq} can be reformulated equivalently as 
\beq\label{eq:poly_dec}
P(Y ; \g) = Y^i g(Y ; \g ) > 0,
\eeq
for some $i \geq 0$, where $P, g$ are polynomials  in $Y$ and continuous in $\g$. 
We verify that 
\[
 g(Y, \ell^{-1/2} ) > c > 0, \quad \mathrm{ for \ any } \quad Y \in [0, Y_O], 
\]
uniformly with some $c>0$ using computer assistance. See Appendix \ref{app:comp} for more details. Then using continuity, we prove \eqref{eq:bar_far_prop} for $\g$ close to $\ell^{-1/2}$.

\noindent \textbf{Consequences}. 
Recall $\Om_B^f$ from \eqref{eq:bar_region} and $\cR_{YU}$ from Lemma \ref{lem:bijec}. Since $U_0>0$ \eqref{eq:sonic_Q} and $U_1 = \pa_Y U(0) > 0$ \eqref{eq:root_ineq}, we obtain that $B_u^f$ \eqref{eq:bar_far_u} is increasing for $Y \in [0, Y_O)$. Therefore, for any $(Y, U) \in \Om_B^f$ \eqref{eq:bar_region}, we obtain $U > B_u^f(Y) > B_u^f(0) = U_0 > 0$, and $ \Om_B^f\subset \cR_{YU} $.

\noindent \textbf{Estimate \eqref{eq:DelZ_sign}}
Since both  $U_g$ \eqref{eq:root_DZ} and $B_u^f$ \eqref{eq:bar_far_u} are quadratic polynomials
with $U_g(0) = B_u^f(0) = U_0=\e$, we get 
\[
B_u^f(Y) - U_g(Y)
= ( U_1 - \pa_Y U_g(0) ) Y +  ( 2 - (\ell -1) ) Y^2
 = ( U_1 - \pa_Y U_g(0) ) Y + (3 - \ell )Y^2.
\]
Since $Y > 0, 3 - \ell > 0$ \eqref{eq:para}, and $ U_1 - \pa_Y U_g(0) > 0$ from \eqref{eq:root_ineq}, we prove $U_g(Y) < B_u^f(Y)$ for any $Y >0$ and obtain the first estimate in \eqref{eq:DelZ_sign}.

To prove $\D_Z( (\cZ, \cV)(Y, U)) > 0$ \eqref{eq:DelZ_sign}, using the formula \eqref{eq:DZ_UY}, we only need to show 
\beq\label{eq:DelZ_sign_pf}
 U > U_g(Y) , \quad Z =\cZ(Y, U) > 0.
\eeq
For any $(Y, U) \in \Om_B^f \subset \cR_{YU}$ with  $\cR_{YU}$ defined in Lemma \ref{lem:bijec}, 
using the first estimate in \eqref{eq:DelZ_sign} and Lemma \ref{lem:bijec}, we have $U > B_u^f(Y) > U_g(Y)$ and $\cZ( Y, U) > 0$, which implies \eqref{eq:DelZ_sign_pf}


\noindent \textbf{Estimate \eqref{eq:DelY_sign}}
For $(Y, U) \in \Om_B^f$ , the definition of $\Om_B^f$ \eqref{eq:bar_region} and \eqref{eq:bar_far_validc} imply 
\[
 d  Y -1  < 0, \quad  U - U_{\D_Y}(Y) < B_l^f(Y) -  U_{\D_Y}(Y) < 0,
\]
for any $Y \in (0, Y_O)$, where $Y_O = \f{1}{d}$. Using the definition $\D_Y = ( d  Y -1 ) ( U - U_{\D_Y}(Y))$ \eqref{eq:UY_ODE}, we prove $\D_Y > 0$ \eqref{eq:DelY_sign}. 

Similarly, using \eqref{eq:bar_far_validb}, \eqref{eq:bar_far_validc}, we get 
\[
\D_Y(Y, B_{\al}^f(Y) ) 
= (d Y - 1) ( B_{\al}^f(Y) - U_{\D_Y}(Y) ) > 0, \quad \al = l, u,
\]
for any $Y \in (0, Y_O) , Y_O = \f{1}{d}$. The above estimate and \eqref{eq:bar_far_ds_a}, \eqref{eq:bar_far_ds_b} imply that  $B_{u}^{f}, B_l^f$ are an upper, and lower barrier for the ODE with $Y \in [0, Y_O]$,  respectively. 
\end{proof}

\subsection{Local barrier}\label{sec:bar_loc}

For $\b < -n$ with large $|\b|$ to be chosen, we construct the local barrier as 
\beq\label{eq:bar_loc}
\bal
\BB_\sn{n}^{ \mw{ne} }(Y) = \sum_{i=0}^{n} U_i Y^i+  \b U_{n} Y^{n+1} .
\eal
\eeq

We define 
\beq\label{eq:bar_loc_sign}
\PP_\sn{n}^{ \mw{ne}}  =  \D_Y(Y, \BB^{ \mw{ne} }_{[n]}(Y)) 
\B( \BB^{ \mw{ne} }_{[n]} \B)^{\pr}(Y)
 - \D_U(Y, \BB^{ \mw{ne} }_{[n]}(Y) ). 
\eeq

We have the following results regarding the sign of $\PP_\sn{n}^{ \mw{ne} }$.

\begin{prop}\label{prop:bar_loc}

There exist large absolute constants $n_3,  C \gg 1$  such that for any $n > n_3$, and $ \b < - C n^2 < 0$, the following holds true.

(a) For any $ \kp \in (n- 1/2 , n)$ and  $ Y \leq \mu_n \min(  (  |(n-\kp) \b| )^{1/ { (n- 2)} }  , |\b|^{-1} ) $,  we have $\PP_\sn{n}^{ \mw{ne} } > 0$,

(b) For any $\kp \in (n, n+ 1/2)$ and $ Y \leq \mu_n \min(  (  |(n- \kp ) \b| )^{1/ { (n- 2)} }  , |\b|^{-1} ) $, 
we have $\PP_\sn{n}^{ \mw{ne} } < 0$.

Here, $0 < \mu_n < 1/2$ is some constant depending on $n$.
\end{prop}

\begin{proof}

We require $n_3$ larger than the constant $C$ in Lemma \ref{lem:asym_refine}. Since $|n - \kp |<1/2$, we first choose $\mu_n$ small enough so that 
\beq\label{eq:bar_loc_Y}
Y \leq \mu_n \min(  (  |(n- \kp) \b| )^{1/ { (n- 2 ) } }  , |\b|^{-1} ) < 1 .
\eeq

In the following estimates, we will mainly track the dependence of the constants on $|\kp - n|$, which plays a role as a small parameter, and $n, \b$, which are large.

Below, we shall simplify $\PP_\sn{n}^\mw{ne}$ as $P$, and use $a_m$ to denote the $m$-th coefficient in the power series expansion of $a$. We will show that
\bseq\label{eq:bar_loc_est1}
\beq\label{eq:bar_loc_est1a}
 \PP_\sn{n}^\mw{ne}(Y)
 =  P_{n+1} Y^{n+1} + \cE_n,  \\
 \eeq
where $P_{n+1}$ is given by 
\beq\label{eq:bar_loc_lead}
  P_{n+1} = ( (n+1) \lams  - \laml ) (\b  U_{n} - U_{n+1}) , 
\eeq
and $\cE_n$  satisfies the bound 
\beq\label{eq:bar_loc_est1b} 
 | \cE_n|  \les_n  |\b| | \kp-n|^{-1} Y^{n+2}  +   | \kp -n|^{-2}  Y^{2n-1} + |\b| |\kp-n|^{-2}  Y^{2n}  + |\b|^2 |\kp-n|^{-2} Y^{2n+1} . 
\eeq

\eseq

From \eqref{eq:bar_loc}, \eqref{eq:bar_loc_sign},  $\PP_\sn{n}^\mw{ne}(Y)$ is a polynomial in $Y$. Since the barrier \eqref{eq:bar_loc} and the power series of $U$ agree up to $|Y|^{n+1}$, using Lemma \ref{lem:pow_recur}, we get that the term $Y^i, i \leq n$ in power series expansion of $\PP_\sn{n}^\mw{ne}$ vanishes. Moreover, using the notation $\mfr R_n$ \eqref{eq:recur_topN_coe0}, 
and the derivations \eqref{eq:recur_0}, \eqref{eq:recur_topN_0} with $N=1$ and $a_{n+1, n+1} = (n+1 ) \lams - \laml $  from  \eqref{eq:recur_top},  we obtain the coefficient of $Y^{n+1}$ in $P$
\[
\bal
P_{n+1} & = \mfr R_{n+1}(Y_0, U_0, .., U_n, \b U_n ) 
= \mfr R_{n+1}(Y_0, U_0, .., U_n, \b U_n )  - 
 \mfr R_{n+1}(Y_0, U_0, .., U_n, U_{n+1} ) \\
 & = a_{n+1, n+1} (\b U_n - U_{n+1}) ,
 \eal
\]
which is \eqref{eq:bar_loc_lead}.

Next, we estimate the coefficients $P_m$ of $P$ with $m \geq n+2$. 
Applying Corollary \ref{cor:asym} and tracking the dependence on $|\kp - n|$, we obtain
\beq\label{eq:bar_loc_coe1}
 |U_i| \les_n 1, \  \mathrm{for } \ i \leq n-1, \quad |U_n| \asymp_n |\kp-n|^{-1} ,
 \quad |U_{n+1}| \les n^2 |U_n|.
\eeq

 Next, we estimate the coefficients $\D_{Y, i}, \D_{U,i}$ of the power series of $\D_Y(Y, \BB_\sn{n}^\mw{ne}), \D_U(Y, \BB_\sn{n}^\mw{ne})$. Since $\D_Y(Y, U)$ is linear in $U$ \eqref{eq:UY_ODE}, using \eqref{eq:bar_loc_coe1}, we obtain 
\[
  |(\D_Y)_i | \les_n 1, \mathrm{\ for \ } i \leq n - 1, \quad 
 | (\D_Y)_n | \les_n |\kp-n|^{-1},  \quad 
| (\D_Y)_{i}| \les_n |\b| | \kp-n|^{-1},  \mathrm{\ for \ } i \geq n + 1 .
\]

Since $\D_U$ is quadratic in $U$ \eqref{eq:UY_ODE} with the nonlinear term $2 U^2$, using the product formula \eqref{eq:power_prod} and the estimates \eqref{eq:bar_loc_coe1},  yields 
\[
\bal
   | (\D_U)_i| &\les_n 1, &&  i \leq n- 1, \\ 
 | (\D_U)_i| &\les_n | \kp - n|^{-1},&&  i=n \\
  | (\D_U)_{i}| &\les_n |\b| | \kp-n|^{-1},  &&  n+1 \leq i \leq  2 n - 1 , \\ 
  |\D_U)_{i} |  & \les_n  |\b || \kp-n|^{-1} + | \kp-n|^{-2} ,&& i \geq 2 n.
 \eal
\]

Recall $\BB_\sn{n}^\mw{ne}$ from \eqref{eq:bar_loc}. Using
\[
 \pa_Y \BB_\sn{n}^\mw{ne} = \sum_{1\leq i\leq n} (i U_i) Y^{i-1} + (n+1) \b U_n Y^n, 
\]
the above estimates of of $\D_{Y,i}, \D_{U, i}$, and the product formula \eqref{eq:power_prod}, 
for $|Y|<1$, we obtain 
\[
P =  P_{n+1} Y^{n+1}
+ O_n( |\b| | \kp-n|^{-1} Y^{n+2} +  | \kp -n|^{-2}  Y^{2n-1} + |\b| | \kp-n|^{-2}  Y^{2n} 
+ |\b|^2 | \kp-n|^{-2} Y^{2n+1} ),
\]
and establish \eqref{eq:bar_loc_est1b}. 

Next, we estimate  $P_{n+1}$ and $\cE_n$. Using $\kp = \f{\laml}{\lams}$ \eqref{eq:kappa}, we first rewrite $P_{n+1}$ as follows 
\[
 P_{n+1}   =  ( \f{n+1}{ \kp} - 1 ) \laml (\b U_n - U_{n+1})
= \f{n+1 - \kp}{\kp} \laml (\b U_n - U_{n+1}) . 
\]
Choosing $\b < 0 $ with $n^2 \ll |\b|$ and using \eqref{eq:bar_loc_coe1}, we get $|U_{n+1}| < \f{1}{2 } |\b U_n| $. Since $\lams, \laml > 0$ \eqref{eq:lam_sign}, $|\laml|\asymp 1$, and $|\kp-n|  < \f{1}{2}$, we estimate 
\begin{gather*}
  |P_{n+1}|  \gtr_n |\b U_n - U_{n+1}| \gtr |\b U_n|
\gtr |\b| |\kp-n|^{-1},  \\
  \sgn( P_{n+1})  = \sgn(\b U_n)
= - \sgn(U_n) .
\end{gather*}

To ensure that the error term $\cE_n$ in \eqref{eq:bar_loc_est1} is smaller than $P_{n+1} Y^{n+1}
$, e.g. $|\cE_n| < \f{1}{2} |P_{n+1} Y^{n+1}| $, we require
\[
 |Y|^{n-2} \ll_n  \b |\kp-n|,
 \quad |Y| \b \ll_n 1 ,
\]
which are achieved by choosing $\mu_n$ sufficiently small in \eqref{eq:bar_loc_Y}.
As a result, we get 
\[
 \sgn(P(Y)) = \sgn(P_{n+1} Y^{n+1}) = \sgn( P_{n+1}) = - \sgn(U_n).
\]

Recall that we assume $\kp \in (n- 1/2, n)$ or $\kp \in (n, n+1/2)$. 
Since $U_{n-1} > 0$ from \eqref{eq:U_sign}, using the relation between $U_{n-1}, U_n$ in \eqref{eq:induc_asym2}, we obtain
\[
U_n < 0, \quad \kp \in (n-1/2, n), \quad U_n > 0,  \quad \kp \in (n, n+1/2),
\]
which implies 
\[
P(Y) > 0, \quad 
 \kp \in (n-1/2, n), \quad P(Y) < 0,  \quad \kp \in (n, n+1/2). 
\]
We conclude the proof.
\end{proof}

Next, we derive the relative positions among the barriers and the solution.

\begin{prop}\label{prop:bar_loc_valid}
Let $U_{\D_{\al}}$ be the root of $\D_{\al}(Y, U) = 0, \al = Y, U$ defined in \eqref{eq:root_UY} and $n_3$ be the parameter chosen in Proposition \ref{prop:bar_loc}. For any $n > n_3$ and $\b < - C n^2$ with $C$ large, we have 
\bseq\label{eq:bar_loc_valid}
\begin{align}
  \BB_\sn{n+1}^\mw{ne}(Y) - U(Y) & > 0, \quad  \mathrm{for \ any \ }  \kp \in (n +  1/2, n+1), 
\label{eq:bar_loc_valid_a}
 \\
  \BB_\sn{n}^\mw{ne}(Y) - U(Y) & < 0, \quad \mathrm{for \ any \ } \kp \in (n , n+ 1/2), 
  \label{eq:bar_loc_valid_b}
 \end{align}
\eseq
for  $0 < Y \les_{\kp, n, \b} 1$ and
\begin{align}
  B_u^{f}(Y) &<  U(Y) < B_l^f(Y) , \quad  0 < Y \ll_{\kp, n} 1 . \label{eq:bar_far_validb} 
\end{align}

\end{prop}

For $\kp \in (n, n+1)$, we will use $\BB_\sn{n+1}^\mw{ne}(Y)$ as the local upper barrier for $U(Y)$, and $\BB_\sn{n}^\mw{ne}$(Y) as the local lower barrier.

\begin{proof}

Since $U(Y)$ is local analytic near $Y = 0$, using \eqref{eq:bar_loc}, we get 
\beq\label{eq:bar_loc_valid_pf}
 \BB_\sn{n+1}^\mw{ne}(Y) - U(Y)
=  ( \b U_{n+1} - U_{n+2} ) Y^{n+2} + O_{n, \kp ,\b}(Y^{n+3}).
\eeq

For $ \kp \in (n + 1/2, n+1)$, since $U_n > 0$ from \eqref{eq:U_sign} in Lemma \ref{lem:asym}, 
using \eqref{eq:induc_asym2} in Corollary \ref{cor:asym} with $i = n, n+1$, for $ \b < - C n^2$ with large $C$, we get 
\[
\bga
 \sgn(U_{n+1}) = \sgn(\kp- n-1) \sgn(U_n) = -1, \quad  U_{n+1} < 0, \\
  |U_{n+2}| \les n^2 |U_{n+1}|, \quad \b U_{n+1} - U_{n+2} > 0.
 \ega
\]
Plugging $ \b U_{n+1} - U_{n+2} > 0$ in \eqref{eq:bar_loc_valid_pf}, we prove  \eqref{eq:bar_loc_valid_a}.

Similarly, for $\kp \in (n , n +1/2)$, we get 
\beq\label{eq:bar_loc_valid_pf2}
 \BB_\sn{n}^\mw{ne}(Y) - U(Y)
=  ( \b U_{n} - U_{n+1} ) Y^{n+1} +  O_{n, \kp ,\b}(Y^{n+2}).
\eeq
Since $U_n > 0$ from \eqref{eq:U_sign} in Lemma \ref{lem:asym}, using \eqref{eq:induc_asym2}, for $ \b < - C n^2$ with large $C$, we get 
\[
\bga
 |U_{n+1}| \les n^2 |U_{n}|, \  \b U_{n} - U_{n+ 1} < \b U_n / 2 <  0.
 \ega
\]
Plugging $ \b U_{n} - U_{n+ 1} < 0 $ in \eqref{eq:bar_loc_valid_pf2}, we prove  \eqref{eq:bar_loc_valid_b}.

For \eqref{eq:bar_far_validb}, since we construct $B_l^f(Y), B_u^f(Y)$ agreeing with $U(Y)$ at $Y = 0$ with error $O(Y^2)$, using \eqref{eq:bar_far_valida}, we prove \eqref{eq:bar_far_validb}. 
\end{proof}


\subsection{Intersection between local and far-field barriers}\label{sec:Qs_QO_inter}


In this section, we estimate the location where the local and far-field barriers intersect, and then show that the intersection occurs within the region where the local barriers are valid.

\begin{prop}\label{prop:bar_inter}

Let $\mu_i, n_3$ be the parameters chosen in Proposition \ref{prop:bar_loc}. 
For any $n, \b$ with $n > n_4$, where $n_4> n_3$ is some large parameter, and $\b$ satisfying the assumption in Proposition \ref{prop:bar_loc}, there exists $\e_{n, \b}, \mu_i^{\prime}> 0$ such that the following statements hold true.
\begin{enumerate}[(a)]
\item
 For any $ \kp \in (n+1 - \e_{n, \b}, n+1)$, there exists $Y_I > 0$ satisfying 
 \[Y_I \leq {\mu}_{n+1}^{\pr} |\kp-n-1|^{1/(n-1) } < \f{1}{2} {\mu}_{n+1} |\b|^{-1}\,,\] such that 
\beq\label{eq:bar_inter_a}
 B^{f}_u(Y) < \BB^\mw{ne}_\sn{n+1}(Y)  <   B^{f}_l(Y) , \ Y \in (0, Y_I),
 \quad  B^{f}_u( Y_I) = \BB^\mw{ne}_\sn{n+1}( Y_I).  
\eeq

\item For any $ \kp \in (n , n + \e_{n, \b} )$, there exists $Y_I > 0$ with  
\[ Y_I \leq \mu_n^{\pr}  |\kp - n|^{1 / (n-2)}  < \f{1}{2} {\mu}_n |\b|^{-1}\,,\] such that 
\beq\label{eq:bar_inter_b}
  B^{f}_u(Y) <  \BB^\mw{ne}_\sn{n}(Y) <  B^{f}_l(Y)  , \ Y \in (0, Y_I),
 \quad  B^{f}_l( Y_I) = \BB^\mw{ne}_\sn{n}( Y_I).  
\eeq
\end{enumerate}


\end{prop}

\begin{proof}

By definitions of $F = B_l^f, B_u^f, \BB_\sn{n}^\mw{ne}$ \eqref{eq:bar_far}, \eqref{eq:bar_loc}, we have 
\[
 F(0) = U_0, \quad F^{\pr}(0) = U_1 .
\]
Hence, these functions agree at $Y = 0$ up to $O(Y^2)$. Moreover, from Proposition \ref{prop:bar_loc_valid}, we have 
\[
  \pa_Y^2 B^{f}_u(0) <  \pa_Y^2 \BB^\mw{ne}_\sn{n+1}(0) = U_2,
 \quad  U_2 = \pa_Y^2 \BB^\mw{ne}_\sn{n}(0) <  \pa_Y^2  B^{f}_l(0 ) .
\]

Below, we focus on the proof of  \eqref{eq:bar_inter_a}. For $ \kp \in (n+ \f{1}{2} , n+1)$, from \eqref{eq:induc_asym2}, \eqref{eq:U_sign}, we get 
\beq\label{eq:bar_inter_pf1}
U_n > 0, 
\ U_{n+1} < 0,  \ |U_{n+1}| \asymp_n |\kp-n-1|^{-1}, \ |U_{n+2}| \les n^2 |U_{n+1}|. 
\eeq
Using the assumptions of $\b < - n^2 , \mu_n < 1$ in Proposition \ref{prop:bar_loc}, for $ Y \in (0, \f{1}{2 |\b| }\mu_{n+1} )$, we obtain
\[
 0 < Y < \f{1}{2}\mu_{n+1} |\b|^{-1} < \f{1}{2}|\b|^{-1} < \f{1}{2 n^2} , 
 \quad |\b Y | < \f{1}{2},
\]
which along with \eqref{eq:bar_inter_pf1} imply 
\beq\label{eq:bar_inter_pf2}
\bal
\BB_\sn{n+1}^\mw{ne}(Y) - B_u^f(Y) 
 & \leq \sum_{i=2}^{n} U_i Y^i + U_{n+1} Y^{n+1} + \b U_{n+1} Y^{n+2}
 - \pa_Y^2 B_u^f(0 ) Y^2 + C Y^3 \\
 & \leq (U_2 - \pa_Y^2 B_u^f(0 ) ) Y^2 
 + C_n Y^3 + U_{n+1} Y^{n+1}+ \b U_{n+1} Y^{n+2} \\
 & \leq (U_2 - \pa_Y^2 B_u^f(0 ) ) Y^2 
 + C_n Y^3 + \f{1}{2} U_{n+1} Y^{n+1} .
\eal
\eeq

Since $a = U_2 - \pa_Y^2 B_u^f(0 ) > 0$, choosing $Y_*$ with 
\[
Y_*^{n-1} = 4 a  |U_{n+1}|^{-1},
\] 
we obtain 
\[
\BB_\sn{n+1}^\mw{ne}(Y_*) - B_u^f(Y_*)  \leq (U_2 - \pa_Y^2 B_u^f(0 ) + C_n U_{n+1}^{-1/(n-1)}- 2 a ) Y_*^2 
= ( C_n U_{n+1}^{-1/(n-1)}-  a ) Y_*^2 .
\]
Using the estimate of $U_{n+1}$ in \eqref{eq:bar_inter_pf1} and requiring $|\kp - n- 1| <\e_{n, \b}$ small enough, 
we get
\[
\BB_\sn{n+1}^\mw{ne}(Y_*) - B_u^f(Y_* )  < 0.
\]

Using this estimate, \eqref{eq:bar_loc_valid_a}, and continuity, we obtain that $\BB_\sn{n+1}^\mw{ne}(Y)$ and $ B_u^f(Y)$ intersect at some $Y_I \in (0, Y_*)$. We assume that $Y_I$ is the intersection with the smallest value in $(0, Y_*)$. The smallest of $Y_I$ and \eqref{eq:bar_loc_valid_a} imply the first inequality in \eqref{eq:bar_inter_a}. Moreover, we have
\beq\label{eq:bar_inter_pf3}
0 < Y_I < Y_* \les_n | \kp -n-1|^{1/ (n-1)} .
\eeq

For a fixed $\b$, by choosing $| \kp -n-1|$ small enough, we can ensure that $Y_I < \f{1}{2} \mu_n \b^{-1}$. Moreover, since $U_{n+1}< 0$, from \eqref{eq:bar_far_valida}, we obtain. 
\[
U_2 - \pa_Y^2 B_u^f(0) < \pa_Y^2 B_l^f(0) - \pa_Y^2 B_u^f(0) . 
\]
Thus, by further requiring $|\kp -n-1|$ small, which leads to a smaller upper bound for $Y_I$ in \eqref{eq:bar_inter_pf3}, and using \eqref{eq:bar_inter_pf2}, we establish
\[
\BB_\sn{n+1}^\mw{ne}(Y) - B_u^f(Y) 
\leq (U_2 - \pa_Y^2 B_u^f(0 ) ) Y^2 
 + C_n Y^3
< B_{l}^f(Y) - B_u^f(Y)  ,
\]
for $Y \in (0, Y_I)$. Rearranging the inequality, we prove the second inequality in 
\eqref{eq:bar_inter_a} and the result (a) in Proposition \ref{prop:bar_inter}.

The result (b) is proved similarly using the property that $U_n > 0, |U_n| \asymp_n |\kp -n|^{-1}$ in such a case. 
\end{proof}

Next, we show that the solution must exit the region $\Om_{B}^f$ \eqref{eq:bar_region} via the 
boundaries $E_{B, i}$ \eqref{eq:bar_region_edge}

\begin{prop}\label{prop:bar_exit}
For any $n > n_4$ with $n_4$ chosen in Proposition \ref{prop:bar_inter}, we have the following results. For $i = 1, 2$, there exists $\kp_i \in (n, n+1 )$, such that the local smooth solution $U^{(\kp_i)}(Y)$ starting at $Q_s$ with parameter $\kp_i$ first intersects $\pa \Om_B^f$ at $E_{B, i}$. 

\end{prop}

\begin{proof}

From Proposition \ref{prop:bar_loc_valid} and the property that $\BB_\sn{n}^\mw{ne}, U, \BB_\sn{n+1}^\mw{ne}$ 
\eqref{eq:bar_loc} agree at $Y =0$ up to error $O(Y^n)$, we know that 
\[
  U(Y),  \ \BB_\sn{n}^\mw{ne}(Y), \ \BB_\sn{n+1}^\mw{ne}(Y)
\]
remain in $\Om_B^f$ \eqref{eq:bar_region} for $0 < Y <\d^{\pr}$ with small $\d^{\pr}$. 

We focus on the proof of the case with $i = 1$. We fix $ n > n_4$ and choose $\b < 0$ with 
\beq\label{eq:bar_beta}
 2 \mu_{n+1}^{\pr} <  \mu_{n+1}  |\b|^{1 / (n-1)} , 
 \quad |\b| > C n^2,
\eeq
where $C, \mu_{\cdot}$ are the parameters chosen in Proposition \ref{prop:bar_loc}, 
and  $\mu^{\pr}_{\cdot}$ in Proposition \ref{prop:bar_loc_valid}, respectively. From Proposition \ref{prop:bar_loc_valid} and the above discussion, we have 
\beq\label{eq:bar_pf1}
 B_u^f(Y) < U(Y) <\BB_\sn{n+1}^\mw{ne}(Y) 
\eeq
for $Y \in (0, \d^{\pr} )$ with $\d^{\pr} = \d^{\pr}(n, \b, \kp)$ sufficiently small. Since $\D_Y > 0 $ in $\Om_B^f$ from Proposition \ref{prop:bar_loc_valid}, and $\PP_\sn{n+1}^\mw{ne}(Y) > 0$ for $ \kp \in (n+1/2, n+1)$ and $Y \in (0, Y_\mw{bar} ]$ with 
\[
Y_\mw{bar} \teq \mu_{n+1} \min( | |\b| (\kp-n-1)|^{1/{(n-1)}}, |\b|^{-1} ), 
\]
using a barrier argument, we obtain that 
\beq\label{eq:bar_pf2}
U(Y) < \BB_\sn{n+1}^\mw{ne}(Y) ,\quad Y \in (0, Y_\mw{bar}). 
\eeq

Next, we choose $\kp $ in the range of $(n+1 - \e_{\b, n}, n+1)$ defined in Proposition \ref{prop:bar_inter}. From the choice of $\b$ \eqref{eq:bar_beta} and the inequality of $Y_I$ in 
Proposition \ref{prop:bar_inter}, we get 
\[
Y_I \leq  \min( {\mu}_{n+1}^{\pr} |\kp-n-1|^{1/(n-1) } , \f{1}{2} \mu_{n+1} |\b|^{-1} )
< \f{1}{2} Y_\mw{bar}. 
\]

Due to \eqref{eq:bar_pf1} and the fact that $\BB_\sn{n+1}^\mw{ne}(Y)$ intersects $B_u^f(Y)$ for $Y$ within the validity of barrier \eqref{eq:bar_pf2}, the solution $U(Y)$ must intersect $B_u^f(Y)$ at some $0 < Y < Y_I$. Using \eqref{eq:bar_inter_a} in Proposition \ref{prop:bar_inter}, we get
\[
U(Y) < \BB_\sn{n+1}^\mw{ne}(Y) \leq B_l^f(Y),\  Y \in (0, Y_I].
\]
Thus, $U(Y)$ cannot intersect $B_l^f(Y)$ for $Y \in (0, Y_I]$. We conclude the proof in the case of $i=1$. 

The proof of the case of $i=2$ is similar and is omitted.
\end{proof}

\subsection{Proof of Proposition \ref{prop:Qs_QO}}\label{sec:shoot_Qs_QO}

In this section, we first construct the smooth solution $V(Z)$ to the ODE \eqref{eq:ODE} near $Z = 0$. Then we use a shooting argument to glue the curve $(Z, V(Z))$ under the map $(\cY, \cU)$ 
 \eqref{eq:sys_UY} and the smooth solution starting from $Q_s$ and prove Proposition \ref{prop:Qs_QO}.

We have the following result from \cite[Proposition 3.3]{shao2024self} and its proof. 
\begin{prop}[Proposition 3.3, \cite{shao2024self}]\label{prop:V_near0}
Let $Z_0$ be the coordinate of the sonic point in \eqref{eq:root_P}. The ODE \eqref{eq:ODE} has a unique solution $ V_F^{(\kp)} \in C^{\infty}( [0,Z_0))$ with $V_F^{(\kp)}(0) = 0 $. Moreover, 
near $Z=0$, it has a power series expansion 
\footnote{
  In \cite{shao2024self}, the authors rewrote the ODE \eqref{eq:ODE} as an ODE for $\Phi = V / Z$ and $\varkappa = Z^2$ and expand $\Phi$ as a power series of $\varkappa$. In particular, the smooth local solution $V$ can be written as $V(Z) = Z g(Z^2)$ for some smooth functions $g$.
}
\beq\label{eq:V_near0}
V_F^\rn{\kp}(Z) = \sum_{ i \geq 0} V_i Z^i, \quad V_{2i} = 0, 
\quad 
  |V_{2i+1} | \leq \mfr C_i C^i, \ 
   \forall \ i \geq 0,
\eeq
where $\mfr C_i$ is the Catalan number \eqref{eq:cat_num} and $C$ is some absolute constant. \footnote{
Although, the constant $C$ depends on the parameters $p, d, l, \g$, since we restrict $p, d,l, \g$
 \eqref{eq:para3} to specific ranges, following the same estimates as those in \cite{shao2024self}
lead to an absolute constant.
}
\end{prop}

We only need the local existence of $V_F^\rn{\kp}(Z)$ for $ Z\in[0, \d_1]$ with some $\d_1 > 0$. The proof is standard and also follows from the argument in \cite[Section 2]{buckmaster2022smooth} by bounding the power series coefficients. The power series of $V_F^\rn{\kp}$ only contains the odd power $Z^{2i+1}$ due to symmetry.

Using the map $(\cY, \cU)$ \eqref{eq:sys_UY} from $(Z, V)$ to $(Y, U)$ coordinates and Proposition \ref{prop:V_near0}, we construct  
\beq\label{eq:QO_glue}
(Y_F^\rn{\kp}, U_F^{(\kp)} )(Z)
= (\cY, \cU) ( Z, V_F^{(\kp)}(Z)), \quad Z \in [0, \d ).
\eeq
where $F$ is short for \textit{far}. See the orange curve in Figure \ref{fig:coordinate2} for an illustration of $(Y_F^\rn{\kp}, U_F^{(\kp)} )(Z)$. 

We have the following asymptotics of $U_F^{(\kp)}(Z),Y_F^\rn{\kp}(Z)$ for small $Z$.
\begin{lem}\label{lem:UF_asym}
There exists $C_1$ sufficiently large and $\d_3 > 0$ sufficiently small, such that for any  $\kp> C_1$ and $Z \in (0, \d_3]$, the functions $Y_F^\rn{\kp}(Z), U_F^{(\kp)}(Z)$ defined in \eqref{eq:QO_glue}
satisfies 
\bseq\label{eq:UF_asym}
\begin{gather}
  0 < c_1 Z^2 \leq  Y_O - Y_F^\rn{\kp}(Z)  \leq c_2  Z^2 ,
  \qquad Y_F^\rn{\kp}(Z)  \in (0, 1) , 
  \label{eq:UF_asyma}  \\
\B| U_F^{(\kp)}(Z)  - \f{ C_{\infty}}{ Y_F^\rn{\kp}(Z)- Y_O} \B| \les 1, 
\qquad U_F^{(\kp)}(Z) > 0,
\label{eq:UF_asymb} 
\end{gather}
\eseq
where the implicit constants, e.g. $c_1, c_2$, are uniform in $\kp$, $C_{\infty}$
is given by
\[
    C_{\infty} = -(\g+1)^3(V_3 - V_1^2 + V_1^3) \neq 0,     
\]
and $V_1, V_3$ are the power series coefficients in \eqref{eq:V_near0} given by
\beq\label{eq:V_formula}
V_1 = \f{d-1}{d (\g+1)},
\quad V_3= \f{1}{d+2} \B( (d-1)\B(  -  \f{2 }{\g+1}  V_1^2 + V_1^3 + V_1^2 \B) + V_1(2 V_1 + \ell (V_1 - 1)^2) \B).
\eeq
\end{lem}

The proof follows from expanding $(Y, U_F^{(\kp)})(Z)$ near $Z=0$ and using the map \eqref{eq:sys_UY}, which is elementary but tedious. We defer it to Appendix \ref{app:UF_asym}.

Recall the barrier functions $B_l^f(Y), B_u^f(Y)$ from \eqref{eq:bar_far_l}, \eqref{eq:bar_far_u}. We have the following results.

\begin{lem}[Computer-assisted]\label{lem:limit}
We have 
\[
\lim_{Y \to Y_O} B_l^f(Y) \cdot (Y- Y_O) = \f{e_1 Y_O + e_2 Y_O^2 - U_0}{d} < C_{\infty}.
\]
\end{lem} 

We obtain the limit by definition of $B_l^f$ \eqref{eq:bar_far_l}.  We verify the scalar inequality in Lemma \ref{lem:limit} directly using Interval arithmetic. 

Since $Y - Y_O < 0$ and $B_u^f(Y)$ is uniformly bounded for $Y$ near $Y_O$, using Lemma \ref{lem:limit} and Lemma \ref{lem:UF_asym}, we obtain that
there exists $\d_4 \in (0, \d_3) $ with $\d_3$ chosen in Lemma \ref{lem:UF_asym} such that 
\beq\label{eq:UF_asym2}
  B_l^f(Y_F^\rn{\kp}(Z)) >  U_F^{(\kp)}(Z) > B_u^f(Y_F^\rn{\kp}(Z)) .
\eeq
for any $Z \in (0, \d_4]$ and any $\kp \in (n,n+1)$.

\subsubsection{Shooting argument}\label{sec:shoot_sub}

Denote by $U^{(\kp)}$ the local analytic function near $Q_s$ constructed in Proposition \ref{prop:analy}. We prove Proposition \ref{prop:Qs_QO} using a shooting argument. Let $n_4$ be the large parameter determined in Propositions \ref{prop:bar_inter}. We fix $n > n_4$ and consider $\kp \in (n, n+1)$. 

Recall the region $\Om_B^f$ from \eqref{eq:bar_region}. From \eqref{eq:bar_far_validd}, in Proposition \ref{prop:bar_loc_valid}, for any $\kp \in (n, n+1)$, we have $U^{(\kp)}(Y) \in \Om_B^f$ for $ 0 < Y \ll 1$. Moreover, since $\D_Y >0$ in $\Om_B^f$ from Proposition \ref{prop:bar_loc_valid}, the solution curve $(Y, U(Y))$ either remains in $\Om_B^f$ for all $Y \in [0, Y_O)$ or exits the region $\Om_B^f$ via one of the edges $E_B^i$ \eqref{eq:bar_region_edge}.
We define the extension $U_{\mw{ext}}^{(\kp)}(Y)$ of $U^{(\kp)}$ according to one of two cases.

(a) If $U^{(\kp)}(Y) \in \Om_B^f$ for $Y \in [0, Y_O)$, we define
\bseq\label{eq:U_ext}
\beq\label{eq:U_exta}
U_{\mw{ext}}^{(\kp)}(Y) = U^{(\kp)}(Y), \quad Y \in [0, Y_O).  
\eeq

(b)
Otherwise, suppose that $(Y, U(Y))$ first intersects $\Om_B^f$ at $(Y_*, U^{(\kp)}(Y_*) ) \in E_{B, i}$ with $Y_* < Y_O$ for $i \in \{1, 2\}$. We define 
\beq
\bal
& U_{\mw{ext}}^{(\kp)}(Y) = 
\begin{cases}
U^{(\kp)}(Y),  \quad & Y \in [0, Y_*], \\
B_{\al}^f(Y), \quad  & Y \in [Y_*, Y_O) , 
\end{cases}
\eal
\eeq
\eseq
with $B_{\al}^f(Y) = B_u^f(Y)$ if $i =1$ and $B_{\al}^f(Y) = B_l^f$ if $i = 2$. That is, we extend $U^{(\kp)}(Y)$ using the barrier function $B_l^f$ or $B_u^f$.

Since $(Y, U_{\mw{ext}}^{(\kp)}(Y))$ is in the closure of $\Om_B^f$, where we have $\D_Y(Y, U) >0$ if $Y > 0$, using the continuity of the ODE solution to \eqref{eq:UY_ODE} in $\kp, Y$, we get that $ U_{\mw{ext}}^{(\kp)}(Y)$ is continuous in $(Y, \kp) \in [0, Y_O) \times (n, n+1)$. 

From Proposition \ref{prop:bar_exit}, for $i = 1,2$, there exists $\kp^*_{i} \in (n, n+1)$ and  $Y_i \in (0, Y_O)$ such that:
\begin{enumerate}[i)]
    \item The solution $(Y, U^{(\kp^*_1)}(Y))$ first exit $\Om_B^f $ at 
$(Y, B_l^f(Y))$ with $Y = Y_1 \in (0, Y_O)$.
\item The solution $(Y, U^{(\kp^*_2)}(Y))$ first exit $\Om_B^f $ at 
$(Y, B_u^f(Y))$ with $Y = Y_2 \in (0, Y_O)$.
\end{enumerate}

Since $Y_1, Y_2 < Y_O$, using the estimate of $ Y_F^\rn{\kp} $ \eqref{eq:UF_asyma}, we can choose $\d_Y > 0$ small enough and independent of $\kp \in (n, n+1)$  such that
\beq\label{eq:glue_delY}
0< \d_Y < \d_4, \quad  Y_F^\rn{\kp}(\d_Y) > \max(Y_1, Y_2), \quad  Y_F^\rn{\kp}(\d_Y) < Y_O.
\eeq
We define 
\[
  g( \kp) =  U_{\mw{ext}}^{(\kp)}(  Y_F^\rn{\kp}(\d_Y) )  - U_F^{(\kp)}( \d_Y).
\]

Using the continuity of $U_{\mw{ext}}^{(\kp)}, U_F^{(\kp)}$, we obtain that $g(\kp)$ is continuous in $\kp$. Using the definition of $(Y_i , \kp^*_i)$, the bound \eqref{eq:UF_asym2}, 
and \eqref{eq:glue_delY}, 
we get
\[
g(\kp_1^*) = B_u^f( Y_F^\rn{\kp_1^* }(\d_Y)) - U_F^{( \kp_1^* )}( \d_Y)  < 0,
\quad g(\kp_2^*) = B_l^f( Y_F^\rn{\kp_2^*}(\d_Y)) - U_F^{(\kp_2^*)}( \d_Y)    > 0. 
\]
Using continuity, we obtain $g(\kp^*) = 0$ for some $\kp^*$ between $\kp^*_1, \kp^*_2$, which implies $ U_{\mw{ext}}^{(\kp^*)}( Y_F^\rn{\kp^*}(\d_Y)) = U_F^{(\kp^*)}( \d_Y) $. Using the bound \eqref{eq:UF_asym2}
for $U_F$, we yield 
\[
 B_u^f( Y_F^\rn{\kp^*}(\d_Y) ) < U_{\mw{ext}}^{(\kp^*)}(Y_F^\rn{\kp^*}(\d_Y)) < B_l^f( Y_F^\rn{\kp^*}(\d_Y) ).
\]

From the definition of $U_{\mw{ext}}^{(\kp^*)}$ in \eqref{eq:U_ext} and the above bound, we obtain that $U_{\mw{ext}}^{(\kp^*)}$ is defined via \eqref{eq:U_exta} and thus $ U^{(\kp^*)}( Y_F^{(\kp^*)}(\d_Y) ) = U_{\mw{ext}}^{(\kp^*)}( Y_F^{(\kp^*)}(\d_Y) ) =U_F^{(\kp^*)}( \d_Y ) $. Using the relation \eqref{eq:QO_glue} and inverting the maps \eqref{eq:invert}, we obtain
\beq\label{eq:glue_ODE0}
  ( \d_Y, V_F^\rn{\kp^*}(\d_Y) )
 = (\cZ, \cV) (Y_F^{\kp^*}, U_F^\rn{\kp^*}) (\d_Y) 
 = (\cZ, \cV)( Y_F^{\kp^*}(\d_Y), U^{\rn{\kp^*}}( Y_F^{\kp^*}(\d_Y)))
\eeq

\subsubsection{Gluing the solution}\label{sec:glue}


We construct solution $V(Z)$ to the ODE \eqref{eq:ODE} using the maps \eqref{eq:UY_to_VZ} from $U^\rn{\kp^*}(Y)$ and then glue it with $V_F^\rn{\kp^*}$ obtained above.

We use the map \eqref{eq:UY_to_VZ} to construct the curve
\beq\label{eq:glue_ODE1}
(Z^\rn{\kp^*}, V^\rn{\kp^*}) (Y) = (\cZ, \cV)(Y, U^\rn{\kp^*}(Y) ), \quad  Y \in J, 
\quad J = [-c, Y^\rn{\kp^*}(\d_Y)],
\eeq
for some $0 < c \ll 1$. Using \eqref{eq:glue_ODE1} with $Y = Y^\rn{\kp^*}(\d_Y)]$
and \eqref{eq:glue_ODE0}, we get 
\beq\label{eq:glue_ODE1b}
Z^\rn{\kp^*}(
Y^\rn{\kp^*}(\d_Y)) =\d_Y, 
\quad 
 V^\rn{\kp^*}(  Y^\rn{\kp^*}(\d_Y)))  =  V_F^\rn{\kp^*}( \d_Y).
\eeq

Using the second identity in \eqref{eq:Mv} and the chain rule, we obtain 
\beq\label{eq:glue_ODE2}
 \f{ \f{ d  V^\rn{\kp^*}}{d Y} }{ \f{d Z^\rn{\kp^*}}{d Y} } =
\f{ \pa_Y \cV + \pa_U \cV \f{\D_U}{\D_Y}  }{  \pa_Y \cZ + \pa_U \cZ \f{\D_U}{\D_Y}  }
= \f{ m^{-1} \D_V}{ m^{-1} \D_Z} =
  \f{\D_V}{\D_Z} (Z^\rn{\kp^*}(Y), V^\rn{\kp^*}(Y) ).
\eeq

Next, we show that $Z^\rn{\kp^*}(Y)$ is strictly decreasing and invertible. Using \eqref{eq:Mv}, we get 
\beq\label{eq:glue_dZ}
\bal
   \frac{dZ^\rn{\kp^*}}{dY}(Y)
  & =  \pa_U \cZ  \cdot \f{d U^\rn{\kp^*} }{d Y} + \pa_Y \cZ  
  =  \pa_U \cZ  \cdot \f{\D_U}{\D_Y} + \pa_Y \cZ \\
  & =  m^{-1}(Y, U) \f{\D_Z( Z^\rn{\kp^*}, V^\rn{\kp^*})}{\D_Y(Y, U)} \B|_{ U = U^\rn{\kp^*}(Y)  }.
  \eal
\eeq
Recall, in Section \ref{sec:shoot_sub}, we showed $(Y, U^\rn{\kp^*}(Y) ) \in \Om_B^f$ for $Y \in (0, Y_F^\rn{\kp^*} (\d_Y)]$; hence, from  \eqref{eq:DelY_sign} and \eqref{eq:DelZ_sign} we obtain the inequalities  $\D_Z, \D_Y>0$.  Using these sign inequalities, together with $m \neq 0$ \eqref{eq:Mv}, yields
\[
 \frac{dZ^\rn{\kp^*}}{dY}(Y) \neq 0 , \quad \mw{ for} \quad  Y \in (0, Y_F^\rn{\kp^*} (\d_Y)]. 
 \]

Using \eqref{eq:dZ_dY} and the definition of $ Z^\rn{\kp^*}(Y)
= \cZ(Y, U^\rn{\kp^*}(Y))$ \eqref{eq:glue_ODE1}, we get $(Z^\rn{\kp^*})^{\pr}(0) < 0$. Using continuity, we obtain 
\[
\frac{dZ^\rn{\kp^*}}{dY}(Y) < 0,  \quad \mw{ for} \quad    Y \in J^{\pr} \teq [ -c^{\pr}  ,  Y^\rn{\kp^*}(\d_Y)], 
\]
with $0 < c^{\pr}  \ll 1 $ and $c^{\pr} < c$. Thus, $Z^\rn{\kp^*}$ is strictly decreasing, 
invertible, and smooth on $J^{\pr}$. 
We define the inverse as 
$(Z^\rn{\kp^*})^{-1}$ and construct a solution 
\beq\label{eq:glue_ODE3}
V_\mw{ODE}^\rn{\kp^*}(Z) = V^\rn{\kp^*}( (Z^\rn{\kp^*})^{-1}(Z)  ), 
\quad Z \in [ \d_Y, Z^\rn{\kp^*}(-c)] .
\eeq
Using \eqref{eq:glue_ODE3}, \eqref{eq:glue_ODE2}, and the chain rule, we obtain a smooth solution $V_\mw{ODE}^\rn{\kp^*}$ to the ODE \eqref{eq:ODE}. Using \eqref{eq:glue_ODE1b} and \eqref{eq:glue_ODE0}, we get 
\[
V_\mw{ODE}^\rn{\kp^*}(\d_Y)
=  V^\rn{\kp^*}( (Z^\rn{\kp^*})^{-1}(\d_Y)) 
= V^\rn{\kp^*}(  Y^\rn{\kp^*}(\d_Y))) =  V_F^\rn{\kp^*}( \d_Y).
\]

Since both $V_\mw{ODE}^\rn{\kp^*}(Z), V_F^\rn{\kp^*}(Z)$ solve the ODE \eqref{eq:ODE} smoothly with the same data at $Z= \d_Y$, $ V_F^\rn{\kp^*}(Z)$ is smooth in $[0, \d_1]$ covering $Z = \d_Y$ (see Proposition \ref{prop:V_near0} and its following discussion), using the uniqueness, we obtain $V_\mw{ODE}^\rn{\kp^*} = V_F^\rn{\kp^*} $ and construct a smooth ODE solution $V(Z)$ to \eqref{eq:ODE} on $[0, Z^\rn{\kp^*}(-c) ]$ with $Z^\rn{\kp^*}(-c) > Z^\rn{\kp^*}(0) = \cZ(0, U_0)
= Z_0$ \eqref{eq:root_P}. Applying \eqref{eq:glue_ODE3} and \eqref{eq:glue_ODE1} with $Y = (Z^\rn{\kp^*})^{-1}(Z)$, we prove \eqref{eq:sonic_agree} for some $\e_1>0$.

\subsubsection{Proof of other properties}\label{sec:other_prop}
Since $(\cZ, \cV)$ map the sonic point $(0, U_0)$ to $(Z_0, V_0)$ in the $(Z, V)$ system, using \eqref{eq:glue_ODE1} with $Y = 0$ and \eqref{eq:glue_ODE3}, we get $V(Z_0) = V_0$. 
Since $V(Z) = V_F^\rn{\kp^*}(Z)$ for small $Z$ is constructed by Proposition \ref{prop:V_near0}, we have $V(0) = 0$ and $V(Z) = Z g(Z^2)$ for some $g \in C^{\infty}
([0, Z_0 + \e_1])$ with small $\e_1>0$.

To prove \eqref{eq:Qs_QO_prop1}, using Lemma \ref{lem:bijec} and the property 
that $\cV(Y, U) < \cZ(Y, U)$ \eqref{eq:sys_UY} with $(Y, U) \in \cR_{YU}$ is equivalent to 
\[
 (1 - Y) ( \f{1}{1 + \g} U + 1 - Y ) <  U + (1- Y)^2  \iff Y > - \g,
\]
we only need to estimate the $(Y, U)$ coordinate of the solution curve:
\beq\label{eq:ODE_include}
(\cY, \cU) ( Z, V(Z)) \in D \teq \{ (Y, U) : -\g < Y < 1, U > 0 \},
\quad \mw{for} \ Z \in [0, Z_0 + \e_1],
\eeq
with some $\e_1 > 0$. From Lemma \ref{lem:UF_asym} and the definitions of  $\Om_B^f$ \eqref{eq:bar_region} and $B_u^f$ \eqref{eq:bar_far_u}, we have 
\[
\bal
(Y_F^\rn{\kp_*}  (Z), U_F^\rn{\kp_*}(Z)) & \in (-\g, 1) \times [U_0, \infty) 
\subset D ,  \\
(Y,   U^\rn{\kp_*}(Y)) & \in \bar \Om_B^f  \subset [0, 1) \times [U_0, \infty) \subset D ,
\eal
\]
for $Z \in [0, \d_3]$ and $Y \in [0, Y^\rn{\kp^*}(\d_Y)] $. From the relation \eqref{eq:QO_glue} and \eqref{eq:glue_ODE1} with $Z^\rn{\kp^*}(0) = Z_0$ and \eqref{eq:glue_ODE1b}, we prove \eqref{eq:ODE_include} for $Z \in [0, Z_0]$. Using continuity and by choosing $\e_1> 0$ small enough, we prove \eqref{eq:ODE_include} for $Z \in [0, Z_0 + \e_1]$. We conclude the proof of Proposition 
\ref{prop:Qs_QO}.

\section{Lower part of $Q_s$}\label{sec:Q_lower}

In this section, we consider odd $n$ and any $\kp \in (n, n+1)$.  We use a barrier argument to show that the local analytic solution $(Y, U^{(\kp)}(Y))$ constructed in Proposition \ref{prop:analy} can be continued for $Y < 0$ and cross the curve $\D_Y =0$ (red curve) below the point $Q_s$ (see Figure \ref{fig:coordinate2}). 
We justify these in Propositions \ref{prop:bar_final}, \ref{prop:inter_final}, \ref{prop:inter_U_root}. Since the curve $\D_Y = 0 $ below $Q_s$ in the system \eqref{eq:ODE_UY} is a upper barrier and $U=0$ is another barrier, the solution curve must further cross $Y = 0$ \eqref{eq:sys_UY} (see Figure \ref{fig:coordinate2} for an illustration), 
which corresponds to the curve of $\D_V(Z, V) = 0$  in the original system  \eqref{eq:ODE} (see the red curve  between $P_2, P_s$ in Figure \ref{fig:coordinate1}). 
We justify it in Proposition \ref{prop:ds_ODE}. Afterward, in Lemma \ref{lem:far_extend} in Section \ref{sec:proof}, we can extend the solution of the 
$(Z,V)$ \eqref{eq:ODE} to $Z = \infty$. This give rises to a smooth solution $V(Z)$ to the ODE \eqref{eq:ODE} for $Z \in C^{\infty}[Z_0 - \e_1, \infty)$ for some $\e_1\ll 1 $.

\subsection{The local upper barrier} 
 We introduce a local barrier 
\beq\label{eq:up_bar}
\GG_\sn{n} = \sum_{i =0 }^n U_i Y^i .
\eeq
where $U_i,i\geq 0$ are the power series coefficients constructed in Section \ref{sec:pow_sonic}. In Proposition \ref{prop:bar_final}, we show that $\GG_\sn{n}$ remains a valid local upper barrier of $U$ for $Y \in [- 2 (  C^* \kp)^{-1},  0)$, where $ C_*$ is defined in \eqref{eq:C_asym}. In Propositions \ref{prop:inter_final}, \ref{prop:inter_U_root}
we show that in the range $Y \in [- 2 (  C^* \kp)^{-1},  0)$, the local upper barrier and the solution must intersect the global upper barrier $U_{\D_Y}$. 



\begin{prop}\label{prop:bar_final}
There exists $C$ large enough such that for any $n$ odd with $n > C$ and $\kp \in (n, n+1)$, we have
\footnote{The reader should not confuse the polynomial $\PP_\sn{n}(Y)$ with the coefficient $P_n$ of $\PP_\sn{n}^\mw{ne}(Y)$ in \eqref{eq:bar_loc_est1a}.} 
\beq\label{eq:bar_final_cond}
 \PP_\sn{n}(Y) = \GG_\sn{n}^{\pr}(Y) \D_Y( Y, \GG_\sn{n}(Y) ) - \D_U(  Y, \GG_\sn{n}(Y) ) > 0
\eeq
for $Y \in [- 2 (  C^* \kp)^{-1},  0)$, where $ C_*$ is defined in \eqref{eq:C_asym}. 
\end{prop}

The proof relies on the estimates of $U_i$ in Lemmas \ref{lem:asym}, \ref{lem:asym_refine}.

\begin{proof}

Recall $\d = 0.05$ from \eqref{eq:para3}. For $q$ to be chosen, we define $\tau, q_i$ according to 
\bseq\label{eq:q_tau}
\beq\label{eq:q_taua}
\bal
 \f{\tau }{1 - \tau} & = 1 + 4 \d, \quad  q_1 = 2(1 + \d ) q  , \quad  q_2 =  (2 (1 + \d))^{ \tau } q^{1 - \tau} , \\
  q_3 & = \f{1 + \d }{1 + 2 \d} ,
\quad q_4 =  \f{ ( 1 - \d)(1 + 4 \d) }{1 + 2 \d}  .
\eal
\eeq

Since $\d = 0.05$ and $2 (1 + \d)^{1 + 4 \d} / 4 < 1$, by choosing $q \in (4^{-1}, 3^{-3})$ close to $4^{-1}$, we can obtain $\tau, q_i$ satisfying 
\beq\label{eq:q_taub}
\tau \in (\f{1}{2}, \f{3}{4}), 
\quad q_1 < 1, 
\quad q_2 =  ( (2 (1 + \d))^{1 + 4 \d} q )^{1 -\tau} < 1 , \quad 
q_3 < 1, \quad q_4 > 1.
\eeq

\eseq

Recall $C_*$ from \eqref{eq:C_asym}. We define 
\beq\label{eq:Y_thres}
\quad \th_2 =   ( (1 + 2 \d)  C_* \kp)^{-1},
\quad \th_3 =  2 ( C_* \kp)^{-1} .
\eeq

We estimate $\PP_\sn{n}(Y)$ for $Y \in [- \th_1, 0),  Y \in  [- \th_2, -\th_1], Y \in [-\th_3, -\th_2] $ separately.

\vspace{0.1in}
\paragraph{\bf{Expansion of $\PP_\sn{n}$}}

Denote by $a_m$ the $m$-th coefficient in the power series of $a(Y)$. Using the formula $ (\GG_\sn{n}^2(Y))^{\prime} = 2 \GG_\sn{n}^{\prime} \GG_\sn{n}$ and the definition of $\GG_\sn{n}$ \eqref{eq:up_bar}, 
we obtain 
\beq\label{eq:bar_iden_sq}
\bal
 (\GG_\sn{n}^{\prime} \GG_\sn{n})_m & = \f{ m + 1}{2} ( \GG_\sn{n}^2)_m , 
 \quad ( \GG^2_\sn{n})_{2n + 1} = 0, \quad  ( \GG_\sn{n}^2)_{2n } = U_n^2 ,
 \quad ( \GG_\sn{n}^2)_{2n-1} = 2 U_{n-1} U_n.
 \eal
\eeq

Next, we perform power series expansion for $\PP_\sn{n}$ \eqref{eq:bar_final_cond}. Denote by $P_{n, m}$ the $m-th$ coefficient of $Y^m$ in $\PP_\sn{n}$. Since $\GG_\sn{n}$ agree with the local solution $U(Y)$ up to $O(|Y|^{n+1})$, 
following the derivation of $P_{n+1}$ in \eqref{eq:bar_loc_lead} in the proof of Proposition \ref{prop:bar_loc}, we obtain that the term $P_{n, i} Y^i , i\leq n$ vanishes in $\PP_\sn{n}$ and the leading order term of $\PP_\sn{n}$ is given by 
\bseq\label{eq:bar_Pn_exp}
\beq\label{eq:bar_Pn_expa}
 P_{n, n+1} Y^{n+1} , \quad P_{n, n+1}=  ( (n+1) \lams - \laml )( 0 - U_{n+1}) .
\eeq

For the coefficient of $Y^i, i \geq n+2$, using the explicit formula of $\D_Y, \D_U$ \eqref{eq:UY_ODE}, we obtain 
\beq\label{eq:bar_Pn_expb}
\bal
\PP_\sn{n}(Y) & = Y^{n+1}( (n+1) \lams - \laml )( 0 - U_{n+1}) 
+ Y^{n+2} ( n  U_n  B - 2  U_n( B-(d-1) )  ) \\
& \quad + \sum_{ n+3 \leq m \leq 2 n}  Y^m ( d ( \GG_\sn{n}^{\pr} \GG_\sn{n} )_{m-1}
 - ( \GG_\sn{n}^{\pr} \GG_\sn{n} )_{m} - 2 ( \GG_\sn{n}^2)_m  ) .
 \eal
\eeq
\eseq
We note that the term $U(f(Y) + (d-1) Y(1-Y))$ in $\D_U$ and $(Y-1) f(Y)$ in $\D_Y$ \eqref{eq:UY_ODE} only contributes to the term $C_i Y^i, i \leq n+2$ in $\PP_\sn{n}$. 

Using the identity \eqref{eq:bar_iden_sq} and the expansion \eqref{eq:bar_Pn_expb}, we get 
\beq\label{eq:G_bar0}
\bal
& P_{n, m} = P_{n, m, 1}  (\GG_\sn{n}^2)_m - P_{n, m, 2} (\GG_\sn{n}^2)_{m+1}, \\
\quad 
& P_{n, m, 1} =  \f{ d m }{2} - 2  ,\quad P_{n, m, 2} = \f{m+1}{2}, 
\quad m \geq n + 2 .
\eal
\eeq

Plugging $\kp = \f{\laml}{\lams}$ \eqref{eq:kappa} to \eqref{eq:bar_Pn_expa}, we obtain 
\bseq\label{eq:G_bar1}
\beq
P_{n, n+1} = (\f{n+1}{\kp} - 1) \laml ( 0 - U_{n+1})
= - \f{n+1 - \kp }{ \kp } \laml  U_{n+1}. 
\eeq

For $\kp \in (n, n+1)$, applying $\laml >0$ \eqref{eq:lam_sign} and  the estimates of $U_n$ \eqref{eq:U_sign}, \eqref{eq:induc_asym2}, we further obtain $U_{n+1} < 0$ and 
\beq\label{eq:Pn_main_low}
P_{n, n+1} \geq (1 - \d) \laml (n+1)  C_* U_n > 0.
\eeq

\eseq 
Thus, for $n$ odd and $Y < 0$, we get $P_{n, n+1}  Y^{n+1} > 0$.

For the last two terms $P_{n, 2n} Y^{2n}, P_{n, 2n-1} Y^{2n-1}$ in the expansion \eqref{eq:bar_Pn_expb}, 
using \eqref{eq:bar_iden_sq}, \eqref{eq:G_bar0}, and then the estimates $U_n \gtr U_{n-1} n^2$ \eqref{eq:induc_asym2}, \eqref{eq:U_sign} (with $\kp \in (n,n+1)$), we get 
\beq\label{eq:G_bar_top}
\bal
 P_{n, 2n} &= P_{n, 2n, 1} (\GG_\sn{n}^2)_{2n} 
= P_{n, 2n, 1} U_n^2 > 0 , \\
 P_{n, 2n-1} &=  P_{n, 2n-1, 1} ( \GG_\sn{n}^2)_{2n-1}
 - P_{n, 2n-1, 2} ( \GG_\sn{n}^2)_{2n}  \\
 & = P_{n, 2n-1,1} 2 U_n U_{n-1} - P_{n, 2n-1,2} U_n^2 \\&
=  - (1 + O(n^{-2})) P_{n, 2n-1,2} U_n^2  < 0.
\eal
\eeq
Hence, we obtain $Y^{2n} P_{n, 2n} , Y^{2n-1} P_{n, 2n-1} > 0$. 

\vspace{0.1in}
\paragraph{\bf{Ideas of the remaining estimates}}

It remains to estimate the term $P_{n,m}Y^m$ for $ 2 \leq  m - n \leq n-2$. We use the asymptotics for $U_i$ in Lemma \ref{lem:asym_refine} to show that $P> 0$. For $|Y|$ very small, we treat all the terms $ P_{n, m} Y^m$ as perturbation to $ P_{n, n+1} Y^{n+1}$. For $|Y|$ relatively large and $m$ close to $2n$, we treat $ P_{n, m} Y^m$ as perturbation to $
P_{n,2n-1} Y^{2n-1} +  P_{n, 2n} Y^{2n}$. We choose $l_1$ to be some large absolute constant, and perform the estimates in four cases: 
\[
m- n \leq l_1, \quad l_1 < m-n \leq n / 8, \quad 
n / 8 \leq m-n\leq \tau n + 2, 
\quad  \tau n \leq m-n \leq n-2.
\]

\vspace{0.1in}
\paragraph{\bf{Case I: $ 2 \leq m - n \leq l_1$ }}

Using \eqref{eq:induc_asym2}, we obtain 
\[
 ( \GG_\sn{n}^2)_m = \sum_{i + j = m, i, j \leq n} U_i U_j 
 \les_{l_1} (m-n) U_n \les_{l_1} U_n,
 \quad  |P_{n, m} | \les_{l_1} n U_n. 
\]
Thus, for $m-n \leq l_1$, using the above estimate and the lower bound of $P_{n, n+1}$ \eqref{eq:Pn_main_low}, we obtain 
\beq\label{eq:G_bar2}
    \sum_{ n+2 \leq m \leq n+ l_1}  |P_{n, m} Y^m| 
 \leq C(l_1) P_{n, n+1} Y^{n+ 2} .
\eeq

By choosing $|Y| \ll_{l_1} 1$, we can treat it as perturbation to $P_{n, n+1} Y^{n+1}$.

For $ m \geq n + 3$, using \eqref{eq:Um_convex} in Lemma \ref{lem:asym_refine}, we obtain 
\bseq\label{eq:G_bar3}
\beq
  |( \GG_\sn{n}^2)_m | = \sum_{i+j =m, i, j \leq n} U_i U_j
  \les U_n U_{m-n} \sum_{l \geq 0} ( \f{3}{2})^{-l} \les U_n U_{m-n},
  \eeq
  which along with the decomposition of $P_{n, m}$ in \eqref{eq:G_bar0} implies
  \beq\label{eq:G_bar3b}
  \quad  |P_{n, m}| \les n U_n U_{m+1 - n}. 
\eeq
\eseq

\vspace{0.1in}
\paragraph{\bf{Case II: $l_1 <  m-n \leq  n/ 8$}}

Using \eqref{eq:Um_asym_ref2} with $l = 5$ and \eqref{eq:G_bar3b}, for $|Y| \leq \th_3 $ with $\th_3= 2 ( C_* \kp)^{-1} $ \eqref{eq:Y_thres}, we obtain 
\beq\label{eq:G_case1}
\bal
 |Y^{m} P_{n, m} |
& \les |Y|^{m} n U_n U_{m+1 - n} 
\\&\les   |Y|^{n+1} n U_n \th_3^{m-1-n} 
 (   C_* \kp)^{ m + 1 - n - 5} 4^{-(m-n)}  \\
 & 
\les |Y|^{n+1} n U_n \kp^{-2} (\th_3  C_* \kp 2^{-1})^{m-1-n} 2^{-(m-n)}
\\&\les |Y|^{n +1} n U_n  \kp^{-2} 2^{-(m-n)},
 \eal
\eeq
which is dominated by $P_{n, n+1} Y^{n+1}$ due to the estimate of $P_{n,n+1}$ \eqref{eq:Pn_main_low}.

\vspace{0.1in}
\paragraph{\bf{Case III: $ n / 8 \leq  m-n \leq  \tau n + 2$}}
Denote 
\beq\label{eq:def_m1}
m_1 = m -n.
\eeq
In this case, we have $m_1  \in [n/8, \tau n + 2]$. Using \eqref{eq:Um_asym_ref3} with $q$ chosen in \eqref{eq:q_tau} and \eqref{eq:G_bar3b}, for $|Y| \leq \th_3$ with $\th_3 =2 ( C_* \kp)^{-1} $ \eqref{eq:Y_thres}, we obtain 
\beq\label{eq:G_bar4}
\bal
 |Y ^{m} P_{n, m} |
 & \les  |Y|^{m} n U_n U_{ m -n + 1}   \\
& \les |Y|^{n+1} n U_n U_{m_1 + 1}  |Y|^{m_1 -1 } \\
&\les  |Y|^{n+1} n U_n  \th_3^{m_1- 1} 
 ( (1 + \d)  C_* \kp)^{ m_1 + 1  }  q^{\min( m_1 +1 , n - m_1-1)} \\
 & \les |Y|^{n+1} n U_n \kp^2  (2 (1 + \d))^{m_1} q^{\min(m_1, n-m_1)}.
 \eal
\eeq

Next, we further simplify the upper bound. We introduce 
\[
J = (2 (1 + \d))^{m_1} q^{\min(m_1, n-m_1)} .
 \]

We estimate $J$ in the case of $m_1 \leq n-m_1$ and $m_1 > n -m_1$ separately. Using \eqref{eq:q_tau}, we get 
\[
\bal
  J &\leq (2 (1 + \d))^{m_1} q^{m_1}  =   q_1^{m_1} \leq q_1^{n/8}, && \mathrm{for \ } n / 8 \leq m_1 \leq n /2. \\
 J & \leq C  (2 (1 + \d))^{ m_1-2} q^{\min(m_1 - 2, n- (m_1-2)}
 = C  (2 (1 + \d))^{ m_1-2} q^{  n- (m_1-2)} , &&  \mathrm{for \ } m _1 \in [n/2 ,  \tau n + 2].
  \eal
\]

Denote $\tau_0 = \f{m_1-2}{n}$. From \eqref{eq:def_m1} and the assumption of $m$, we get $\tau_0 \leq \tau$. Using $q < 1$ and $q_2$ in \eqref{eq:q_tau}, we further estimate the second case as follows 
\[
  J \leq  C 2(1 + \d )^{\tau_0 n } q^{ (1 - \tau_0) n }
\leq C (2(1 + \d )^{\tau} q^{1 - \tau})^n \leq C q_2^n .
\]

Combining the above estimates, we establish
\beq\label{eq:G_case2}
 |Y^{m} P_{n, m}| \les 
 |Y|^{n+1} n U_n n^2 q_5^{n},
 \quad q_5 = \max( q_1^{1/8},q_2) < 1.
\eeq

\vspace{0.1in}
\paragraph{\bf{Case IV: $ \tau n \leq m- n \leq n - 2 $ }}

Recall $\th_2, \th_3$ from \eqref{eq:Y_thres}. We consider $|Y| \leq \th_2$ and $|Y| \in [\th_2, \th_3]$ separately. We define $m_1 = m-n$ as \eqref{eq:def_m1}.
For $|Y| \leq \th_2$, following \eqref{eq:G_bar3} and applying \eqref{eq:Um_asym_ref1} with $l=0$, we obtain 
\beq\label{eq:G_case31}
\bal
 |Y^m P_{n, m} | 
& \les |Y|^{n+1} n U_n U_{m_1+1} |Y|^{m_1 - 1} \\&
\les |Y|^{n+1} n U_n \th_2^{m_1 - 1} 
((1 + \d)  C_* \kp)^{m_1 + 1} \\
& \les |Y|^{n+1} n U_n \kp^2  (\f{1+\d}{1 + 2\d})^{m_1}
\\
&
\les  |Y|^{n+1} n U_n n^2 q_3^{n/2}.
\eal
\eeq
Since $q_3 < 1$ \eqref{eq:q_taub}, the upper bound is very small compared to $ |Y|^{n+1} n U_n$.

Combining \eqref{eq:G_bar2}, \eqref{eq:G_case1}, \eqref{eq:G_case2}, \eqref{eq:G_bar_top}, and \eqref{eq:G_case31}, summing these estimates over $m =n+2,.. , 2n$, and using the lower bound of $P_{n, n+1}$ \eqref{eq:Pn_main_low}, we prove 
\beq\label{eq:G_bar_pf1}
\PP_\sn{n} \geq P_{n,n+1} Y^{n+1}( 1 + o_n(1))  > 0, \quad  Y \in [-\th_2, 0),
\eeq
where we use $a = o_n(1)$ to denote $\lim_{n \to \infty} |a|  = 0$. 

Next, for $ Y \in [-\th_3, -\th_2]$ and $l$ with $ \tau n + n \leq  2l , l \leq n-1$, we show that 
\[
Y^{2l + 1} P_{n, 2l+1}  + Y^{2l} P_{n, 2l} > 0.
\]

Using \eqref{eq:G_bar0} and \eqref{eq:Usq_grow}, we obtain 
\[
P_{n, 2l + 1} \leq - \f{2l+2}{2}( 1 - C n^{-1}) ( \GG_\sn{n}^2  )_{2l+2},
\quad |P_{n, 2l} | \leq \f{2l+1}{2} (1 + Cn^{-1}) (\GG_\sn{n}^2)_{2l+1}. 
\]

Since $2l \geq \tau n + n$ and $l \leq n$, using \eqref{eq:Usq_grow}, we get
\[
 \f{  |P_{n, 2l+1}| }{ |P_{n,2l}|} > \f{\tau(1-\d)}{1 - \tau} (1 - Cn^{-1}) ( C_* \kp) .
\]
Using $\f{\tau}{1-\tau } = 1 + 4 \d$ \eqref{eq:q_tau} and $|Y| \geq \th_2
= ( (1 + 2\d) C_* \kp)^{-1}$ \eqref{eq:Y_thres}, we yield 
\[
 \f{ | Y^{2 l+1} P_{n, 2 l+ 1} | }{ | Y^{2 l+} P_{n, 2 l} | }
 > \f{\tau(1-\d)}{( 1 - \tau)(1 + 2\d) } (1 - Cn^{-1})
 = \f{(1 + 4 \d) (1 - \d)}{1 + 2 \d} (1 - C n^{-1}) = q_4 (1 - C n^{-1}).
\]
Since $q_4>1$ \eqref{eq:q_tau}, for $n$ sufficiently large and $ Y < -\th_2$, since $Y^{2l+1} P_{n,2l+1} > 0$,   we prove 
\beq\label{eq:G_case32}
Y^{2 l+1} P_{n, 2 l+ 1} + Y^{2l} P_{n, 2l} > 0, \quad Y \leq -\th_2.
\eeq

Summing \eqref{eq:G_case32} over $\f{1}{2} (\tau n + n) \leq  l \leq n-1$ and then combining it with \eqref{eq:G_case1}, \eqref{eq:G_case2}, and \eqref{eq:G_bar_top}, we prove 
\[
\PP_\sn{n} \geq P_{n,n+1} Y^{n+1}( 1 + o_n(1))  > 0, \quad  Y \in [-\th_3, -\th_2).
\]

We conclude the proof. 
\end{proof}

\subsection{Intersection between the barrier functions  }

Next, we estimate the first intersection between the local barrier and the global barrier. Recall $U_{\D_Y}$ defined in \eqref{eq:root_UY}.

\begin{prop}\label{prop:inter_final}

Let $\bar C$ be the parameter defined in Lemma \ref{lem:asym_refine}. There exists $n_6 > \bar C$ large enough, such that the following statement holds true. 
For any $n > n_6$, there exists $Y_I$ with $ - (1 + 2 \d)| C_* \kp|^{-1} < - | U_n|^{1/ n}  < Y_I < 0$ such that 
\[
U_{\D_Y}(Y) < \GG_\sn{n}(Y), \quad Y \in (Y_I, 0),
\quad U_{\D_Y}(Y_I) = \GG_\sn{n}(Y_I). 
\]
As a result, we have
\[
 \D_Y(Y, \GG_\sn{n}(Y)) <0, \quad  Y \in (Y_I, 0) .
\]

\end{prop}

\begin{proof}

Denote 
\beq\label{eq:nota_inter}
a_1 = \pa_Y U_{\D_Y}(0),
\quad \HH_n(Y) =  \GG_\sn{n}(Y) - U_{\D_Y}(Y)  ,
\quad \th = |U_n|^{ - 1 / n } .
\eeq

Firstly, using a direct computation, we know that 
\[
\GG_\sn{n}(0) = U_{\D_Y}(0),
\  0<  \pa_Y \GG_\sn{n}(0) = U_1 < a_1, 
\quad  a_1 -U_1 \les n^{-1}.
\]
Using the above estimate and the definition of $U_{\D_Y}$ \eqref{eq:root_UY}, for $ - \e_2 <  Y <0$ and $\e_2 = \e_2(n, \kp)$ sufficiently small, we obtain
\beq\label{eq:inter_final_pf0}
\HH_n(Y) > 0 . 
\eeq

Since $U_{\D_Y}$ is smooth near $Y=0$, for some absolute constant $\e_3 > 0$ and $a_2$, we have 
\beq\label{eq:inter_final_pf1}
U_{\D_Y}(Y) \geq U_0 + a_1 Y + a_2 Y^2, \quad |Y| \leq \e_3.
\eeq

Denote 
\beq\label{eq:inter_final_F}
 F(Y) = \sum_{ i\leq n-1} U_{i+1} Y^{i} - a_2 Y  - a_1.
\eeq
We have $F( Y) = U_1 - a_1 + O_n(Y)< 0$ for $|Y|$ sufficiently small. From \eqref{eq:inter_final_pf1} and \eqref{eq:inter_final_F}, we get
\beq\label{eq:inter_final_pf3}
 \HH_n(Y)  \leq Y F(Y), \quad |Y| \leq \e_3.
\eeq

Next, we show that $F(-\th) > 0$ for $\th = |U_n|^{ - 1 / n }$ \eqref{eq:nota_inter}.  Using \eqref{eq:Um_grow2} and $|U_n|^{1/n} \gtr n$, for $m \leq n-1$, we get 
\[
\f{|U_m \th^{m-1}|  }{ |U_n \th^{n-1}| } 
= 
 \f{U_m}{ U_n^{m/n}}  \leq \min( C 2^{ - \min(n-m, m)}, C_{n-m} n^{-2/3}, C_m n^{-m}),
 \quad \f{|a_2 \th|}{ |U_n \th^{n-1}|}
 \les \th^2 \les n^{-2}.
\]

It follows 
\[
\bal
\cR & \teq |a_2 \th |
+ \sum_{  2 \leq m \leq n-2 } | U_m \th^{m-1}|
\leq |U_{n} \th^{n-1}| ( C_l n^{-2/3} + C \sum_{ l\leq m \leq n-l }  2^{- \min(m-n, m)})  \\
 & \leq |U_{n} \th^{n-1}| ( C_l n^{-2/3} + C 2^{-l}) .
\eal
\]
Choosing $l$ large and then $n$ large, 
we further obtain
\[
\cR  \leq \f{1}{2} |U_n \th^{n-1}| .
\]
Since $n$ is odd, it follows 
\[
\quad F(-\th ) \geq U_n \th^{n-1} - a_1 - \cR \geq 
 \f{1}{2}  U_n \th^{n-1}   + (U_1 - a_1) 
= \f{1}{2} |U_n|^{1/n}  + (U_1 - a_1) \geq C n + (U_1 - a_1)  > 0.
\]

Since $\th = |U_n|^{-1/n} \les n^{-1}$ and $\e_3$ in 
\eqref{eq:inter_final_pf1}, \eqref{eq:inter_final_pf3} is absolute constant, by further requiring $n$ large enough, we yield $0 < \th <  \e_3$. Thus using \eqref{eq:inter_final_pf3}, we obtain
\[
 \HH_n(-\th) \leq -\th F(-\th) < 0.
\]

Since $\HH_n(Y) > 0$ for $-\e_2 < Y< 0$ \eqref{eq:inter_final_pf0} and $\HH_n$ is continuous, there exists $Y_I$ with $ -\th = -|U_n|^{1/n} < Y_I < 0$ such that $\HH_n(Y_I) = 0$ and $\HH_n(Y) > 0$ for $Y \in ( Y_I, 0]$. Using \eqref{eq:Um_grow2}, since $\d =0.05$ \eqref{eq:para3} and $(1-\d)^{-1} < 1 + 2 \d$, by further choosing $n$ large, we obtain 
\beq\label{eq:Un_est1}
|Y_I| < |U_n|^{-1/n} < (1 + 2 \d)  (C_* \kp)^{-1}.
\eeq

Finally, using the definition of $\D_Y$ \eqref{eq:UY_ODE}, we yield 
\[
  \D_Y( Y, \GG_\sn{n}(Y)) = ( d Y - 1)( \GG_\sn{n} - U_{\D_Y}) < 0 ,\quad Y \in (Y_I, 0).
\]
We conclude the proof. 
\end{proof}

Below, we show that the solution must intersect the curve $U_{\D_Y}$ \eqref{eq:root_UY}.

\begin{prop}\label{prop:inter_U_root}

For any $n$ odd large enough, we have the following results. Let $Y_I$ be defined in Proposition \ref{prop:inter_final}. There exists $Y_I^{\pr} \in [Y_I, 0)$ and $Y_I^{\pr} > -\g$ such that 
$ U(Y_I^{\pr}) = U_{\D_Y}(Y_I^{\pr} )$ and 
\beq\label{eq:inter_U_root}
 U_{\D_Y}(Y) <  U(Y) <  \GG_\sn{n}(Y)   , \quad \mw{for } \quad  Y \in ( Y_I^{\pr}, 0) .
\eeq

\end{prop}

\begin{proof}

Recall the definition of $\GG_{\sn{n}}$ from \eqref{eq:up_bar}. Firstly, we have $\pa_Y^i (\GG_\sn{n} - U)(0)$for $ i \leq n$ and
\[
\pa_Y^{n+1} (\GG_\sn{n} - U)(0) = - U_{n+1}.
\]
Since $n$ is odd, using $\kp \in (n, n+1)$ and \eqref{eq:U_sign}, \eqref{eq:induc_asym2}, we obtain $-U_{n+1} > 0, - U_{n+1} Y^{n+1} > 0$. Using the above estimates and
\[
 \GG_\sn{n}(0) = U_{\D_Y}(0), \quad 
\pa_Y\GG_\sn{n}(0) = U_1 < \pa_Y U_{\D_Y}(0), 
\]
we establish 
\[
  U_{\D_Y}(Y) <   U(Y) < \GG_\sn{n}(Y) ,
\]
for $ Y \in (-\d^{\pr}, 0)$ with some small parameter $\d^{\pr} = \d^{\pr}(\kp, n) > 0$.

From Proposition \ref{prop:inter_final} and Proposition \ref{prop:bar_final}, 
for some $Y_I \in ( - 2 ( C_* \kp)^{-1} ,0)$,  we have
\[
 \f{d}{d Y} ( \GG_\sn{n}(Y) - \f{ \D_U(Y, U)}{\D_Y(Y, U)} \B|_{ U = \GG_\sn{n}(Y)} = \f{\PP_\sn{n}(Y)}{ \D_Y(Y, \GG_\sn{n}(Y))}  < 0, \quad Y \in ( Y_I, 0).
 \]
 Thus, $\GG_\sn{n}(Y)$ is a upper barrier for $U(Y)$ with $Y \in (Y_I, 0)$.

 Since $\GG_\sn{n}(Y)$ intersects $  U_{\D_Y}(Y) $ at $Y_I <0$, by continuity, we obtain that 
 $U(Y)$ intersects $  U_{\D_Y}(Y)$ at some $Y_I^{\pr} \in (Y_I, 0)$. 

Since $Y_I > - (1 + 2\d) |C_* \kp|^{-1}$ from Proposition \ref{prop:inter_final} and $\g > \ell^{-1/2}$, by choosing $\kp$ large enough, we obtain $Y_I^{\pr} > Y_I > - \g$. We complete the proof. 
\end{proof}


We have the following relative positions among the roots and the barrier functions.

\begin{prop}\label{prop:root_pos} 
Let $U_g, U_{\D_U}$ be the functions defined in \eqref{eq:root_DZ} and \eqref{eq:root_UY}.  There exists $C>0$ large enough such that for any $n, \kp$ with $n$ odd, $n > C$, and $\kp \in (n, n+1)$, we have 
\beq\label{eq:root_pos}
   U_g(Y)  >  G_\sn{n}(Y) ,  \quad   U_{\D_U}(Y)  >  G_\sn{n}(Y) .
\eeq
for $ - |U_n|^{1/n} < Y < 0$, where $ C_*$ is defined in \eqref{eq:C_asym}. As a result, we have
\beq\label{eq:root_pos2}
   U_g(Y) > U(Y), \quad  U_{\D_U}(Y)  > U(Y).
\eeq
for $Y \in [ Y_I^{\pr}, 0)$, where $Y_I^{\pr}$ is defined in Proposition \ref{prop:inter_final}.

\end{prop}


\begin{proof}
We consider $\ell^{-1/2} < \g < 1$. Denote $J= [ - |U_n|^{- 1/n} , 0)$. 
From the definitions of $\e$ \eqref{eq:para2}, $U_g, U_{\D_U}$ \eqref{eq:root_UY}, and $G_\sn{n}$ \eqref{eq:up_bar}, we have $U_g(0) = U_{\D_U}(0) = G_\sn{n}(0) = U_0 = \e$. Thus, we only need to show
\[
 \sum_{ 1 \leq i \leq n} U_i Y^i  < \pa_Y g(0) Y + C_2 Y^2
\]
for $Y \in J$ and some absolute constant $C_2$ independent of $n$, where $g = U_g$ or $ g = U_{\D_U}$.
Since $Y < 0$, using \eqref{eq:root_ineq}, we only need to show that the inequality
\beq\label{eq:root_pos_est0}
 C_3 Y +  \sum_{ 2 \leq i \leq n} U_i Y^{i-1} + c  > 0,
\eeq
with some absolute constant $C_3$ independent of $n$ and $c > 0$ determined in \eqref{eq:root_ineq} holds for any $Y \in J$. By requiring $n$ large, we have \eqref{eq:Un_est1} and 
\[
 |Y| < |U_n|^{-1/n} < (1 + 2 \d) ( C_* \kp)^{-1}. 
\]

Since $n$ is odd and $\kp \in (n, n+1)$, using \eqref{eq:U_sign}, we get $U_n > 0$ and $ U_n Y^{n-1} > 0$. 
Below, we further bound  $C_3 Y, U_i Y^{i-1}, i\leq n-1$.

For $i \leq \f{n}{8}$, using \eqref{eq:Um_asym_ref2} with $l = 3$ and $r =  \f{1 + 2 \d}{4} < 1$, 
for $Y \in J$,  we get 
\beq\label{eq:root_pos_est1}
\bga
 |C_3 Y | + |U_2 Y| \les n^{-1},  \\
 | U_i Y^{i-1} | \leq n^{-2} + (C_* \kp)^{i-3}   C_* \kp)^{-(i-1)}
\les n^{-2}, \quad  3 \leq i\leq  \f{n}{8}.
\ega
\eeq
For  $ \f{n}{8} \leq i < n -n^{1/3} + 2$ , applying \eqref{eq:Um_grow2a} and $U_i^{1/i} \leq n$ from \eqref{eq:Um_asym_ref1}, we obtain 
\beq\label{eq:root_pos_est2}
|U_i Y^{i-1}| = (U_i Y^i)^{ (i-1)/i} | U_i^{1/i} |
\les 2^{- \min(i, n-i) \f{i-1}{i} }  n 
\les 2^{- \f{1}{2} n^{1/3} }  n , \quad   \f{n}{8} \leq i < n - n^{1/3}.
\eeq

For $ n-1 \geq j \geq n - n^{1/3}$, we estimate the cases $|Y| \leq n^{-4/3}$ and $|Y| >
n^{-4/3}$ separately. 

For $|Y| \leq n^{-4/3}$, using \eqref{eq:Um_asym_ref1} with $l = 0$ and $j \leq n-1$, we get 
\bseq
\beq\label{eq:root_pos_est31}
| U_j Y^{j-1} | \les n ( C n \cdot n^{-4/3})^{j-1} \les n^{-2}. 
\eeq
Combining \eqref{eq:root_pos_est1}-\eqref{eq:root_pos_est31} and  $U_n Y^{n-1}>0$, we prove \eqref{eq:root_pos_est0} for $n$ large enough with $|Y| \leq n^{-4/3 }$.

For $|Y| > n^{-4/3}$, we estimate $U_{2j} Y^{2j-1} + U_{2j +1} Y^{2 j}$ together 
with $n - n^{1/3} < 2j + 1\leq n$. Using \eqref{eq:induc_asym2} and \eqref{eq:U_sign}, for $n$ large enough, we get 
\beq\label{eq:root_pos_est32}
\bga
U_{2j+1} > C U_{2j} n \cdot \f{2j + 1}{\kp - 2j - 1} > C n^{5/3} U_{2j}, \\
U_{2j} Y^{2j-1} + U_{2j +1} Y^{2 j}
= |Y^{2j-1}| ( U_{2j+1} |Y| - U_{2j} ) > 
U_{2j}|Y^{2j-1}| ( C n^{5/3} n^{-4/3} - 1 ) > 0.
\ega
\eeq
\eseq
Combining \eqref{eq:root_pos_est1},\eqref{eq:root_pos_est2},\eqref{eq:root_pos_est32}, we prove \eqref{eq:root_pos_est0} with $ n^{-4/3} \leq |Y| \leq |U_n|^{-1/n}$. 

The inequalities in \eqref{eq:root_pos2} follow from \eqref{eq:inter_U_root}, \eqref{eq:root_pos} and 
$|Y_I^{\pr}| < |Y_I| < U_n^{1/n}$ proved in Propositions \ref{prop:inter_final}, \ref{prop:inter_U_root}. 
\end{proof}

We have the following estimate of $\D_Y(Y, U(Y))$ and $U(Y)$ near $Y_I^{\pr}$. 
\begin{lem}\label{lem:DelY}
There exists constant $C_U> 0$ such that 
\[
|\D_Y(Y, U(Y)) | \geq C_U |Y- Y_I^{\pr}|^{1/2} 
\]
for any $Y \in [ Y_I^{\pr}, Y_I^{\pr}/2]$. Moreover, we have $ U(Y) \geq C_U $ for any $Y \in [Y_I^{\pr}, 0]$.
\end{lem}

\begin{proof}

Firstly, we have
\[
  \f{d}{d Y} \D_Y^2(Y, U(Y))
 = 2 \D_Y  \B( \pa_Y \D_Y + \pa_U \D_Y \f{d U(Y)}{d Y}  \B)
 = 2 \D_Y \cdot  \pa_Y \D_Y + 2 \D_U \cdot  \pa_U \D_Y.
\]
Using $\pa_U \D_Y = d Y - 1 $ \eqref{eq:ODE}, $ 
\D_Y( Y_I^{\pr}, U(Y_I^{\pr})) = 0 $ from Proposition \ref{prop:inter_U_root} and the formula of $\D_Y$ in \eqref{eq:root_UY}, $\D_U( Y_I^{\pr}, U(Y_I^{\pr})) < 0 $ from \eqref{eq:root_pos2} and \eqref{eq:root_UY}, and the continuity of $U(Y)$ on $[Y_I^{\pr}, 0]$, we get
\[
\lim_{Y \to (Y_I^{\pr})^+}   \f{d}{d Y} \D_Y^2(Y, U(Y)) = 
\lim_{Y \to (Y_I^{\pr})^+}   2 \D_U \cdot  \pa_U \D_Y = 
2  \D_U( Y_I^{\pr}, U(Y_I^{\pr})) (d Y_I^{\pr} - 1) > 0. 
\] 

Since $\D_Y( Y_I^{\pr}, U(Y_I^{\pr})) = 0$ and $\D_Y( Y, U(Y)) \neq 0$ for any $Y \in (Y_I^{\pr}, 0)$,  we obtain 
\[
\D_Y^2(Y, U(Y)) \geq  C_U ( Y- Y_I^{\pr})^{1/2},
\quad \Rightarrow  \quad  |( \D_Y(Y, U(Y)) )^{-1}  | \les C_U |Y- Y_I^{\pr}|^{-1/2},
\]
for any $Y \in [Y_I^{\pr}, Y_I^{\pr}/2]$ and some $C_U > 0$. Since $(Y, U(Y))$ is continuous on $Y \in [Y_I^{\pr}, 0]$, we have 
$|Y, U(Y)| \les D_U$ for some constant $D_U$. Using the ODE \eqref{eq:UY_ODE}, we obtain 
\[
 \B| \f{\D_U}{ U \cdot \D_Y } \B| \les C_U |Y- Y_I^{\pr}|^{-1/2}, \quad \mw{for} \ 
 Y \in (Y_I^{\pr}, Y_I^{\pr}/2 ],
\]
for some constant $C_U >0$. Since $|Y- Y_I^{\pr}|^{-1/2}$ is integrable, we obtain $U(Y_I^{\pr}) \geq  C_{U,2} U(Y_I^{\pr}/2) > 0$ for some $C_{U,2} >0$. Since $U_{\D_Y}(Y)$ is increasing in $Y <0$ (see Lemma \ref{lem:mono}), using \eqref{eq:inter_U_root}, we get $U(Y) \geq U_{\D_Y}(Y) \geq U_{\D_Y}(Y_I^{\pr}) = U(Y_I^{\pr}) >0 $ for any $Y \in [ Y_I^{\pr}, 0)$, we conclude the proof.
\end{proof}


\subsection{Extension across $Y=0$}

In this section, we show in Proposition \ref{prop:ds_ODE} that after $U(Y)$ crosses $U_{\D_Y}(Y)$, we can smoothly extend the solution curve $(Y, U(Y))$ slightly beyond $Y = 0$. See Figure \ref{fig:coordinate2} for an illustration. In Section \ref{sec:proof}, we will use Proposition \ref{prop:ds_ODE} to prove Theorem \ref{thm:main}.

We introduce the triangle region below $\min( U_{\D_Y}(Y), U_g(Y))$
\beq\label{eq:dom_tri}
\bal
& \Om_{\mw{tri}, 1} \teq \{ (Y, U) : -\g < Y < 0,  \ 0<  U < \min(  U_g(Y), U_{\D_Y}(Y)) ,  
\  U_{\D_Y}(Y)  < U  \}, \\
& \Om_{\mw{tri}, 2} \teq \{ (Y, U) : -\g < Y < 0, \ 0 < U < \min(  U_g(Y),  U_{\D_Y}(Y)),
\   U < U_{\D_Y}(Y) \}.
\eal
\eeq
From the definition of $\Om_{\mw{tri}, i}$ in \eqref{eq:dom_tri} and \eqref{eq:root_UY}, we get 
\beq\label{eq:dom_tri_sign}
\bal
\D_U(Q) &< 0, \ \forall Q \in \Om_{\mw{tri}, 1} , \ \Om_{\mw{tri}, 2}, \\
 \D_Y(Q)  &< 0,  \ \forall Q \in \Om_{\mw{tri}, 1}, \quad 
 \D_Y(Q)  > 0, \ \forall Q \in \Om_{\mw{tri}, 2}. 
\eal
\eeq

Since the ODE becomes singular on $(Y, U) = (Y, U_{\D_Y(Y)})$, we desingularize the ODE \eqref{eq:UY_ODE} 
\bseq\label{eq:ODE_ds}
\beq
  \f{d U_\ds}{d \xi} = \D_U(Y_\ds(\xi), U_\ds(\xi)),  \quad  \f{d Y_{\ds} }{d \xi} = \D_Y(Y_\ds(\xi), U_\ds(\xi)) ,
\eeq
where 
\textit{ds} is short for desingularized. 
Let $Y_I^{\pr} \in (-\g, 0) $ be the intersection determined in Proposition \ref{prop:inter_U_root}. 
Using estimates of $U(Y)$ in \eqref{eq:inter_U_root},\eqref{eq:root_pos2}, we can pick the initial data as
\beq\label{eq:ODE_ds_init}
 ( Y_\ds, U_\ds)(0) =  ( Y_I^{\pr} / 2,  U( Y_I^{\pr} / 2) ), \quad ( Y_\ds, U_\ds)(0) \in \Om_{\mw{tri}, 1}. 
\eeq
\eseq
Define 
\beq\label{eq:Q_ds}
Q_{\ds}(\xi) = (Y_{\ds}(\xi), U_{\ds}(\xi) ). 
\eeq

 Since $\D_Y(Y, U(Y)) < 0$ for $Y \in (Y_I^{\pr}, 0)$, we define a map $\eta(Y)$ 
\beq\label{eq:ODE_ds_map}
 \f{ d \eta(Y)}{ d Y } = \f{1}{\D_Y( Y, U(Y))} < 0,
 \quad  \eta( \f{Y_I^{\pr}}{2} ) = 0.
\eeq
We have $\eta(Y)  \in C^{\infty}(Y_I^{\pr}, 0)$ and it is a bijection. For the desingularized ODE, we show:

\begin{prop}\label{prop:ds_ODE}

There exists $\xi_1, \xi_2 , \xi_3, \d_5 >0$ with $0< \xi_1< \xi_1 + \d_5 < \xi_2 < \xi_3$ such that:  \

(a) The solution $(Y_{\ds}, U_{\ds})$ can be obtained from  $(Y, U(Y))$ by a reparametrization.
\beq\label{eq:ODE_equiv2}
(Y_{\ds}, U_{\ds})(\xi) = 
 (\eta^{-1}(\xi), U( \eta^{-1}(\xi))),  \quad \mw{for} \quad 
\xi \in ( 0, \xi_1 ] ,\quad  \xi_1 =  \eta(Y_I^{\pr}) ;
\eeq

(b) For any $\xi \in [0, \xi_1]$, we have 
\[
0 < U_{\ds}(\xi) < \min( U_g( Y_{\ds}(\xi)), U_{\D_U}(Y_{\ds}(\xi) ) ), 
\quad Y_\ds(\xi) > -\g.
\]

(c) We have $ (Y_{\ds}(\xi), U_\ds(\xi) ) \in \Om_{\mw{tri}, 2}$ for any $\xi \in (\xi_1, \xi_2)$ and $Y_{\ds}(\xi_2) = 0$;

(d) For any $\xi \in (\xi_2, \xi_3]$, we have $0 < Y_{\ds}(\xi) < 1$. For any $\xi \in [\xi_2,  \xi_3]$ (including $\xi_2$), we have
\[
 0< U_\ds(\xi) < \min( U_{\D_Y}(\cdot ), U_{\D_U}(\cdot ), U_g(\cdot))(Y_{\ds}(\xi))  .
 \]

 (e) In particular, for any $\xi \in [0, \xi_3]$, we have 
 \[
  0 < U_{\ds}(\xi) < \min( U_g( Y_{\ds}(\xi)), U_{\D_U}(Y_{\ds}(\xi)) ),
  \quad 
   Y_\ds(\xi) > -\g .
 \]

\end{prop}

\subsubsection{Estimates of $\eta$}
We define $\xi_1 = \eta(Y_I^{\pr})$. Since $\f{1}{\D_Y( Y, U(Y))}$ is integrable near $Y = Y_I^{\pr}$ from Lemma \ref{lem:DelY}, 
using the equation of $\eta$ \eqref{eq:ODE_ds_map}, we obtain that $\eta(Y_I^{\pr})$ is bounded.

Using the ODE \eqref{eq:UY_ODE}, the definition of $\eta$ \eqref{eq:ODE_ds_map}, and the chain rule, we obtain that the ODE system of $(\eta^{-1}(\xi), U( \eta^{-1}(\xi)))$ in $\xi$ is the same as \eqref{eq:ODE_ds}. Moreover, two ODEs have the same initial condition at $\xi = 0$ due to $\eta^{-1}(0) = Y_I^{\pr}/2$ (see \eqref{eq:ODE_ds_map}) and \eqref{eq:ODE_ds_init}. Using the uniqueness of the ODE solution and continuity, we prove \eqref{eq:ODE_equiv2}.

Using the relation \eqref{eq:ODE_equiv2} and the estimate of $(Y, U(Y))$ with $Y \in (Y_I^{\pr},0)$ in \eqref{eq:root_pos2}, $Y_\ds(\xi)\geq Y_I^{\pr}>-\g$, and Proposition \ref{prop:inter_U_root}, we obtain 
\beq\label{eq:ODE_van}
Q_{\ds}(\xi) \in \Om_{\mw{tri},1}, \  \forall \xi \in [0, \xi_1), 
\quad \D_Y( Y_{\ds}(\xi_1), U_{\ds}(\xi_1 )) = 0 ,
\quad \D_U( Y_{\ds}(\xi_1), U_{\ds}(\xi_1 )) < 0.
\eeq


Using \eqref{eq:root_pos2} with $Y \in [Y_I^{\pr}, 0)$ (\textit{including $Y_I^{\pr}$}), $Y_I^{\pr} > -\g$ from Proposition \ref{prop:inter_U_root}, $U(Y) > 0$ for $Y \in [Y_I^{\pr}, 0]$ from Lemma \ref{lem:DelY}, and the relation \eqref{eq:ODE_equiv2}, we prove result (b) in Proposition \ref{prop:ds_ODE}.


\subsubsection{Solving the desingularized ODE}

Next, we solve the ODE \eqref{eq:ODE_ds} for $\xi \geq \xi_1$ and prove item (c), (d), (e) in Proposition \ref{prop:ds_ODE}. We consider
\[
F(\xi) = U_{\D_Y}( Y_\ds(\xi)) - U_{\ds}(\xi).  
\]
Using \eqref{eq:ODE_van}, we get $F(\xi) < 0 $ for $\xi \in (0, \xi_1)$ and $F(\xi_1) = 0$. Moreover, using the ODE \eqref{eq:ODE_ds}, we compute 
\[
F^{\pr}(\xi_1)=  U^{\pr}_{\D_Y}(Y_{\ds}(\xi_1)) \cdot \D_Y(Q_{\ds}(Y(\xi_1)))
 - \D_U(Q_{\ds}(Y(\xi_1))) =  - \D_U(Q_{\ds}(Y(\xi_1))) > 0.
\]
We obtain the last inequality using the strict inequalities in result (a) in Proposition \ref{prop:ds_ODE} at $\xi = \xi_1$ and \eqref{eq:root_UY}. Using continuity, for $\xi \in (\xi_1, \xi_1 + \d_5]$ with small $\d_5>0$, we obtain 
\beq\label{eq:Om2_enter}
F(\xi) > 0, \quad Q_{\ds}(\xi) \in \Om_{\mw{tri}, 2} .
\eeq

\vspace{0.1in}
\paragraph{\bf{Intersect $Y = 0$}}

Next, we  show that $Q_{\ds}(\xi)$ intersects the curve $Y = 0$. 
Due to the sign conditions \eqref{eq:dom_tri_sign}, starting at 
$Q_{\ds}(\xi_1)$,
we know that if $Q_{\ds}(\xi) \in \Om_{\mw{tri}, 2}$, we have
\beq\label{eq:exit_mono}
    \f{d}{d \xi} U_{\ds} ( \xi ) =   \D_U( Q_{\ds}(\xi)) < 0,  \quad   \f{d}{d \xi} Y_{\ds} ( \xi ) 
    = \D_Y( Q_{\ds}(\xi)) > 0 .
\eeq

Due to continuity and \eqref{eq:Om2_enter}, we get $Q_{\ds}(\xi) \in \Om_{\mw{tri}, 2}$ for $\xi \in (\xi_1 ,  \xi_2)$ with some $\xi_2 > \xi_1 + \d_5 $. We assume that $\xi_2$ is maximal value such that 
$Q_{\ds}(\xi) \in \Om_{\mw{tri}, 2}$ for all $\xi \in (\xi_1 ,  \xi_2)$ and $\xi_2$ is allowed to be $\infty$. Below, we show that $\xi_2  < \infty$ and $Y(\xi_2) = 0$. 

Using the monotonicity \eqref{eq:exit_mono}, \eqref{eq:ODE_equiv2} with $\xi = \xi_1$, and $Y_I^{\pr} > -\g$ from Proposition \ref{prop:inter_U_root}, we get 
\beq\label{eq:exit_Y_low}
 0> Y_\ds(\xi ) > Y_{\ds}(\xi_1) = Y_I^{\pr} > - \g , 
 \quad  0< U_\ds(\xi) <  U_\ds(\xi_1), \quad \xi \in (\xi_1, \xi_2) .
\eeq

For $\xi \in (\xi_1 + \d_5, \xi_2)$, 
since $Q_{\ds}(\xi) \in \Om_{ \mw{tri}, 2}$ and $g(Y) =  U_{\D_Y}(Y), U_{\D_U}(Y), U_g(Y) $ is increasing in $Y$ for $ Y \in ( -\g, 0)$ from Lemma \ref{lem:mono}, using \eqref{eq:exit_mono}, we get
\beq\label{eq:exit_1}
U_\ds(\xi) - g( Y_\ds(\xi) )
\leq U_{\ds}( \xi_1 + \d_5 ) - 
 g( Y_\ds(\xi_1 + \d_5) ) < -c < 0,  
\eeq
for some constant $c > 0$ independent of $\xi$, where the last inequality follows from $Q_\ds( \xi_1 + \d_5) \in \Om_{ \mw{tri}, 2}$ \eqref{eq:dom_tri}. Using $Y_\ds(\xi) < 0< U_\ds(\xi)$ \eqref{eq:exit_Y_low}, 
\eqref{eq:exit_1}, and the uniform boundedness of $Q_\ds(\xi)$ \eqref{eq:exit_Y_low}, we estimate $\D_U, \D_Y$ in \eqref{eq:UY_ODE} as 
\bseq\label{eq:exit_xi2}
\begin{align}
\D_U(Q_{\ds}(\xi)) & = 2 U ( U + f(Y) + (d-1) Y (1 - Y))
\geq - C U_\ds(\xi) ,  \label{eq:exit_xi2_a} \\
\D_Y(Q_{\ds}(\xi) ) & = (d Y - 1) ( U - U_{\D_Y}) 
(Q_{\ds}(\xi)) > c > 0, 
\label{eq:exit_xi2_b}
\end{align}
\eseq
for all $\xi \in (\xi_1 + \d_5, \xi_2)$ and some constant $C, c > 0$ independent of $\xi$.

Due to \eqref{eq:exit_xi2_b} and $Y <0$ for any $(Y, U) \in \Om_{\mw{tri}, 2}$, the curve $Q_\ds(\xi)$ must exit $\Om_{\mw{tri},2}$ with finite $\xi_2$. Using \eqref{eq:exit_xi2_a} and the ODE of $U$ in \eqref{eq:ODE_ds}, we get $U(\xi_2) \geq U(\xi_1 + \d_5) e^{ -C(\xi_2 - \xi_1)} > 0$.
Since $Y_{\ds}, U_{\ds}$ are continuous in $\xi$, taking $\xi \to \xi_2^-$ in \eqref{eq:exit_1}, we get 
\beq\label{eq:exit_pf3}
 0 < U_\ds(\xi_2) <  \min(U_{\D_Y}( \cdot ), U_{\D_U}( \cdot ), U_g( \cdot) ) (Y_\ds(\xi_2)) .
\eeq
Using \eqref{eq:exit_pf3}, the definition of $\Om_{\mw{tri},2}$ in \eqref{eq:dom_tri}, and $U(\xi_2) > 0$, we obtain that $Q_{\ds}(\xi_2)$ can only exit  $\Om_{\mw{tri},2}$ via $Y = 0$. Thus, we prove $Y(\xi_2) = 0$ and the result (c) in Proposition \ref{prop:ds_ODE}.

\vspace{0.1in}
\paragraph{\bf{Crossing the line $Y= 0$}}

Recall that \eqref{eq:exit_xi2_b} applies to any $\xi \in (\xi_1 + \d_5, \xi_2)$. Since $U_{\ds}, Y_{\ds}$ are continuous, taking $\xi \to \xi_2$, we get 
\[
Y^{\pr}(\xi_2) =   \D_Y( Q_{\ds}(\xi_2)) > 0.
\]

Since $Y(\xi_2) = 0$, using the above estimate, \eqref{eq:exit_pf3}, and the continuities of $Y_{\ds}, U_{\ds}$, and then choosing $\xi_3>\xi_2$ with $\xi_3 -\xi_2$ small enough, we prove the result (d) in Proposition \ref{prop:ds_ODE}. 

Since $Q_\ds(\xi) \in \Om_{tri, i}$ \eqref{eq:dom_tri} implies $Y_\ds(\xi) > -\g$ and $
0 < U_\ds(\xi) < U_g( Y_\ds(\xi)), U_{\D_U}(Y_\ds(\xi))$.  Combining results (b)-(d) in Proposition \ref{prop:ds_ODE}, we prove result (e).  We complete the proof.

\section{Proof of Theorem \ref{thm:main}}\label{sec:proof}

In this section, we prove Theorem \ref{thm:main}. 

Recall $Z_{\pm}(V), Z_V(V)$ from \eqref{eq:root_mono_ZV0}, which relate to the roots of $\D_V , \D_Z$. We introduce
\beq\label{eq:dom_far}
\Om_{ \mw{far}} = \{ (Z, V):  V \in (-1, 1),  \ 
Z > Z_+(V), \ Z > Z_V(V) , Z > V \}. 
\eeq

The proof of Theorem \ref{thm:main} follows from the following two results.

\begin{prop}\label{prop:QO_Q2}
There exists $C >0$ large enough such that for any odd $n$ with $n > C$, there exists $\kp_n \in (n, n+1)$ and $\e_1 > 0$, such that the ODE \eqref{eq:ODE} admits a solution $V(Z) \in C^{\infty}([0, Z_2 ])$ with $Z_2 > Z_0$,
$V(0) = 0, \  V(Z_0) = V_0$, 
\beq\label{eq:prop_ODE1}
V(Z) \in (-1, 1) , \quad V(Z) <Z, \quad \mw{for} \quad Z \in (0, Z_2], 
\eeq
and $(Z_2, V(Z_2)) \in \Om_\mw{far}$, where $Z_0, V_0$ is defined in \eqref{eq:root_P}.

\end{prop}

\begin{lem}\label{lem:far_extend}
Given any $(Z_2, V_2)$ with $Z_2>0$ and $(Z_2, V_2) \in \Om_{\mw{far}}$,  the ODE solution to \eqref{eq:ODE} starting at $Z = Z_2, V(Z_2) = V_2$ admits a smooth solution  $(Z, V(Z))$ for any $Z \geq Z_2$ with 
$(Z, V(Z)) \in \Om_{\mw{far}}$ and $\lim_{Z\to\infty} V(Z) = V_{\infty}$ for some $V_{\infty} \in (- 1, 1)$. 
\end{lem}

We defer the proofs of Proposition \ref{prop:QO_Q2} and Lemma \ref{lem:far_extend} 
to Sections \ref{sec:QO_Q2} and \ref{sec:global}, respectively.
See the black curve in Figure \ref{fig:coordinate1} for an illustration of the solution curve $(Z, V)$. 
The region $\Om_{ \mw{far}}$ lies to the right of the red and blue curves in Figure \ref{fig:coordinate1}.

\begin{proof}[Proof of Theorem \ref{thm:main}]


Using Proposition \ref{prop:QO_Q2}, there exists $C$ large enough, such that for any $n > C$, $n$ is odd, and some $\kp_n \in (n,n+1)$,  we can construct a smooth solution $V^\rn{\kp_n}(Z)$ to the ODE \eqref{eq:ODE} in $Z \in [0, Z_2]$ with the properties in Proposition \ref{prop:QO_Q2}.  Since 
$(Z_2, V^\rn{\kp_n}(Z_2))\in \Om_\mw{far}$ and $ \Om_\mw{far}$ \eqref{eq:dom_far} is open, we can choose $Z_2^{\pr} < Z_2$ with $(Z_2^{\pr}, V^\rn{\kp_n}(Z_2^{\pr})) \in \Om_\mw{far}$. Applying Lemma \ref{lem:far_extend} with $(Z_2, V_2)  \rightsquigarrow (Z_2^{\pr}, V(Z_2^{\pr})) $, and using the uniqueness of ODE, we construct a global smooth solution to the ODE \eqref{eq:ODE} with 
$V^\rn{\kp_n}(0)  = 0$ and the estimates \eqref{eq:thm_prop}.

From Proposition \ref{prop:Qs_QO}, we obtain $V^\rn{\kp_n}(Z) = Z g(Z^2)$ for some $g \in C^{\infty}([0, Z_0 + \e_1]$ with some $\e_1>0$. Since $V^\rn{\kp_n} \in C^{\infty}([0, \infty))$, we can extend this property for some $g  \in C^{\infty}([0, \infty))$.

Next, we estimate the profile $W$ for $\rho$ \eqref{eq:SS_ansatz}, which satisfies \eqref{eq:W_ODE}: 
\[
\f{\pa_Z W}{W} = J_W(Z), \quad  J_W(Z) = \f{2}{(p-1)\ell} \cdot \f{1}{Z-V}
\B(  (d-1) \f{V}{Z (1 - V^2)}  - \f{d-1}{\g+1} - \f{Z V - 1}{ 1 - V^2}  \cdot V^{\pr}  \B).
\]

From \eqref{eq:thm_prop}  and Lemma \ref{lem:far_extend}, we obtain 
$C^{\infty}([0, \infty))$, 
$V(0) = 0$,
\beq\label{eq:limit_V}
|V| < 1, \quad V(Z) < Z \ \mw{for} \ Z > 0, \quad \lim_{Z \to \infty} V(Z) = V_{\infty} \in (-1, 1) .
\eeq
Using the power series expansion \eqref{eq:V_near0} in Proposition \ref{prop:V_near0} and \eqref{eq:V_formula}, we have $|V^{\pr}(Z) - \f{d-1}{d (\g+1)} | \les Z^2$. Thus, using the above estimates, we obtain that the denominators in $J_W(Z)$ are non-zero for $Z > 0$, the singularity $Z$ is cancelled near $Z = 0$, and $J_W(Z) \in C^{\infty}([0, \infty))$. 

For any $Z > 2$ large enough, 
estimating $V^{\pr}$ using the ODE \eqref{eq:ODE}, we derive the asymptotics 
\[
 |V^{\pr}(Z)| = |\f{\D_V}{\D_Z}| \les Z^{-2}, 
 \quad  \B| J_W(Z) + \f{2(d-1)}{ (p-1) \ell (\g +1)} \f{1}{Z} \B| \les Z^{-2}. 
\]

Choosing $W(0) = 1$, we obtain 
\beq\label{eq:limit_W}
W(Z) =  \exp( \int_0^{Z} J_W(Z) d Z ) \in C^{\infty}([0, \infty) ),
\quad  \lim_{Z \to \infty}  W(Z) Z^a  = W_{\infty} \neq 0,
\eeq
with $a = \f{2(d-1)}{ (p-1) \ell (\g +1)}$. We further construct the profile $\Phi$ using \eqref{eq:Phi_ODE} with $\Phi(0) = 0$. Using these profiles, the self-similar ansatz \eqref{eq:SS_var}, \eqref{eq:SS_ansatz}, and \eqref{eq:rel_eulera}, we construct the smooth profile $\Phi, W$ for $\phi, \rho$ and obtain a smooth self-similar imploding solution \eqref{eq:euler_implod} to the relativistic Euler equations \eqref{eq:rel_euler}. The asymptotics \eqref{eq:blow_asym} follows from the limits of $V, W$ in 
\eqref{eq:limit_V}, \eqref{eq:limit_W} and the formulas of $u^0, u^i$ in \eqref{eq:rel_eulera}, 
\eqref{eq:euler_implod},\eqref{eq:SS_var}.  We complete the proof.
\end{proof}

\subsection{Proof of Proposition \ref{prop:QO_Q2}}\label{sec:QO_Q2}

We assume that $n$ is odd and large enough. Using Proposition \ref{prop:Qs_QO}, we construct a smooth solution $V^\rn{\kp_n} \in C^{\infty}[0, Z_0 + \e_1]$ with some $\e_1>0$ and
\beq\label{eq:glue_final1}
 (Z, V^\rn{\kp_n}(Z)) = (\cZ, \cV)( Y, U^{(\kp_n)}(Y)), 
 \quad |Y| < \e_1.
 \eeq

\vspace{0.1in}
\paragraph{\bf{First gluing}}
In Proposition \ref{prop:inter_U_root}, we show that $U^{(\kp_n)}(Y)$ can be extended smoothly for $Y < 0$ up to $Y = Y_I^{\pr}$. We define the map following \eqref{eq:glue_ODE1}
\[
 (Z_\mw{a}, V_\mw{a}) (Y) = (\cZ, \cV)(Y, U^\rn{\kp_n}(Y) ), \quad  Y \in J,
\quad J = (  Y_I^{\pr}, \e_1), 
\] 

From Proposition \ref{prop:inter_U_root}, \eqref{eq:root_pos2}, we have $ U_{\D_Y}(Y) <  U^\rn{\kp_n}(Y) < U_g(Y)$ for $Y \in (  Y_I^{\pr}, 0)$. Thus, along the solution curve $( Y, U^\rn{\kp_n}(Y) ) $ with $Y \in (  Y_I^{\pr}, 0)$, using \eqref{eq:root_UY}, \eqref{eq:DZ_UY} and following the proof of \eqref{eq:DelZ_sign}, we obtain $\D_Y  < 0  , \  \D_Z <  0 $. Following \eqref{eq:glue_dZ}, we obtain 
\[
\frac{dZ_\mw{a}}{dY}(Y) = m^{-1}(Y, U^\rn{\kp_n}(Y)) \f{\D_Z}{\D_Y} \neq 0,
\]
for any $Y \in (Y_I^{\pr}, 0)$. Using $Z_\mw{a}^{\pr}(0) < 0$ and continuity, we obtain 
\beq\label{eq:glue_m_sign}
m(Y, U^\rn{\kp_n}(Y)) < 0, 
\quad
 Z_\mw{a}^{\pr}(Y) < 0,  \quad Y \in ( Y_I^{\pr}, 0) ,
\eeq
and  $Z_\mw{a}(Y)$ is invertible  for $Y \in (Y_I^{\pr}, 0)$. Constructing a solution $V_\mw{a}(Z_\mw{a}^{-1}(Z))$ to the ODE \eqref{eq:ODE}, 
and using the gluing argument in Section \ref{sec:glue}, we extend the relation 
\eqref{eq:glue_final1} to $(Y_I^{\pr}, \e_1)$: 
\beq\label{eq:glue_final2}
 (Z, V^\rn{\kp_n}(Z)) = (\cZ, \cV)( Y, U^{(\kp_n)}(Y)),  \quad Y_I^{\pr} < Y < \e_1. 
\eeq
Moreover, $V^\rn{\kp_n}(Z)$ solves the ODE \eqref{eq:ODE} smoothly for $Z \in [0, \cZ(Y_I^{\pr}, U^\rn{\kp_n}(Y_I^{\pr})  ) ) $.

\vspace{0.1in}
\paragraph{\bf{Second gluing}}
Next, we construct another solution using the desingularized ODE \eqref{eq:ODE_ds}
\beq\label{eq:glue_final3}
 (Z_\mw{b}, V_\mw{b}) (\xi) = (\cZ, \cV)(Y_\ds(\xi), U_\ds(\xi) ), \quad \xi \in [0, \xi_3]
\eeq
where $\xi_i$ is defined in Proposition \ref{prop:ds_ODE}. 
Below, we consider arbitrary $\xi \in [0, \xi_3]$. We show that $Z_\mw{b}^{\pr}(\xi) > 0$.
Using the ODE \eqref{eq:ODE_ds} 
and the second identity in \eqref{eq:Mv}, we yield 
\[
\bal
 \frac{dZ_\mw{b}}{d\xi}(\xi)
 & = \pa_Y \cZ  \cdot Y_\ds^{\pr} 
  + \pa_U \cZ  \cdot U_\ds^{\pr} 
  = \pa_Y \cZ  \cdot \D_Y  
  + \pa_U \cZ  \cdot \D_U \\
  & = m^{-1}(Y_\ds(\xi), U_\ds(\xi) ) \D_Z(Z_\mw{b}(\xi), V_\mw{b}(\xi)) ,
  \eal
\]
From result (e) in Proposition \ref{prop:ds_ODE}, we have $U_\ds(\xi) < U_g( Y_\ds(\xi))$. 
Using \eqref{eq:DZ_UY}, we obtain $\D_Z(Z_\mw{b}(\xi), V_\mw{b}(\xi) ) <  0 $. Since  $(Y_\ds(\xi), U_\ds(\xi) ) \neq Q_s$, due to \eqref{eq:Mv}, we obtain $m(Y_\ds(\xi), U_\ds(\xi)) \neq 0$. Using \eqref{eq:glue_m_sign} and continuity, we obtain 
$m < 0$ and $Z_\mw{b}^{\pr}(\xi) > 0$.
Thus, $Z_\mw{b}(\xi)$ is invertible.

Following the gluing argument in Section \ref{sec:glue}, we construct a smooth solution 
\beq\label{eq:glue_final4}
V_\mw{ODE}^\rn{\kp_n}(Z )= V_\mw{b}(Z_\mw{b}^{-1}(Z)), \quad Z  \in Z_\mw{b}([0, \xi_3]), 
\eeq
to the ODE \eqref{eq:ODE}. Using the definition of $V_\mw{ODE}^\rn{\kp_n}$,  \eqref{eq:glue_final3}, \eqref{eq:ODE_ds_init}, and \eqref{eq:glue_final2} in order, we get 
\[
\bal
(Z_\mw{b}(0), V_\mw{ODE}^\rn{\kp_n}( Z_\mw{b}( 0 ))
 & = ( Z_\mw{b}(0), V_\mw{b}( 0)
= (\cZ, \cV)( Y_\ds( 0 ), U_\ds( 0 ) ) 
 =(\cZ, \cV)( Y_I^{\pr} / 2, U^\rn{\kp_n}( Y_I^{\pr}  / 2) ) \\
 & = ( Z, V^\rn{\kp_n}(Z) ) |_{Z = \cZ (  Y_I^{\pr} / 2, U^\rn{\kp_n}( Y_I^{\pr}  / 2) )  }. 
 \eal
\] 
 Using the uniqueness of ODE, we obtain $V_\mw{ODE}^\rn{\kp_n} =  V^\rn{\kp_n}$. Since $Z_\mw{b}$ is increasing, gluing $V_\mw{ODE}^\rn{\kp_n}$ and $V^\rn{\kp_n}$, we construct the smooth ODE solution to \eqref{eq:ODE} $V^\rn{\kp_n} \in C^{\infty}( [0,  Z_\mw{b}(\xi_3) ])$
 with $ Z_\mw{b}(\xi_3) > Z_0$.

\vspace{0.1in}
\paragraph{\bf{Estimates \eqref{eq:prop_ODE1}}}
Recall that $(Z_0, V_0)$ is the sonic point \eqref{eq:root_P}. We define 
\beq\label{eq:Z2_V2}
 Z_2 = Z_\mw{b}(\xi_3), \quad V_2 =  V^\rn{\kp_n}(Z_2) .
\eeq
From Propositions \ref{prop:inter_U_root}, \ref{prop:ds_ODE}, and Lemma \ref{lem:DelY}, the solution curves $(Y, U^\rn{\kp_n}(Y))$ with $Y \in [Y_I^{\pr}, 0]$ and $(Y_{\ds}(\xi), U_{\ds}(\xi)) $ with $\xi \in [0, \xi_3]$ are in $ \{ (Y, U) : 1 > Y > -\g, U > 0 \} $. From \eqref{eq:glue_final2}, \eqref{eq:glue_final3}, 
under the $(\cZ, \cV)$ map \eqref{eq:UY_to_VZ}, these curves are mapped to $(Z, V^\rn{\kp_n}(Z) )$ with $Z \in [Z_0, Z_\mw{b}(\xi_3)]$. Following the proof of \eqref{eq:Qs_QO_prop1} 
in Section \ref{sec:other_prop},
we prove the estimates \eqref{eq:prop_ODE1} with $Z \in [Z_0, Z_2]$. Estimates \eqref{eq:prop_ODE1} with $Z \in ( 0, Z_0]$ have been established in Proposition \ref{prop:Qs_QO} .

\vspace{0.1in}
\paragraph{\bf{Location}} 

Finally, 
we prove $P = (Z_2, V_2) \in \Om_\mw{far} $ for $Z_2, V_2$ defined in \eqref{eq:Z2_V2} and $\Om_\mw{far} $ defined in \eqref{eq:dom_far}. Using \eqref{eq:glue_final4} with $Z_2 =  Z_\mw{b}(\xi_3)$ and then \eqref{eq:glue_final3}, we yield 
\[
(Z_2, V_2) = P = ( Z_\mw{b}(\xi_3), V_\mw{b}(\xi_3))
= (\cZ, \cV)( Y_\ds(\xi_3), U_{\ds}(\xi_3)). 
\]
Using result (d) in Proposition \ref{prop:ds_ODE}, we get $1 > Y_\ds(\xi_3) > 0, U_\ds(\xi_3) < U_g( Y_\ds(\xi_3))$. Using Lemma \ref{lem:bijec} for the map $(\cZ, \cV)$, 
\eqref{eq:DZ_UY} for $\D_Z$,  \eqref{eq:ODE} and \eqref{eq:sys_UY} for $\D_V$, we obtain 
\[
\D_V(P)> 0,  \quad \D_Z(P) < 0 ,\quad  Z_2 >  0, \quad  1 > V_2 > 0.
\]

Using the identities \eqref{eq:root_mono_ZV0} and $\D_V(P)> 0$, we obtain $Z_2 > Z_V(V_2)$. 
From \eqref{eq:root_mono_ZV0}, \eqref{eq:root_mono_ZV},
condition $\D_Z(P) < 0$ implies $Z_2 > Z_+(V_2)$ or $Z_2 < Z_-(V_2)$. 
Since $ 0 < V_2 <1 < \ell$, we get 
\[
Z_2 > Z_V(V_2) = \f{(1 + \g)V_2}{ 1 + \g V_2^2}  > V_2 > \f{\ell^{1/2} V_2-1}{ \ell^{1/2} - V_2 }
= Z_-(V_2).
\]
It follows $Z_2 >Z_+(V_2)$ and $Z_2 > V_2$. We conclude the proof of Proposition \ref{prop:QO_Q2}.


\subsection{Proof of Lemma \ref{lem:far_extend}}\label{sec:global}

By continuity, there exists $Z_3 > Z_2$ such that the $C^1$ solution  $V(Z)$ to the ODE \eqref{eq:ODE}
exists and $(Z, V(Z)) \in \Om_{\mw{far}}$ for $Z \in [Z_2, Z_3)$. 
We assume that $Z_3$ is the maximal value with such a property. Our goal is to show that $Z_3 = \infty$. Since $\D_Z, \D_V$ do not vanish for $(Z, V) \in \Om_{\mw{far}}$, and $\D_Z, \D_V$ are polynomials in $Z, V$, we get $V(Z) \in C^{\infty}[Z_2, Z_3)$.

For any $Z \in [Z_2, Z_3)$, using sign properties in $\Om_{\mw{far}}$ \eqref{eq:dom_far} and \eqref{eq:root_mono_ZV0}, we get
\bseq\label{eq:far_est1}
\beq
\D_Z(Z, V(Z)) < 0, \quad \D_V(Z, V(Z)) > 0, \quad 
 \f{d V}{d Z} = \f{\D_V}{\D_Z} \B|_{(Z, V(Z))} < 0.
\eeq
Thus, $V(Z)$ is decreasing in $Z$ for $Z \in [Z_2, Z_3)$.

Since $g(v) = v, Z_{\pm}(v), Z_V(v)$ is increasing in $v$ \eqref{eq:root_mono_ZV} and $(Z_2, V_2) \in \Om_{\mw{far}}$, we obtain 
\beq\label{eq:far_est1b}
\bal
Z - g(V(Z)) > Z_2 - g( V(Z_2)) = c_g > 0,  \quad g(v) = v, Z_{\pm}(v), Z_V(v) .
\eal
\eeq

Since $V(Z) \in (-1, 1)$ for all $Z \in [Z_2, Z_3)$, it follows 
\beq\label{eq:far_est1c}
\D_Z(Z, V(Z)) \leq Z_2 (1 - \ell) c < 0 .
\eeq
for some constant $c>0$. For  $Z \in [Z_2, Z_3)$ with $|Z| \leq b$, since $|V(Z)| < 1$, we estimate $\D_V$ \eqref{eq:ODE}:
\beq\label{eq:far_est1d}
  |\D_V| \leq C_b (1 - V^2) \leq C_b \min(1 - V, 1 + V) ,
\eeq
\eseq
for some constant $C_b > 0$. For any $Z \in [Z_2, Z_3)$ with $|Z| \leq b$, 
combining  \eqref{eq:far_est1}, we have 
\[
 \f{d (V+1)}{d Z} \geq - C_b( V+1),
 \quad  \f{d (1 - V)}{d Z} \geq - C_b( 1 - V),
\]
which implies 
\beq\label{eq:far_est2}
V(Z) + 1 \geq e^{- C_b( b-Z_2)} (V_2 + 1) > 0, \quad
1- V(Z)  \geq e^{- C_b( b-Z_2)} (1 - V_2 ) > 0.
\eeq

If $Z_3$ is bounded, we choose $b > Z_3$. For $Z \in (Z_2, Z_3)$, since $ V(Z)$ is bounded, $\D_V$ is bounded \eqref{eq:far_est1d}, and $\D_Z$ is bounded away from $0$ \eqref{eq:far_est1c}, we can solve the ODE \eqref{eq:ODE} with $C^{\infty}$ solution in $(Z_2, Z_3 + \d_7)$ with $\d_7>0$ sufficiently small. 
Using the uniform estimates \eqref{eq:far_est1b},\eqref{eq:far_est2} and choosing $\d_7$ small enough, we obtain that $(Z, V(Z)) \in \Om_{\mw{far}}$ \eqref{eq:dom_far} for $ Z \in (Z_2, Z_3 + \d_7)$. 
This contradicts the maximality of $Z_3$. Thus, we have $Z_3 = \infty$.

\appendix
\section{Some derivations and estimates }
In this section, we derive the ODE \eqref{eq:ODE} in Appendix \ref{app:ODE_deri} 
and prove Lemma \ref{lem:UF_asym} in Appendix \ref{app:UF_asym}.

\subsection{Equations of the profiles}\label{app:ODE_deri}

The ODE \eqref{eq:ODE} has been first derived in \cite{shao2024self}. For completeness, we derive it below. With the ODE solution, we can further construct the profiles $W, \Phi$ in \eqref{eq:SS_ansatz}.

We consider radially symmetric solutions  $\phi, \rho$ to \eqref{eq:wave_phase} and \eqref{eq:vel}. Recall the self-similar ansatz \eqref{eq:SS_ansatz} and the notations \eqref{eq:SS_var} 
\beq\label{eq:ODE_deri_nota}
\bga
  r = |x|,   \quad  z = \f{x}{ T- t}, \quad Z = |z|, \\
 \rho^{- \f{p-1}{2} } \pa_t \phi =  \f{1}{  (1 - V^2)^{1/2}} \teq C(Z), 
 \quad  \rho^{- \f{p-1}{2} } \pa_r \phi =  \f{V}{  (1 - V^2)^{1/2}} \teq D(Z).
 \ega
\eeq
We will \textit{only} use the notations $C, D$ in this section.

Since $\pa_r$ and $\pa_t$ commute for smooth functions, we get 
\beq\label{eq:ODE_deri1}
 \pa_t ( D \rho^{\f{p-1}{2} } ) = \pa_r ( C \rho^{\f{p-1}{2} }  )
\eeq

Using the self-similar ansatz \eqref{eq:SS_ansatz} with $c=1$, we get
 \beq\label{eq:ODE_deri2}
\pa_t \rho = (T-t)^{-1} ( a \rho + Z \pa_Z \rho) 
=a  (T-t)^{-1}  \rho + Z \pa_r \rho.
 \eeq
 We further compute $\pa_r  \rho / \rho $. Dividing $\rho^{ (p-1)/2}$ on both sides of \eqref{eq:ODE_deri1} and using  \eqref{eq:ODE_deri2}, we get 
 \[
0 =  \f{p-1}{2} \cdot \f{ D \pa_t \rho  - C \pa_r \rho }{\rho} 
  + \pa_t D - \pa_r C = \f{p-1}{2} (D Z -C) \f{\pa_r \rho}{\rho}
+ \f{p-1}{2} a \f{D}{T-t} + \pa_t D - \pa_r C, 
 \]
which implies 
\beq\label{eq:ODE_deri3}
 \f{\pa_r \rho}{\rho} = \f{1}{C - D Z} \B( a \f{D}{T-t} + \f{2}{p-1}(\pa_t D - \pa_r C)  \B) .
\eeq
Next, we use \eqref{eq:wave_phase} to  derive the equation for $V$. 
Since $\phi$ is radially symmetric, using $\D \phi = \pa_r \pa_r \phi + \f{d-1}{r} \pa_r \phi$, we can rewrite \eqref{eq:wave_phase} as 
\[
\pa_t (\pa_t \phi \rho^2) =
 \pa_r ( \pa_r \phi \rho^2 ) + \f{d-1}{r} \pa_r \phi \rho^2.
\]

Using \eqref{eq:ODE_deri_nota} and $\rho^{ (p-1)/2 + 2} = \rho^{(p+3)/2} $, we further rewrite it as 
\[
 \pa_t ( C \rho^{ (p+3)/2} ) - \pa_r( D \rho^{ (p+3)/2}   )
 = (d-1) r^{-1} D \rho^{ (p+3)/2}.
\]

Dividing $\rho^{ (p+3)/2}$ on both sides and using \eqref{eq:ODE_deri2}, we obtain 
\bseq
\begin{align}
 \f{p+3}{2} \cdot \f{C \pa_t \rho - D \pa_r \rho  }{\rho}
+ \pa_t C - \pa_r D &=  (d-1) \f{1}{r} D   \label{eq:ODE_deri4a},  \\
\f{p+3}{2}  \cdot \B(  \f{ (C Z - D) \pa_r \rho  }{\rho}
+  a \f{C}{T-t} \B) + \pa_t C - \pa_r D &=  (d-1) \f{1}{r} D  \label{eq:ODE_deri4b} .
\end{align}
\eseq

Since $Z = \f{r}{T-t}$, for $F(Z) = C(Z), D(Z)$, we obtain 
\beq\label{eq:ODE_deri5}
\pa_t F(Z) = \f{1}{T-t} Z \pa_Z F, \quad \pa_r F(Z) =\f{1}{T-t} \pa_Z F,
\quad \f{1}{r} D = \f{1}{T-t} \cdot \f{1}{Z} D.
\eeq

Substituting \eqref{eq:ODE_deri3}, \eqref{eq:ODE_deri5} in \eqref{eq:ODE_deri4b} and cancelling the factor $(T-t)^{-1}$, we obtain
\beq\label{eq:ODE_deri6}
\f{p+3}{2} \B(  a C + \f{C Z - D}{C - D Z} ( a D + \f{2}{p-1} ( Z \pa_Z D - \pa_Z C) )  \B)
+ Z \pa_Z C -\pa_Z D = (d-1) \f{D}{Z}. 
\eeq

Using the formulas of $C, D$ \eqref{eq:ODE_deri_nota}, we compute 
\beq\label{eq:ODE_deri7}
\f{1}{C} \pa_Z C = \f{1}{C} V^{\pr} \cdot V (1 - V^2)^{-3/2}
= V^{\pr}  \cdot \f{V}{ 1 - V^2} ,
\quad  \f{1}{C} \pa_Z D =\f{1}{C} V^{\pr} \cdot  (1 - V^2)^{-3/2} 
= V^{\pr} \cdot  \f{1}{1 - V^2}.
\eeq

From \eqref{eq:ODE_deri_nota}, $p = \f{4}{\ell-1} + 1$ \eqref{eq:para}, we have 
\[
D = C V, \quad \f{p+3}{p-1} = \ell, \quad a \f{p+3}{2} = ( \f{4 }{\ell-1} + 4) \f{2(d-1)}{ (p-1) \ell (\g+1)}
=  \f{d-1}{\g+1}. 
\]

 Thus dividing $C$ on both sides of \eqref{eq:ODE_deri6} and using \eqref{eq:ODE_deri7}, we obtain
\beq\label{eq:ODE_deri8}
 \ell \f{ Z - V}{ 1 - Z V}  \cdot \f{Z -V}{1 - V^2}   \cdot V^{\pr}
+ \f{Z V - 1}{ 1 - V^2}  \cdot V^{\pr}
 + \f{d-1}{\g+1} \B( 1 + \f{Z -V}{1 - V Z} \cdot V  \B) = (d-1) \f{V}{Z (1 - V^2)} 
\eeq

Rewriting the above equations, we derive the ODE \eqref{eq:ODE}:
\[
\f{d V}{d Z} = \f{(d-1) (1- V^2) \B(  \f{1}{\g + 1} (1 -V^2) Z -  V (1 - V Z)  \B) }{ 
Z( (1 - Z V)^2 - \ell (V - Z)^2 )  } .
\]

Once we construct the ODE solution $V \in C^{\infty}([0, \infty))$ with $V \in (-1, 1)$, we further construct the profile $W$ \eqref{eq:SS_ansatz} for $\rho$. Plugging \eqref{eq:SS_ansatz2},
\eqref{eq:ODE_deri5}, \eqref{eq:ODE_deri7}, $a = \f{2 (d-1)}{ (p-1) \ell (\g+1)}$, $D = C V$ in \eqref{eq:ODE_deri3}, and cancelling the power of $T-t$, we obtain 
\bseq\label{eq:W_ODE}
\beq
\f{ \pa_Z W}{W} = \f{1}{1 - V Z} \B(  \f{2 (d-1)}{ (p-1) \ell (\g+1)} V + \f{2}{p-1} \cdot  \f{Z-V}{1 - V^2} V^{\pr}(Z) \B) \teq J_W(Z).
\eeq

Using \eqref{eq:ODE_deri8} $\times  \f{2}{(p-1)\ell} \cdot \f{1}{Z-V}$, we can rewrite the right hand side as 
\beq
J_W(Z) = \f{2}{(p-1)\ell} \cdot \f{1}{Z-V}
\B(  (d-1) \f{V}{Z (1 - V^2)}  - \f{d-1}{\g+1} - \f{Z V - 1}{ 1 - V^2}  \cdot V^{\pr}  \B).
\eeq
\eseq


Using the equation for $\pa_r \phi$ in \eqref{eq:ODE_deri_nota} and the ansatz \eqref{eq:SS_ansatz}, we obtain 
\beq\label{eq:Phi_ODE}
 \pa_Z \Phi(Z ) =  \f{ V \cdot  W^{ (p-1)/2 } }{ (1 -V^2)^{3/2}}. 
\eeq
After we construct $W$, we can construct $\Phi$ from the above ODE.

\subsection{Proof of Lemma \ref{lem:UF_asym}}\label{app:UF_asym}

In this section,  we prove Lemma \ref{lem:UF_asym}. We simplify $
Y_F^\rn{\kp}, U_F^{(\kp)}$ constructed via \eqref{eq:QO_glue} as $Y, U$,
and $ V_F^\rn{\kp}$ as $V$. 

Using Proposition \ref{prop:V_near0}, we can expand the smooth solution $V(Z)$ near $Z=0$ as follows 
\beq\label{eq:V_asym}
V(Z) = V_1 Z +  V_3 Z^3 +  O(Z^5),  \quad 
V^{\pr}(Z) = V_1 + 3 V_3 Z^2 + O(Z^4), \quad \forall Z \ll 1 .
\eeq
for $Z \in [0, \d_V]$. In the above and the following derivations, the implicit constants and $\d_V > 0$ are uniformly in $\g$ due to Proposition \ref{prop:V_near0}.

Firstly, we compute  $V_1 = \f{d}{d Z} V$. Using the ODE \eqref{eq:ODE} and evaluating at $(Z, V) = 0$, we obtain 
\beq\label{eq:V_asym_V1}
V_1 =  \f{(d-1) }{\g + 1} - (d-1) V_1 \quad  \Rightarrow \quad V_1 = \f{d-1}{ d (\g + 1)} .
\eeq

Next, we compute $ V_3$. We expand $\D_Z(Z,V(Z)), \D_V(Z,V(Z))$ near $Z =0$ up to the term $Z^3$
\[
\bal
\D_Z  &= Z ( (1 - Z V)^2 - \ell ( V-Z)^2)
= Z ( (1 - V_1 Z^2 )^2 -  \ell ( V_1 Z-Z)^2 ) + O(Z^4) \\
& = Z( 1 - (2 V_1 + \ell (V_1 - 1)^2) Z^2  ) + O(Z^4) , \\
\D_V & = (d-1) (1- V^2) \B(  \f{1}{\g + 1} (1 -V^2) Z -  V (1 - V Z)  \B) \\
& = (d-1)(1 - V_1^2 Z^2)  \B(  \f{1}{\g + 1} (1 -V_1^2 Z^2) Z -  (V_1 Z + V_3 Z^3) (1 - V_1 Z^2)  \B
) + O(Z^4) \\
& = Z (d-1) (1 - V_1^2 Z^2)\B( \f{1}{\g+1} (1 -V_1^2 Z^2) - ( V_1 + (V_3 - V_1^2) Z^2 ) \B) 
+ O(Z^4) \\
& = Z(d-1) \B(  \f{1}{\g+1} - V_1 +  (  -  \f{2 }{\g+1}  V_1^2 + V_1^3 - V_3 + V_1^2 ) Z^2  \B) + O(Z^4) . 
\eal
\]

Using the ODE $\D_Z \f{d V}{d Z} = \D_V$ \eqref{eq:ODE} and tracking the term up to $O(Z^4)$, we get 
\[
Z( 1 - (2 V_1 + \ell (V_1 - 1)^2) Z^2  ) (V_1 + 3 V_3 Z^2)
 =  Z(d-1) \B(  \f{1}{\g+1} - V_1 +  (  -  \f{2 }{\g+1}  V_1^2 + V_1^3 - V_3 + V_1^2 ) Z^2  \B) + O(Z^4) .
\]
Matching the $Z$ term yields $V_1$. Matching $Z^3$ term, we get 
\[
(3 + (d-1) )V_3
- V_1 (2 V_1 + \ell (V_1 - 1)^2)
= (d-1)\B(  -  \f{2 }{\g+1}  V_1^2 + V_1^3 + V_1^2 \B) .
\]
Rearranging the identity, we obtain $V_3$ in \eqref{eq:V_formula}. Identity \eqref{eq:V_asym} gives the formula of $V_1$ in \eqref{eq:V_formula}.

Plugging the asymptotics \eqref{eq:V_asym} into \eqref{eq:sys_UY}, near $Z=0$, we get 
\beq\label{eq:UY_asym1}
\bal
 \cY(Z, V(Z))- Y_O 
 & = 1 - \f{1}{d} - (\g+1) \f{V}{Z} \cdot \f{1- VZ}{1-V^2}  \\
 & = 1 - d^{-1} - (\g+1) (V_1 + V_3 Z^2) (1 - V Z )(1 + V^2) + O(Z^4) \\
 & =1 - d^{-1} - (\g+1) (V_1 + V_3 Z^2) (1 - V_1 Z^2 )(1 + V_1^2 Z^2) + O(Z^4) \\
 & = - (\g +  1)(V_3 - V_1^2 + V_1^3) Z^2 + O(Z^4) , \\
 \cU(Z, V(Z)) & = \f{(\g+1)^2}{Z^2} + O(1). 
 \eal
 \eeq
The constant term in $\cY- Y_O$ is $0$ since $1 - d^{-1} - (\g + 1) V_1 =0$ \eqref{eq:V_asym_V1}. 
Using Lemma \ref{lem:limit}, we obtain $ - (\g +  1)(V_3 - V_1^2 + V_1^3) < c < 0$ uniformly for $\g$ close to $\ell^{-1/2}$ and some absolute constant $c$. 
Choosing $\d_3 \in (0,\d_V)$ small enough, for any $Z \in [0, \d_3]$, estimate \eqref{eq:UY_asym1} implies \eqref{eq:UF_asyma}.

 For $Z \in [0, \d_3]$, using the estimates of $\cU, \cY$ \eqref{eq:UY_asym1} and the definition \eqref{eq:QO_glue}, we prove
 \[
 \B| U(Z) - \f{ (-(\g+1)^3(V_3 - V_1^2 + V_1^3)) }{Y(Z) - Y_O} \B|
 = \B|U(Z) - \f{ (\g+1)^2 }{Z^2} + O(1) \B| = O(1) ,
 \]
 and \eqref{eq:UF_asymb}. We conclude the proof of Lemma \ref{lem:UF_asym}.

\section{Details of the computer-assisted part}\label{app:comp}

In this section, we discuss the companion files for computer-assisted proof. We performed the rigorous computation using SageMath, and the code is attached to the paper. There are two Sage files: 
\textbf{Wave\_induction.ipynb, Wave\_barrier.ipynb}, which can be implemented on a personal laptop in less than 1 minute. Moreover, all the outputs are recorded in a Jupyter Notebook without implementing the codes. 

For the variables/functions with different notations in this paper  (left side) and in the code (right side), we provide the relation between two notations below: 
\[
\bal
& \ \e,  \ \ell, \  \g,   \ f(Y) ,
\ \hat {\d}, 
 \ \laml, \ \lams , \ \mfr{C} , \  \hat U
 \\
\rightsquigarrow  & \ \mw{ep}, \ l, \  \mathrm{ga}, \ \mw{fy},  \ \mw{del\_hat},
\ \mw{lam\_l }, \ \mw{ lam\_s } , 
\  \mw{ C\_mfr} , \ \mw{U\_hat}  \\
& \ U_{\D_Y}, \ U_{\D_U},  \ \D_U, \ \D_Y , 
\ \D_{U, n}, \ \D_{Y, n}, \ c_{U, n}, \ c_{Y, n}  \\
\rightsquigarrow 
& \ \mw{U\_DelY},\ \mw{U\_DelU},  \ \mw{del\_U}, \ \mw{del\_Y} ,  
\ \mw{DelUn}, \ \mw{DelYn} , \ \mw{dU\_DelUn}, \ \mw{dU\_DelYn} ,
\ 
\eal
\]
The functions and parameters are defined in: $\e$ \eqref{eq:para2}, $\ell$ \eqref{eq:para1}, $\D_U, \D_Y, f(Y)$ \eqref{eq:UY_ODE}, 
 $\hat \d$ \eqref{eq:para3}, $\laml, \lams$ \eqref{eq:grad_lam}, 
 $U_{\D_Y}, U_{\D_U}$ \eqref{eq:root_UY},  $\D_{U, n}, \D_{Y, n}$ \eqref{eq:Tal_del}, 
 $c_{U, n}, c_{Y, n}$ \eqref{eq:recur_const}.

We need to use computer assistance to verify Lemmas \ref{lem:asym_claim}, \ref{lem:induc_claim},  \ref{lem:limit}, and Proposition \ref{prop:bar_f}. 

\vspace{0.1in}
\paragraph{\bf{Wave\_induction.ipynb}}

The goal is to verify Lemmas \ref{lem:asym_claim},\ref{lem:induc_claim}. It consists of three steps.

(1) Derive the power series coefficients $U_i$ at the limit case $\g = \ell^{-1/2}$ using the recursive formula \eqref{eq:recur_top} in Lemma \ref{lem:pow_recur}. 

(2) Follow Section \ref{sec:induc_pf} to derive the constants 
$C_{J_i}, C_{\cE}$ in Lemma \ref{lem:induc_claim} and verify Lemma \ref{lem:induc_claim}.

(3) Verify Lemma \ref{lem:asym_claim} using the value $U_i$ at the limit case $\g = \ell^{-1/2}$ derived in Step (1).

\vspace{0.1in}
\paragraph{\bf{Wave\_barrier.ipynb}}

The goal is to verify \eqref{eq:bar_far_prop} and \eqref{eq:root_ineq} in Proposition \ref{prop:bar_f} and the scalar inequality in Lemma \ref{lem:limit}.  It consists of the following steps.

(1) Derive the power series coefficients $U_i$ at the limit case $\g = \ell^{-1/2}$, which is the same as those in \textrm{Wave\_induction.ipynb}. We only need $U_0, U_1, U_2$.

(2) Introduce a function \textrm{Poly\_sign} to derive upper and lower bounds of a polynomial $P(t)$ over $[a, b]$ with $0 \leq a \leq b $. We can decompose $P(t) = P_+(t) - P_-(t)$ with $P_+, P_-$ being polynomials with non-negative coefficients. For $t \geq 0$, $P_{\pm}$ is increasing. Thus, by dividing $[a, b]$ into $a=  t_0< t_1 <.. < t_m = b$ and using 
\[
P(t)  \geq P_{+, l} - P_{-, u}, \quad P(t) \leq P_{+, u} - P_{-, l},
\]
for $t \in [t_i, t_{i+1}]$, where 
\[
\quad P_{\pm, l} = 
\min( P_{\pm}( t_i ), P_{\pm}(t_{i+1}) ),
\quad P_{\pm, u} = 
\max( P_{\pm}( t_i ), P_{\pm}(t_{i+1}) ),
\]
we estimate the upper and lower bound of $P$ in each interval $[t_i, t_{i+1}]$. Maximize or minimize the estimates over all sub-intervals yields the bounds of $P$ in $[a, b]$.

(3) Construct the barrier functions $B_l^f, B_u^f$ following Section \ref{sec:bar_far}. Verify 
\eqref{eq:bar_far_prop} and \eqref{eq:root_ineq} in the limit case $\g = \ell^{-1/2}$ following the ideas in the proof of Proposition \ref{prop:bar_f} and verify Lemma \ref{lem:limit}.

\section*{Acknowledgments}
 The work of T.B.~ has been supported in part by the NSF grants DMS-2243205 and DMS-2244879, as well as the Simons Foundation Mathematical and Physical Sciences collaborative grant 'Wave Turbulence.' The work of J.C.~has been supported in part by the NSF grant DMS-2408098. We would like to acknowledge the insightful and fruitful discussions we had with Hans Lindblad and 
Jalal Shatah.

\bibliographystyle{plain}
\bibliography{selfsimilar}

\end{document}